\documentclass[reqno]{amsart}
\usepackage{amssymb}
\usepackage{amsmath,amsfonts,amssymb}
\usepackage{amsfonts, amsmath, amssymb, amscd, amsthm}

\newtheorem{theorem}{Theorem}[section]
\newtheorem{lemma}{Lemma}[section]
\newtheorem{proposition}{Proposition}[section]
\newtheorem{corollary}{Corollary}[section]

\theoremstyle{remark}
\newtheorem{remark}{Remark}[section]

\begin{document}

\title[centroaffine hypersurfaces with parallel cubic form]
{Classification of the locally strongly convex \\centroaffine
hypersurfaces with \\parallel cubic form}

\author[X. X. Cheng, Z. J. Hu and M. Moruz]{Xiuxiu Cheng, Zejun Hu and Marilena Moruz}


\thanks{This project was supported by NSFC (Grant No. 11371330)}

\date{}


\begin{abstract}
In this paper, we study locally strongly convex centroaffine
hypersurfaces with parallel cubic form with respect to the
Levi-Civita connection of the centroaffine metric. As the main
result, we obtain a complete classification of such centroaffine
hypersurfaces. The result of this paper is a centroaffine version of
the complete classification of locally strongly convex
equiaffine hypersurfaces with parallel cubic form due to Hu, Li and
Vrancken \cite{HLV2}.
\end{abstract}
\maketitle

%===================================================================================================
\numberwithin{equation}{section}

\section{Introduction}\label{sect:1}

In centroaffine differential geometry, we study properties of
hypersurfaces in the $(n+1)$-dimensional affine space
$\mathbb{R}^{n+1}$ equipped with its standard affine flat connection
$D$, that are invariant under the centroaffine transformation group
$G$ in $\mathbb{R}^{n+1}$. Here, by definition, $G$ is the subgroup
of affine transformation group in $\mathbb{R}^{n+1}$ which keeps the
origin $O\in \mathbb{R}^{n+1}$ invariant. Let $M^n$ be an
$n$-dimensional smooth manifold. An immersion $x: M^n\rightarrow
\mathbb{R}^{n+1}$ is said to be centroaffine hypersurface if the position
vector $x$ (from $O$) for each point $x\in M^n$ is transversal
to the tangent plane of $M$ at $x$. In this case, the position vector
$x$ defines the so-called {\it centroaffine normalization} modulo
orientation. For any vector fields $X$ and $Y$ tangent to $M$, we
have the centroaffine formula of Gauss:
\begin{equation}\label{eqn:1.1}
D_{X}x_{*}(Y)=x_{*}(\nabla_{X}Y) + h(X,Y)(-\varepsilon x) ,
\end{equation}
where $\varepsilon=1$ or $-1$. In this paper, we always assume that
$x: M^n\rightarrow \mathbb{R}^{n+1}$ is a non-degenerate
centroaffine hypersurface, i.e., the bilinear $2$-form $h$, defined
by \eqref{eqn:1.1}, remains non-degenerate. Moreover, associated
with \eqref{eqn:1.1} we call $-\varepsilon x$, $\nabla$ and $h$ the
centroaffine normal, the induced connection and the {\it
centroaffine metric} induced by $-\varepsilon x$, respectively.

Let $N(h)$ denote the dimension of the maximal negative definite
subspaces of the bilinear form $h$ with respect to $\varepsilon=-1$.
For a locally strongly convex centroaffine hypersurface, i.e.,
$N(h)=0$ or $N(h)=n$, we can choose $\varepsilon$ such that the
centroaffine metric $h$ is positive definite. In that situation, if
$\varepsilon=1$ we say that the hypersurface is elliptic and, if
$\varepsilon=-1$ we call the hypersurface hyperbolic (cf. section 2
of \cite{LLS}). We refer to \cite{DVV,LW1,W} for some interesting
studies on centroaffine hypersurfaces.

Given a non-degenerate centroaffine hypersurface $x: M^n\rightarrow
\mathbb{R}^{n+1}$, we denote by $\hat{\nabla}$ the Levi-Civita
connection of $h$. Then the difference tensor $K$, defined by
$K(X,Y):=K_XY:=\nabla_XY-\hat{\nabla}_XY$, and the cubic form
$C:=\nabla h$ are related by the equation
\begin{equation}\label{eqn:1.2}
C(X,Y,Z)=-2h(K_XY,Z)=-2h(K_XZ,Y).
\end{equation}

It is well-known (cf. \cite{LLS,LSZH,SSV}) that a centroaffine
hypersurface immersion is uniquely determined, up to a centroaffine
transformation, by its centroaffine metric and its cubic form (this
means that the cubic form plays the role of the affine second
fundamental form). Hence, in centroaffine differential geometry the
problem of classifying affine hypersurfaces with parallel cubic form
(i.e., $\hat\nabla C=0$) is quite natural and important. In
\cite{LW}, Li and Wang considered this problem the first time by
studying the so-called {\it canonical} centroaffine hypersurface.
Here, a centroaffine hypersurface is called {\it canonical} if its
centroaffine metric $h$ is flat and its cubic form $C$ satisfies
$\hat{\nabla}C=0$.

We should recall that in equiaffine differential geometry, the
problem of classifying locally strongly convex affine hypersurfaces
with parallel cubic form has been studied intensively, from the
earlier beginning paper by Bokan, Nomizu and Simon \cite{BNS}, and
then \cite{DV,DVY,HLLV,HLSV} by some others, to the very recent
complete classification of Hu, Li and Vrancken \cite{HLV2}. We also
refer to the latest development due to Hildebrand \cite{Hi},
however, from the geometric viewpoint the arguments in \cite{Hi} is
difficult to be followed for us.

In centroaffine differential geometry, compared with its counterpart
in equiaffine differential geometry, the important {\it apolarity
condition} does not exist. The lack of the apolarity condition
brings serious difficulties to the solution of the problem of
classifying centroaffine hypersurfaces with parallel cubic form. To
our knowledge, besides Li and Wang \cite{LW}, the only known results
concentrating on this problem is by Liu and Wang \cite{LW2}, where
$2$-dimensional centroaffine surfaces were classified under the
condition that the traceless cubic form $\tilde{C}$ is parallel,
i.e. $\hat{\nabla}\tilde{C}=0$. As $\hat{\nabla}C=0$ implies that
$\hat{\nabla}\tilde{C}=0$, Liu and Wang's classification list should
include all immersions satisfying $\hat{\nabla}C=0$.

In this paper, restricting our attention to locally strongly convex
centroaffine hypersurfaces in $\mathbb{R}^{n+1}$, we will solve the
above problem by establishing a complete classification of all
centroaffine hypersurfaces with parallel cubic form. Similar to the
one in \cite{HLSV,HLV1,HLV2}, our classification depends heavily on
the characterization of the so-called (generalized) Calabi product
construction of centroaffine hypersurfaces (cf. \cite{LLS,LW}).
Indeed, such characterization tells how to decompose a complicated
centroaffine hypersurface into lower dimensional ones that have been
well known.

To state the main result of this paper, we first recall that if
$\psi_i:M_i\rightarrow \mathbb{R}^{n_i+1}$, where $i=1,2,$ are
non-degenerate centroaffine hypersurfaces, then, following
\cite{LLS,LW}, for a constant $\lambda\not=0,-1$, we can define the
(generalized) Calabi product of $M_{1}$ and $M_{2}$ by
\begin{equation}\label{eqn:1.3}
\psi(u,p,q)=(e^{u}\psi_{1}(p)
,e^{-\lambda u}\psi_{2}(q)),\ p\in M_{1},\
q\in M_{2},\ u \in \mathbb{R}.
\end{equation}

Similarly, the (generalized) Calabi product of $M_1$ and a point is
defined by
\begin{equation}\label{eqn:1.4}
\tilde\psi(u,p)=(e^{u}\psi_{1}(p),e^{-\lambda u}),\ p\in M_{1},\ u
\in \mathbb{R}.
\end{equation}

Note that a straightforward calculation shows that the Calabi
product of two centroaffine hypersurfaces with parallel cubic form
(resp. the Calabi product of a centroaffine hypersurface with
parallel cubic form and a point) again has parallel cubic form. The
decomposition theorems, which can be seen as the converse of the
previous Calabi product constructions, were first obtained in terms
of $h$ and $K$ in \cite{LW} (Theorem 4.5 therein) and will be
modified  more quantitatively in the present paper (cf. Theorems
\ref{th:3.2} and \ref{th:3.4} below) for maintaining consistency
with Theorems 3 and 4 of \cite{HLV1}. In this paper, we further
develop the techniques, started in \cite{HLSV} and \cite{HLV2} when
dealing with equiaffine hypersurfaces, in order to obtain the
following complete classification.

\begin{theorem}\label{th:1.1}

Let $M^n$ be an $n$-dimensional locally strongly convex centroaffine
hypersurface in $\mathbb{R}^{n+1}$ with $\hat{\nabla}C=0$. Then, we
have either
\begin{enumerate}
\item[(i)] $M^n$ is an open part of a locally strongly
convex hyperquadric ($C=0$); or
\vskip 1mm
\item[(ii)] $M^n$ is obtained as the Calabi product of a
lower dimensional locally strongly convex
centroaffine hypersurface with parallel cubic form and a point; or
\vskip 1mm
\item[(iii)] $M^n$ is obtained as the Calabi product of two
lower dimensional locally strongly convex
centroaffine hypersurfaces with parallel cubic form; or
\vskip 1mm
\item[(iv)] $n=\frac{1}{2}m(m+1)-1,\ m\ge3$, $M^n$ is centroaffinely equivalent to the standard
embedding of
$\mathrm{SL}(m,\mathbb{R})/\mathrm{SO}(m)\hookrightarrow
\mathbb{R}^{n+1}$; or
\vskip 1mm
\item[(v)] $n=\tfrac{1}{4}(m+1)^2-1,\ m\ge5$,
$M^n$ is centroaffinely equivalent to the standard embedding
$\mathrm{SL}(\tfrac{m+1}{2},\mathbb{C})/\mathrm{SU}(\tfrac{m+1}{2})
\hookrightarrow\mathbb{R}^{n+1}$;
or
\vskip 1mm
\item[(vi)] $n=\tfrac{1}{8}(m+1)(m+3)-1,\ m\ge9$,
$M^n$ is centroaffinely equivalent to the standard embedding
$\mathrm{SU}^*(\tfrac{m+3}{2})/\mathrm{Sp}(\tfrac{m+3}{4})\hookrightarrow
\mathbb{R}^{n+1}$;
or
\vskip 1mm
\item[(vii)] $n=26$,
$M^n$ is centroaffinely equivalent to the standard embedding\\
$\mathrm{E}_{6(-26)}/\mathrm{F}_4\hookrightarrow\mathbb{R}^{27}$; or
\vskip 1mm
\item[(viii)] $M^n$ is locally centroaffinely equivalent to the
canonical centroaffine hypersurface
$x_{n+1}=\tfrac{1}{2x_{1}}\sum_{k=2}^{n}x_{k}^{2}
+ x_{1}\ln x_{1}$.
\vskip 1mm
\end{enumerate}
\end{theorem}

\begin{remark}\label{rm:1.1}
Compared to its counterpart of the Classification Theorem in
equiaffine situation \cite{HLV2}, the case (viii) in Theorem
\ref{th:1.1} is exceptional and it is completely newly appeared.
\end{remark}

\vskip 1mm

\begin{remark}\label{rm:1.2}
Theorem \ref{th:1.1} implies that all canonical centroaffine
hypersurfaces but that in (viii) can be decomposed as the Calabi
product.
\end{remark}

\vskip 1mm
\begin{remark}\label{rm:1.3}
Related to Theorem \ref{th:1.1} we have established in \cite{CH} the
classification of locally strongly convex {\it isotropic}
centroaffine hypersurfaces. From the comparison of the main results
in \cite{BD} and \cite{CH} one sees that the {\it isotropic}
condition again have different implications in both equiaffine
theory of hypersurfaces and centroaffine theory of hypersurfaces,
just like Theorem \ref{th:1.1} here and the Classification Theorem
in \cite{HLV2}.
\end{remark}

\vskip 1mm

As direct consequence of Theorem \ref{th:1.1}, and without paying
attention to the Calabi product constructions, the classification of
locally strongly convex canonical centroaffine hypersurfaces can be
formulated as follows:

\begin{corollary}[cf. \cite{LW}]\label{cr:1.1}
Let $x:M^n\to\mathbb{R}^{n+1}$ be a locally strongly convex
canonical centroaffine hypersurface. Then it is locally
centroaffinely equivalent to one of the following hypersurfaces:
\begin{enumerate}
\item[(i)] $x_{1}^{\alpha_{1}}x_{2}^{\alpha_{2}}
\cdots x_{n+1}^{\alpha_{n+1}}=1$, where either $\alpha_{i}>0\ (1\le
i\le n+1)$, or
$$
\alpha_{1}<0\ \ and\ \ \alpha_{i}>0\ (2\le i\le n+1) \ such\ that\
\sum_{i=1}^{n+1}\alpha_i<0.
$$

%\vskip 2mm

\item[(ii)] $x_{1}^{\alpha_{1}}x_{2}^{\alpha_{2}}
\cdots x_{n-1}^{\alpha_{n-1}}(x_{n}^2 +x_{n+1}^2)^{\alpha_{n}}\exp
(\alpha _{n+1} \arctan\tfrac{x_n}{x_{n+1}})=1$, where
$$
\alpha_i<0\ (1\le i\le n-1)\ \ such\ that\ \
2\alpha_{n}+\sum_{i=1}^{n-1}\alpha_i>0,
$$

%\vskip 2mm

\item[(iii)] $x_{n+1}=\tfrac{1}{2x_{1}}
(x_{2}^{2}+\cdots+x_{v-1}^{2}) -x_{1}(-\ln x_{1}+\alpha_{v}\ln
x_{v}+\cdots+\alpha_{n}\ln x_{n})$, where $2\le v\le n+1$,
$\alpha_i\ (v\le i\le n)$ are real numbers satisfying
$$
\alpha_i>0\ (v\le i\le n)\ \ and\ \ \sum_{i=v}^{n}\alpha_i<1.
$$
\end{enumerate}
\end{corollary}

\begin{remark}\label{rm:1.4}
More general canonical centroaffine non-degenerate hypersurfaces
have been discussed by Li and Wang \cite{LW}, where the
classification of canonical centroaffine hypersurfaces in
$\mathbb{R}^{n+1}$ with $N(h)\leq 1$ was established. According to
\cite{LW}, it is easily seen that if $N(h)=0$ then such
hypersurfaces are centroaffinely equivalent to the following
hypersurfaces
$$
x_{1}^{\alpha_{1}}x_{2}^{\alpha_{2}} \cdots
x_{n+1}^{\alpha_{n+1}}=1,
$$
where $\alpha_i\ (1\le i\le n+1)$ are positive real numbers.
\end{remark}

\vskip 2mm

This paper is organized in twelve sections. In Section \ref{sect:2},
we fix notations and recall relevant material for centroaffine
hypersurfaces in affine differential geometry. In Section
\ref{sect:3}, we study both the Calabi product of centroaffine
hypersurfaces and their characterizations. In Section \ref{sect:4},
properties of centroaffine hypersurfaces with parallel cubic form in
terms of a typical basis are presented, so that the classification
problem of such hypersurfaces is divided into $(n+1)$ cases, namely:
$\{\mathfrak{C}_m\}_{1\le m\le n}$ and an exceptional case
$\mathfrak{B}$, depending on the decomposition of the tangent space
into three orthogonal distributions, i.e., $\mathcal{D}_1$ (of
dimension one), $\mathcal{D}_2$ and $\mathcal{D}_3$. The two cases
$\mathfrak{C}_1$ and $\mathfrak{C}_n$ will be settled in this
section. In Section \ref{sect:5}, we settle the exceptional Case
$\mathfrak{B}$.
In Section \ref{sect:6}, we classify locally strongly convex
centroaffine surfaces in $\mathbb{R}^3$ with parallel cubic form.
The result of Section \ref{sect:6} is necessary not only because it
is indispensable to the induction procedure of Theorem \ref{th:1.1},
but also because it fills in a gap in the result of Liu and Wang
\cite{LW2}.

To consider the cases $\{\mathfrak{C}_m\}_{2\le m\le n-1}$, we
follow closely the same procedure as in \cite{HLV2}: we introduce
two extremely important operators, i.e., an isotropic bilinear map
$L:\mathcal{D}_2\times \mathcal{D}_2\rightarrow \mathcal{D}_3$ in
subsection \ref{sect:4.3}, and, for any unit vector $v\in
\mathcal{D}_2$, the symmetric linear map $P_v: \mathcal{D}_2
\rightarrow \mathcal{D}_2$ in subsection \ref{sect:4.4}. With the
help of $L$ and $P_v$, we can give a remarkable decomposition of
$\mathcal{D}_2$ in subsection \ref{sect:4.5}. Then in Sections
\ref{sect:7}\,--\,\ref{sect:11}, according to the decomposition of
$\mathcal{D}_2$ we analyze these cases in much detail in order to
achieve the corresponding conclusion, respectively. Finally in
Section \ref{sect:12} we complete the proof Theorem \ref{th:1.1}.

\vskip2mm

%\noindent{\bf Acknowledgements}. The authors would like to thank the
%referee for their helpful suggestions and comments.

%%%%%%%%%%%%%%%%%%%%%%%%%%%%%%%%%%%%%%%%%%%%%%%%%%%%%%%%%%%%%%%%%%%%%%%%

\numberwithin{equation}{section}
\section{Preliminaries}\label{sect:2}

In this section, we recall basic facts about centroaffine
hypersurfaces. For more details see also \cite{NS} and \cite{SSV}.
Given a centroaffine hypersurface, let $\nabla$, $\hat{\nabla}$, $K$
and $C$ denote the induced connection, the Levi-Civita connection
for the centroaffine metric $h$, the difference tensor and the cubic
form, respectively, and let $X,Y,Z$ denote the tangent vector fields. We
define the Tchebychev form $\hat{T}$ and the Tchebychev vector field
$T$, respectively, by
\begin{equation}\label{eqn:2.1}
n\hat{T}(X)={\rm Tr}\,(K_{X}),\ \ h(T, X)=\hat{T}(X).
\end{equation}
If $T=0$, or equivalently, ${\rm Tr}\,K_X=0$ for any tangent
vector $X$, then $M^n$ is reduced to be the so-called {\it proper
{\rm (equi-)}affine hypersphere} centered at the origin $O$
(cf. also \cite{LSZH}, p.279, and for more details, in Section
1.15.2-3 therein). Using the cubic form $C$ and the Tchebychev form
$\hat{T}$ one can define a traceless symmetric cubic form
$\widetilde{C}$ by
\begin{equation}\label{eqn:2.2}
\begin{aligned}
\widetilde{C}(X,Y,Z):=&-\tfrac{1}{2}C(X,Y,Z)\\
&-\tfrac{n}{n+2}\big[\hat{T}(X)h(Y,Z)+\hat{T}(Y)h(X,Z)
+\hat{T}(Z)h(X,Y)\big].
\end{aligned}
\end{equation}

It is well-known that $\widetilde{C}$ vanishes if and only if $f:
M^n \rightarrow \mathbb{R}^{n+1}$ is a hyperquadric (cf.
Section 7.1 in \cite{SSV}; Lemma 2.1 and Remark 2.2 in \cite{LLSSW}).

Let $\hat{R}$ denote the curvature tensor of $\hat{\nabla}$. Then,
according to the integrability conditions, we have
\begin{equation}\label{eqn:2.3}
\hat{R}(X,Y)Z=\varepsilon(h(Y,Z)X-h(X,Z)Y)-[K_{X},K_{Y}]Z,
\end{equation}
\begin{equation}\label{eqn:2.4}
\hat{\nabla}C(X,Y,Z,W)=\hat{\nabla}C(Y,X,Z,W),
\end{equation}
where $\hat{\nabla}C(X,Y,Z,W):=(\hat{\nabla}_XC)(Y,Z,W)$.

We define the curvature tensor acting as derivation by
\begin{align*}
(\hat{R}(X,Y)K)(Z,U)=\hat{R}(X,Y)K(Z,U)-K(\hat{R}(X,Y)Z,U)-K(Z,\hat{R}(X,Y)U).
\end{align*}
Notice that $\hat{\nabla}C=0$ if and only if $\hat{\nabla}K=0$. Thus
if $\hat{\nabla}C=0$, we have
\begin{align}\label{eqn:2.5}
\hat{R}(X,Y)K(Z,U)=K(\hat{R}(X,Y)Z,U)+K(Z,\hat{R}(X,Y)U).
\end{align}
%

%%%%%%%%%%%%%%%%%%%%%%%%%%%%%%%%%%%%%%%%%%%%%%%%%%%%%%%%%%%%%%%%%%%%%%%%
\numberwithin{equation}{section}

\section{Characterizations of the generalized Calabi product}\label{sect:3}

To prove Theorem \ref{th:1.1}, we should study the (generalized)
Calabi products of centroaffine hypersurfaces as defined in
\eqref{eqn:1.3} and \eqref{eqn:1.4}. In this section, we first
state some elementary calculations on Calabi product, formulated as
Propositions \ref{pr:3.1} and \ref{pr:3.2}. Then, considering the
converse of these propositions, we will prove Theorems \ref{th:3.1},
\ref{th:3.2}, \ref{th:3.3} and \ref{th:3.4}, which demonstrate the
characterizations of the Calabi product in terms of their
centroaffine invariants.

Let $\psi_i:M_i\rightarrow \mathbb{R}^{n_i+1}$ be a locally strongly
convex centroaffine hypersurface of dimension $n_i\ (i=1,2)$. Denote
by $h^i$ the centroaffine metric of $\psi_i$ ($i=1,2$),
respectively. Given the Calabi product $\psi$ and $\tilde\psi$
defined as in \eqref{eqn:1.3} and \eqref{eqn:1.4}, i.e., for
constant $\lambda\not=0,-1$, we have
\begin{equation}\label{eqn:3.1}
\psi(u,p,q)=(e^{u}\psi_{1}(p),\,e^{-\lambda u}\psi_{2}(q)),\ \ p\in
M_{1},\ q\in M_{2},\ u \in \mathbb{R},
\end{equation}
\begin{equation}\label{eqn:3.2}
\tilde\psi(u,p)=(e^{u}\psi_{1}(p),\,e^{-\lambda u}),\ \ p\in M_{1},\
u \in \mathbb{R}.
\end{equation}

Let $\{u_{1},\ldots,u_{n_{1}}\}$ and $\{u_{n_{1}+1},
\ldots,u_{n_{1}+n_{2}}\}$ be local coordinates for $M_{1}$ and
$M_{2}$, respectively. For simplicity, we use the following range of
indices:
$$
1\leq i,j,k\leq n_1,\ \ n_1+1\leq \alpha,\beta,\gamma\leq n_1+n_2.
$$

According to Section 4 of Li and Wang \cite{LW}, we can state the
following result.

\begin{proposition}[cf. \cite{LW}]\label{pr:3.1}
The Calabi product of $M_1$ and $M_2$
$$
\begin{aligned}
\psi:M^{n_{1}+n_{2}+1}:=\mathbb{R}\times M_1\times M_2\rightarrow
\mathbb{R}^{n_1+n_2+2},
\end{aligned}
$$
defined by \eqref{eqn:3.1} is a non-degenerate centroaffine
hyersurface, the centroaffine metric $h$ induced by $\psi$ is given
by
\begin{equation}\label{eqn:3.3}
h=\lambda du^2\oplus\tfrac{\lambda}{1+\lambda}
h^{1}\oplus\tfrac{1}{1+\lambda}h^2,
\end{equation}
with the property
\begin{equation}\label{eqn:3.4}
N(h)=\left\{
\begin{aligned}
&N(h^{1})+N(h^{2}),\ \ \ \ \ \ \ \ \ \ \ \ \ \ \
\lambda>0,\\
&n_1+1-N(h^1)+N(h^2),\ \ \
-1<\lambda<0,\\
&n_2+1+N(h^{1})-N(h^{2}), \ \ \ \ \lambda<-1.
\end{aligned}
\right.
\end{equation}
The difference tensor $K$ of $\psi$ takes the following form:
\begin{align}\label{eqn:3.5}
\begin{split}
&K(\tfrac{\psi _{u}}{\sqrt{\mid\lambda\mid}},\tfrac{\psi _{u}}{\sqrt{\mid\lambda\mid}})
=\lambda_{1}\tfrac{\psi _{u}}{\sqrt{\mid\lambda\mid}},
\ \ \ \ \ \ K(\tfrac{\psi _{u}}{\sqrt{\mid\lambda\mid}},\psi_{u_{i}})=\lambda_{2}\psi_{u_{i}},\\
&K(\tfrac{\psi _{u}}{\sqrt{\mid\lambda\mid}},\psi _{u_{\alpha}})=\lambda_{3}\psi _{u_{\alpha}},
\ \ \ \ K(\psi_{u_{i}},\psi _{u_{\alpha}})=0,
\end{split}
\end{align}
where $\lambda_1,\lambda_2,\lambda_3$ are constants satisfying
\begin{equation}\label{eqn:3.6}
\lambda_1=\lambda_2+\lambda_3,\ \lambda_2\lambda_3=-{\rm
sgn}\,\lambda, \ \lambda_2\neq\lambda_3.
\end{equation}

Moreover, $\psi$ is flat (resp. of parallel cubic form) if and only
if both $\psi_1$ and $\psi_2$ are flat (resp. of parallel cubic
form).
\end{proposition}

\vskip 2mm

Similarly, the following result can be verified easily:

\begin{proposition}\label{pr:3.2}
The Calabi product of $M_1$ and a point
$$
\begin{aligned}
\tilde\psi:M^{n_1+1}=\mathbb{R}\times M_1\rightarrow
\mathbb{R}^{n_1+2}
\end{aligned}
$$
defined by \eqref{eqn:3.2} is a non-degenerate centroaffine
hyersurface, the centroaffine metric $\tilde{h}$ induced by
$\tilde\psi$ is given by
\begin{equation}\label{eqn:3.7}
\tilde{h}=\lambda du^2\oplus\tfrac{\lambda}{1+\lambda} h^{1},
\end{equation}
with the property
\begin{equation}\label{eqn:3.8}
N(\tilde{h})=\left\{
\begin{aligned}
&N(h^{1}),\ \ \ \ \ \ \ \ \ \ \ \ \ \ \ \ \ \ \ \
\lambda>0,\\
&n_1+1-N(h^{1}),\ \ \ \
-1<\lambda<0,\\
&N(h^{1})+1, \ \ \ \ \ \ \ \ \ \ \ \ \ \ \ \lambda<-1.
\end{aligned}
\right.
\end{equation}
The difference tensor $\tilde K$ of $\tilde\psi$  takes the
following form:
\begin{align}\label{eqn:3.9}
\begin{split}
&\tilde K(\tfrac{\tilde\psi_{u}}{\sqrt{\mid\lambda\mid}},\tfrac{\tilde\psi_{u}}{\sqrt{\mid\lambda\mid}})
=\lambda_{1}\tfrac{\tilde\psi_{u}}{\sqrt{\mid\lambda\mid}},
\ \ \tilde
K(\tfrac{\tilde\psi_{u}}{\sqrt{\mid\lambda\mid}},\tilde\psi_{u_{i}})=\lambda_{2}\tilde\psi_{u_{i}},
\end{split}
\end{align}
where $\lambda_1,\lambda_2$ are constants satisfying
\begin{equation}\label{eqn:3.10}
\lambda_1\neq2\lambda_2,\ \ \lambda_1\lambda_2-\lambda_2^2=-{\rm
sgn}\,\lambda.
\end{equation}

Moreover, $\tilde\psi$ is flat (resp. of parallel cubic form) if and
only if $\psi_1$ is flat (resp. of parallel cubic form).
\end{proposition}

\begin{remark}\label{rm:3.1}
From \eqref{eqn:3.4} and \eqref{eqn:3.8}, it is easily seen that if
the Calabi product $\psi$ (resp. $\tilde \psi$) is locally strongly
convex, then the centroaffine metric of $\psi$ (resp. $\tilde\psi$)
induced by $-\varepsilon' \psi$ (resp. $-\varepsilon'\tilde \psi$) is positive,
where $\varepsilon'=-{\rm sgn}\,\lambda$.
\end{remark}

Next, as the converse of Proposition \ref{pr:3.1}, we can prove the
following theorem.
\begin{theorem}\label{th:3.1}
Let $\psi:{M}^{n}\rightarrow\mathbb{R}^{n+1}$ be a locally strongly
convex centroaffine hypersurface. Assume that there exist
distributions $\mathcal{D}_{1}$ (of dimension $1$,
spanned by a unit vector field $T$), $\mathcal{D}_{2}$
(of dimension $n_{1}$) and $\mathcal{D}_{3}$
(of dimension $n_{2}$) such that
\begin{enumerate}
\item[(i)] \ $1+n_{1}+n_{2}=n$,
\item[(ii)]\ the
centroaffine metric $h$ induced by $-\varepsilon \psi$
($\varepsilon=\pm1$) is positive definite,
\item[(iii)]\ $\mathcal{D}_{1}$, $\mathcal{D}_{2}$ and
$\mathcal{D}_{3}$ are mutually orthogonal
with respect to the centroaffine metric $h$,
\item[(iv)]\  there exist constants
$\lambda_1,\lambda_2$, and $\lambda_3$ such that
\begin{align}\label{eqn:3.11}
\begin{split}
K(T,T)&=\lambda_{1}T,\  K(T,V)=\lambda_{2}V,\
K(T,W)=\lambda_{3}W,\ K(V,W)=0,\\
&\forall\,V\in \mathcal{D}_{2},\ W\in \mathcal{D}_{3};\ \
\lambda_1=\lambda_2+\lambda_3,\ \lambda_2\lambda_3=\varepsilon, \
\lambda_2\neq\lambda_3.
\end{split}
\end{align}
\end{enumerate}
Then $\psi:{M}^{n}\rightarrow\mathbb{R}^{n+1}$ can be locally
decomposed as the Calabi product of two lower dimensional locally
strongly convex centroaffine hypersurfaces
$\psi_1:M_{1}^{n_1}\rightarrow \mathbb{R}^{n_1+1}$ and
$\psi_2:M_{2}^{n_2}\rightarrow \mathbb{R}^{n_2+1}$.
\end{theorem}
\begin{proof}
First of all, we have the following lemma, whose proof can be given
exactly by following the proof of Lemmas 1, 2, 3 and 4 of
\cite{HLV1}.
\begin{lemma}\label{lm:3.1} Under the assumptions of Theorem \ref{th:3.1},
for any vector $X\in TM,\ V\in \mathcal{D}_2$ and $W\in
\mathcal{D}_3$, the following hold
$$
\hat{\nabla}_XT=0,\ \hat{\nabla}_XV\in \mathcal{D}_{2},\
\hat{\nabla}_XW\in \mathcal{D}_{3}.
$$
\end{lemma}

Lemma \ref{lm:3.1}, together with the de Rham decomposition theorem,
implies that $(M,h)$ is locally isometric to $\mathbb{R}\times
M_{1}\times M_{2}$, where $T$ is tangent to $\mathbb{R}$, whereas
$\mathcal{D}_{2}$ (resp. $\mathcal{D}_{3}$) is tangent to $M_{1}$
(resp. $M_{2}$).

The product structure of $M$ implies the existence of local
coordinates $(u,p,q)$ for $M$ based on an open subset containing the
origin of $\mathbb{R}^{n_1+n_2+1}$, such that $\mathcal{D}_{1}$ is
given by $dp=dq=0$, $\mathcal{D}_{2}$
(resp. $\mathcal{D}_{3}$) is given by $du=dq=0$ (resp. $du=dp=0$). We may
assume that $T =\lambda_2\tfrac{\partial}{\partial u}$. Put
\begin{align}\label{eqn:3.12}
\psi_{1}=f(T-\lambda_{3}\psi),\ \psi_{2}=g(\lambda_{2}\psi -T),
\end{align}
where $f$ and $g$ are assumed to be nonzero functions
which depend only on the variable $u$, and are given by
$$
f(u)=\tfrac{1}{\lambda_2-\lambda_3}e^{-u},\ g(u)=
\tfrac{1}{\lambda_2-\lambda_3}e^{-\tfrac{\lambda_3}{\lambda_2}u}.
$$

A straightforward computation, by \eqref{eqn:3.12} and
\eqref{eqn:1.1}, shows that
$$
\begin{aligned}
D_{T}\psi_{1}&=-\lambda_{2}f(T-\lambda_{3}\psi)
+fD_{T}(T-\lambda_{3}\psi)\\
&=f(\lambda_{3}\lambda_{2}-\varepsilon)\psi
+f(-\lambda_{2}+\lambda_{1}-\lambda_{3})T\\
&=0.
\end{aligned}
$$
Similarly
$$
D_{W}\psi_{1}=0,\ \ \ D_{T}\psi_{2}=D_{V}\psi_{2}=0.
$$
The above relations imply that $\psi_{1}$ (resp. $\psi_{2})$ reduces
to a map of $M_{1}$ (resp. $M_{2}$) in $\mathbb{R}^{n+1}$. The facts
$$
\begin{aligned}
&d\psi_{1}(V)=D_{V}\psi_{1}=f(\lambda_{2}-\lambda_{3})V,\\
&d\psi_{2}(W)=D_{W}\psi_{2}=g(\lambda_{2}-\lambda_{3})W
\end{aligned}
$$
show that both maps $\psi_{1}$ and $\psi_{2}$ are actually
immersions. Denoting by $\nabla^{1}$ (resp. $\nabla^{2}$) the
$\mathcal{D}_{2}$ (resp. $\mathcal{D}_{3}$) component of $\nabla$,
we further find that
$$
\begin{aligned}
D_{V}d\psi_{1}(\tilde{V})
&=f(\lambda_{2}-\lambda_{3})D_{V}\tilde{V}\\
 &=f(\lambda_{2}-\lambda_{3})\Big(\nabla^{1}_{V}\tilde{V}
 -\varepsilon h(V,\tilde{V})\psi+\lambda_{2}h(V,\tilde{V}) T\Big)\\
&=d\psi_{1}(\nabla^{1}_{V}\tilde{V})+(\lambda_{2}
-\lambda_{3})\lambda_{2}h(V,\tilde{V})\psi_{1}.
\end{aligned}
$$
Hence $\psi_{1}$ can be interpreted as a centroaffine
immersion contained in an $(n_{1}+1)$-dimensional vector
subspace of $\mathbb{R}^{n+1}$ with induced connection $\nabla^{1}$
and centroaffine metric
\begin{align}\label{eqn:3.13}
h^{1}=\lambda_{2}(\lambda_{2}-\lambda_{3})h.
\end{align}
Similarly, we obtain that $\psi_{2}$ can be interpreted
as a centroaffine immersion contained in an $(n_{2}+1)$-dimensional
vector subspace of $\mathbb{R}^{n+1}$ with induced
connection $\nabla^{2}$ and centroaffine metric
\begin{align}\label{eqn:3.14}
h^{2}=\lambda_{3}(\lambda_{3}-\lambda_{2})h.
\end{align}
As both subspaces are complementary, we may assume that, up to a
linear transformation, the $(n_{1}+1)$-dimensional subspace is
spanned by the first $n_{1}+1$ coordinates of $\mathbb{R}^{n+1}$,
whereas the $(n_{2}+1)$-dimensional subspace is spanned by the last
$n_{2}+1$ coordinates of $\mathbb{R}^{n+1}$.

Solving \eqref{eqn:3.12} for the immersion $\psi$, we have
$$
\psi =\tfrac{1}{(\lambda_{2}-\lambda_{3})f}\psi_{1}
+\tfrac{1}{(\lambda_{2}-\lambda_{3})g}\psi_{2}=(e^{u}\psi_{1},e^{\tfrac{\lambda_3}{\lambda_{2}}u}\psi_{2}).
$$
From Proposition \ref{pr:3.1} we see that
$\psi$ is given as the Calabi product of the immersions $\psi_{1}$
and $\psi_{2}$. Moreover, from \eqref{eqn:3.13} and
\eqref{eqn:3.14}, we know that both $\psi_{1}$ and $\psi_{2}$ are
locally strongly convex.

We have completed the proof of Theorem \ref{th:3.1}.
\end{proof}

In Theorem \ref{th:3.1}, if additionally $M$ has parallel cubic
form, equivalently, $\hat{\nabla}K=0$, then by the totally same
proof as that of Theorem 3 in \cite{HLV1}, we can prove the
following theorem.
\begin{theorem}\label{th:3.2}
Let $\psi:{M}^{n}\rightarrow\mathbb{R}^{n+1}$ be a locally strongly
convex centroaffine hypersurface. Assume that $\hat{\nabla} K=0$
and there exist $h$-orthogonal distributions $\mathcal{D}_{1}$ (of dimension $1$,
spanned by a unit vector field $T$), $\mathcal{D}_{2}$ (of dimension $n_{1}$)
and $\mathcal{D}_{3}$ (of dimension $n_{2}$) such that
\begin{align}\label{eqn:3.15}
\begin{split}
&K(T,T)=\lambda_{1}T,\  K(T,V)=\lambda_{2}V,\
K(T,W)=\lambda_{3}W,\\
&\forall\,V\in \mathcal{D}_{2},\ W\in \mathcal{D}_{3};\ \
\lambda_1\neq2\lambda_2,\ \lambda_1\neq2\lambda_3,\
\lambda_2\neq\lambda_3.
\end{split}
\end{align}
Then $\psi:{M}^{n}\rightarrow\mathbb{R}^{n+1}$ can be locally
decomposed as the Calabi product of two locally strongly convex
centroaffine hypersurfaces $\psi_1:M_{1}^{n_1}\rightarrow
\mathbb{R}^{n_1+1}$ and $\psi_2:M_{2}^{n_2}\rightarrow
\mathbb{R}^{n_2+1}$ with parallel cubic form.
\end{theorem}

Similarly, as the converse of Proposition \ref{pr:3.2}, we can
prove the following theorem.
\begin{theorem}\label{th:3.3}
Let $\psi:{M}^{n}\rightarrow\mathbb{R}^{n+1}$ be a locally strongly
convex centroaffine hypersurface.  Assume that there exist two
distributions $\mathcal{D}_{1}$ (of dimension $1$,
spanned by a unit vector field $T$), $\mathcal{D}_{2}$ (of dimension $n-1$)
such that
\begin{enumerate}
\item[(i)]\ the centroaffine metric $h$ induced by
$-\varepsilon \psi$ ($\varepsilon=\pm1$) is positive definite,
\item[(ii)]\ $\mathcal{D}_{1}$ and $\mathcal{D}_{2}$ are orthogonal
with respect to the centroaffine metric $h$,
\item[(iii)]\ there exist constants
$\lambda_1$ and $\lambda_2$ such that
\begin{align}\label{eqn:3.16}
\begin{split}
K(T,T)&=\lambda_{1}T,\  K(T,V)=\lambda_{2}V,\ \ \forall\, V\in
\mathcal{D}_{2};\\
&\lambda_1\neq2\lambda_2,\
\lambda_1\lambda_2-\lambda_2^2=\varepsilon.
\end{split}
\end{align}
\end{enumerate}
Then $\psi:{M}^{n}\rightarrow\mathbb{R}^{n+1}$ can be locally
decomposed as the Calabi product of a locally strongly convex
centroaffine hypersurface $\psi_1:M_{1}^{n-1}\rightarrow
\mathbb{R}^{n}$ and a point.
\end{theorem}
\begin{proof}
First, it is easily seen from \eqref{eqn:3.16} that we have
$$
\lambda_2\neq 0.
$$
Next, by a proof similar to those for Lemmas 5.6 and 5.7 in
\cite{HLLV}, we can prove the following lemma.
\begin{lemma}\label{lm:3.2}
Under the assumptions of Theorem \ref{th:3.3}, for any vector $X\in
TM$ and $V\in\mathcal{D}_2$, there hold
$$
\hat{\nabla}_XT=0,\ \ \hat{\nabla}_XV\in \mathcal{D}_{2}.
$$
\end{lemma}

From Lemma \ref{lm:3.2}, applying the de Rham decomposition theorem,
we see that $(M,h)$ is locally isometric with $\mathbb{R}\times
M_{1}$ such that $T$ is tangent to $\mathbb{R}$ and
$\mathcal{D}_{2}$ is tangent to $M_{1}$.

The above product structure of $M$ implies the existence of local
coordinates $(u,p)$ for $M$ based on an open subset containing the
origin of $\mathbb{R}^{n}$, such that $\mathcal{D}_{1}$ is given by
$dp=0$ and $\mathcal{D}_{2}$ is given by $du=0$. We may assume that
$T =\lambda_2\tfrac{\partial}{\partial u}$. Put
\begin{align}\label{eqn:3.17}
\psi_{1}=f\big(T-\tfrac{\varepsilon}{\lambda_2}\psi\big),\ \
\psi_{2}=g(\lambda_{2}\psi -T),
\end{align}
where $f$ and $g$ are assumed to be nonzero functions
which depend only on the variable $t$, and are given by
$$
f(u)=\tfrac{1}{2\lambda_2-\lambda_1}e^{-u},\ \ g(u)
=\tfrac{1}{2\lambda_2-\lambda_1} e^{\tfrac{\lambda_2-\lambda_1}{\lambda_2}u}.
$$

It follows from \eqref{eqn:3.17} that
$$
\begin{aligned}
D_{T}\psi_{1}&=-\lambda_{2}f\big(T-\tfrac{\varepsilon}{\lambda_2}
\psi\big)+f\big(D_{T}T-\tfrac{\varepsilon}{\lambda_2}D_{T}\psi\big)\\
&=f\big(-\lambda_{2}+\lambda_{1}-\tfrac{\varepsilon}{\lambda_2}\big)T\\
&=0.
\end{aligned}
$$
Similarly
$$
\begin{aligned}
&D_{T}\psi_{2}=D_{V}\psi_{2}=0,\\
&d\psi_{1}(V)=D_{V}\psi_{1}=(2\lambda_{2}-\lambda_{1})fV.\\
\end{aligned}
$$
The above relations imply that $\psi_{1}$ reduces to a map of $M_{1}$ in
$\mathbb{R}^{n+1}$. Whereas $\psi_{2}$ is a constant
vector in $\mathbb{R}^{n+1}$. Moreover, denoting by $\nabla^{1}$ the
$\mathcal{D}_{2}$ component of $\nabla$, we find that
$$
\begin{aligned}
D_{V}d\psi_{1}(\tilde{V})
&=f(2\lambda_{2}-\lambda_{1})D_{V}\tilde{V}\\
 &=f(2\lambda_{2}-\lambda_{1})\Big(\nabla^{1}_{V}\tilde{V}
 -\varepsilon h(V,\tilde{V})\psi+\lambda_{2}h(V,\tilde{V}) T\Big)\\
&=d\psi_{1}(\nabla^{1}_{V}\tilde{V})+(2\lambda_{2}
-\lambda_{1})\lambda_{2}h(V,\tilde{V})\psi_{1}.
\end{aligned}
$$
Hence $\psi_{1}$ can be interpreted as a centroaffine
immersion contained in an $n$-dimensional vector
subspace of $\mathbb{R}^{n+1}$ with induced connection
$\nabla^{1}$ and affine metric
\begin{align}\label{eqn:3.18}
h^{1}=\lambda_{2}(2\lambda_{2}-\lambda_{1})h.
\end{align}

As $\psi_{2}$ is transversal to the immersion $\psi_{1}$, we may
assume by a linear transformation that $\psi_{1}$ lies in the space
spanned by the first $n$ coordinates of $\mathbb{R}^{n+1}$, whereas
the constant vector $\psi_2$ lies in the direction of the last
coordinate.

Solving \eqref{eqn:3.17} for the immersion $\psi$, we have
$$
\begin{aligned}
\psi &=(e^{ u}\psi_{1},
e^{\tfrac{\lambda_1-\lambda_2}{\lambda_2}u}\psi_{2}).
\end{aligned}
$$
From Proposition \ref{pr:3.2} we see that
$\psi$ is given as the Calabi product of the immersion $\psi_{1}$
and a point. Moreover, from \eqref{eqn:3.18}, we know that
$\psi_{1}$ is a locally strongly convex centroaffine hypersurface.

This completes the proof of Theorem \ref{th:3.3}.
\end{proof}

Similarly, if $M$ in Theorem \ref{th:3.3} is assumed additionally
having parallel cubic form, then as deriving Theorem 4 in
\cite{HLV1}, we can prove the following theorem.
\begin{theorem}\label{th:3.4}
Let $\psi:{M}^{n}\rightarrow\mathbb{R}^{n+1}$ be a locally strongly
convex centroaffine hypersurface. Assume that $\hat{\nabla} K=0$ and
there exist $h$-orthogonal distributions $\mathcal{D}_{1}$ (of
dimension $1$, spanned by a unit vector field $T$) and
$\mathcal{D}_2$ (of dimension $n-1$) such that
\begin{align}\label{eqn:3.19}
\begin{split}
&K(T,T)=\lambda_{1}T,\  K(T,V)=\lambda_{2}V,\ \forall\,V\in D_{2};\
\ \lambda_1\neq2\lambda_2.
\end{split}
\end{align}
Then $\psi:{M}^{n}\rightarrow\mathbb{R}^{n+1}$ can be locally
decomposed as the Calabi product of a locally strongly convex
centroaffine hypersurface $\psi_1:M_{1}^{n-1}\rightarrow
\mathbb{R}^{n}$ with parallel cubic form and a point.
\end{theorem}
%
%%%%%%%%%%%%%%%%%%%%%%%%%%%%%%%%%%%%%%%%%%%%%%%%%%%%%%%%%%%%%%%%%%%%%%%%%%%%%%%%%%%%%%%%%%
\numberwithin{equation}{section}

\section{Elementary discussions in terms of a typical
basis}\label{sect:4}

In this section, we consider an $n$-dimensional $(n\ge2)$ locally
strongly convex centroaffine hypersurface $M^n$ in
$\mathbb{R}^{n+1}$ with $\hat{\nabla}C=0$ and we choose $\varepsilon$ such that the
centroaffine metric $h$ is positive definite. Our method here follows
closely that of \cite{HLSV,HLV2}.

Since $\hat{\nabla}C=0$ implies that $h(C,C)$ is constant, there are
two cases. First, if $h(C,C)=0$, as $h$ being positive definite we
have $C=0$ and $M^n$ is an open part of a quadric which is centered at the origin. If otherwise,
$h(C,C)\not=0$, then $C$ never vanishes. We assume this for the
remainder of this section.

\subsection{The construction of the typical basis}\label{sect:4.1}
~

\vskip 1mm

Let $p\in M$ and $UM_{p}=\{u\in T_{p}M \mid h(u,u)=1\}$. We define a
function on $UM_{p}$ by $f(u)=h(K_{u}u,u)$. Let $e_{1}$ be an
element of $UM_{p}$ at which the function $f(u)$ attains an absolute
maximum. The following lemma about the construction of the typical
basis can be proved totally similar to that of \cite{HLSV} (see also
\cite{VLS} for its earlier version).

\begin{lemma}[See P. 191 of \cite{HLSV}]\label{lm:4.1}
There exists an orthonormal basis $\{e_1,\ldots,e_n\}$ of $T_pM$
satisfying:
\begin{enumerate}
\item[(i)] $K_{e_{1}}e_{i}=\lambda_ie_i$, for $i=1,\ldots,n$,
where $\lambda_1\ (\lambda_{1}>0)$ is the maximum of $f$. Moreover,
for $i\ge2$, the value of $\lambda_i$ satisfies
\begin{equation}\label{eqn:4.1}
(\lambda_{1}-2\lambda_{i})(\varepsilon-\lambda_{1}\lambda_{i}+\lambda_{i}^{2})=0.
\end{equation}
\item[(ii)] for $i\ge2$, if\ \ $\lambda_1=2\lambda_i$, then
$f(e_i)=0$; if $\lambda_1\not=2\lambda_i$, then
$\lambda_{1}^{2}-4\varepsilon>0$ and
$\lambda_i=\mu:=\tfrac{1}{2}(\lambda_{1}-\sqrt{\lambda_{1}^{2}-4\varepsilon})$.
\end{enumerate}
\end{lemma}

\vskip 2mm

According to Lemma \ref{lm:4.1}, for a locally strongly convex
centroaffine hypersurface with parallel cubic form, we have to deal
with the following $(n+1)$-cases:

\vskip 1mm

\textbf{Case $\mathfrak{C}_1$}.\ \ $\lambda_{1}^{2}-4\varepsilon>0$
and $\lambda_{2}=\cdots=\lambda_{n}=\mu$.

\vskip 1mm

\textbf{Case $\mathfrak{C}_m$}.\ \ $\lambda_{1}^{2}-4\varepsilon>0$
and for some $m\ (2\leq m\leq n-1)$,
$$\lambda_{2}
=\cdots=\lambda_{m}=\tfrac{1}{2}\lambda_{1}, \ \
\lambda_{m+1}=\cdots=\lambda_{n}=\mu.
$$

\vskip 1mm

\textbf{Case $\mathfrak{C}_n$}.\ \
$\lambda_{1}^{2}-4\varepsilon\neq0$ and
$\lambda_{2}=\cdots=\lambda_{n}=\tfrac{1}{2}\lambda_{1}$.

\vskip 1mm

\textbf{Case $\mathfrak{B}$}.\ \ $\lambda_{1}^{2}-4\varepsilon=0$
and $\lambda_{2}=\cdots=\lambda_{n}=\tfrac{1}{2}\lambda_{1}$.

\vskip 2mm

In sequel of this paper, we are going to discuss these cases
separately.
\subsection{The settlement of the cases $\mathfrak{C}_1$ and
$\mathfrak{C}_n$}\label{sect:4.2}~

\vskip 1mm

First of all, about Case $\mathfrak{C}_1$, we have the following

\begin{theorem}\label{th:4.1}
If Case $\mathfrak{C}_1$ occurs, then $M^n$ can be locally
decomposed as the Calabi product of an $(n-1)$-dimensional locally
strongly convex centroaffine hypersurface in $\mathbb{R}^n$ with
parallel cubic form and a point.
\end{theorem}
\begin{proof}
In Case $\mathfrak{C}_1$, the difference tensor takes the
following form:
$$
\begin{aligned}
&K(e_1,e_1)=\lambda_1e_1,\ \ K(e_1,e_i)=\mu e_i,\ i=2,\ldots,n.
\end{aligned}
$$
By parallel translation along geodesics (with respect to
$\hat{\nabla}$) through $p$, we extend $\{e_1,\ldots,e_n\}$ to
obtain a local $h$-orthonormal basis denoted by
$\{E_1,\ldots,E_n\}$. Then
$$
\begin{aligned}
&K(E_1,E_1)=\lambda_1E_1,\ \ K(E_1,E_i)=\mu E_i,\
i=2,\ldots,n,\ \lambda_1\neq 2\mu,
\end{aligned}
$$
where both $\lambda_1$ and $\mu$ are defined in Lemma \ref{lm:4.1}.
Applying Theorem \ref{th:3.4}, we
conclude that $M^n$ can be decomposed as the
Calabi product of a locally strongly convex
centroaffine hypersurface with parallel cubic form and a point.
\end{proof}

\begin{theorem}\label{th:4.2}
Case $\mathfrak{C}_n$ does not occur.
\end{theorem}
\begin{proof}
Suppose on the contrary that Case $\mathfrak{C}_n$ does occur. From
(ii) of Lemma \ref{lm:4.1}, we have $f(v)=0$ for any $v\in{\rm
span}\,\{e_2,\ldots,e_n\}$. Then, by polarization, we can show that
\begin{equation}\label{eqn:4.2}
h(K_{e_i}e_j,e_k)=0,\ \ 2\leq i, j, k\leq n.
\end{equation}
Then, for any unit vector $v\in{\rm span}\,\{e_2,\ldots,e_n\}$, we
have
$$
K_{e_1}e_1=\lambda_1 e_1,\ K_{e_1}v=\tfrac{1}{2}\lambda_{1}v,\
K_vv=\tfrac{1}{2}\lambda_{1}e_1.
$$
Accordingly, by taking $X=e_1,\ Y=Z=U=v$ in \eqref{eqn:2.5}, we will
get $\lambda_1=0$. This contradiction completes the proof of Theorem \ref{th:4.2}.
\end{proof}

\vskip 2mm

\subsection{Intermediary cases $\{\mathfrak{C}_m\}_{2\leq
m\leq n-1}$ and an isotropic mapping $L$}\label{sect:4.3} ~

\vskip 1mm

Now, we consider the cases $\{\mathfrak{C}_m\}_{2\leq m\leq n-1}$.
In these cases, we denote by $\mathcal{D}_2$ and $\mathcal{D}_3$ the
two subspaces of $T_pM$:
$$
\mathcal{D}_2={\rm span}\{e_{2},\ldots,e_m\}\ \ {\rm and}\ \
\mathcal{D}_3={\rm span}\{e_{m+1},\ldots,e_n\}.
$$

First of all, we have the following
\begin{lemma}\label{lm:4.2}
Associated with the direct sum decomposition
$T_pM=\mathcal{D}_1\oplus \mathcal{D}_2 \oplus \mathcal{D}_3$, where
$\mathcal{D}_1={\rm span}\{e_{1}\}$, there hold the relations:
\begin{enumerate}
\item[(i)] $K_{e_1}v=\tfrac{1}{2}\lambda_1v,\ for\ any\ v
\in \mathcal{D}_2;\ K_{e_1}w=\mu w,\ for\ any\ w\in \mathcal{D}_3$.

\vskip 1mm

\item[(ii)] $K_{v_1}v_2 -\tfrac{1}{2}\lambda_1 h(v_1,v_2)e_1
\in \mathcal{D}_3,\ for\ any\ v_1,v_2\in \mathcal{D}_2$.

\vskip 1mm

\item[(iii)] $K_{v}w \in \mathcal{D}_2, \ for\ any\ v
\in \mathcal{D}_2,\ w\in \mathcal{D}_3$.
\end{enumerate}
\end{lemma}
\begin{proof}
By definition we have (i). The claim (ii) follows from (ii) of Lemma
\ref{lm:4.1} or directly \eqref{eqn:4.2}. In order to prove the
third claim, we take $X=v\in \mathcal{D}_2$, $Y=w \in \mathcal{D}_3$
and $Z=U=e_1$ in \eqref{eqn:2.5} to obtain that
$$
\lambda_1\hat{R}(v,w)e_{1}=2K(\hat{R}(v,w)e_{1},e_1).
$$
Thus we have $\hat{R}(v,w)e_{1}\in \mathcal{D}_2$.

On the other hand, a direct calculation by \eqref{eqn:2.3} gives
$$
\hat{R}(v,w)e_{1}=-K_vK_we_{1}+K_wK_ve_{1}=\big(\tfrac{1}{2}\lambda_1-\mu\big)K_vw.
$$
Therefore, as $\mu\not=\tfrac{1}{2}\lambda_1$, combining with the
preceding result we get $K_vw \in \mathcal{D}_2$.
\end{proof}

With the remarkable conclusions of Lemma \ref{lm:4.2}, similar to
that in \cite{HLV2}, we can now introduce a bilinear map
$L:\mathcal{D}_2\times \mathcal{D}_2\rightarrow \mathcal{D}_3$,
defined by
\begin{equation}\label{eqn:4.3}
L(v_1,v_2):=K_{v_1}v_2 -\tfrac{1}{2}\lambda_1 h(v_1,v_2)e_1, \
v_1,v_2\in \mathcal{D}_2.
\end{equation}

The following lemmas show that the operator $L$ enjoys remarkable
properties and it becomes an important tool for exploring
information of the difference tensor. As we have
$\lambda_{1}^{2}-4\varepsilon>0$, for simplicity, from now on we
denote $\eta:=\tfrac{1}{2}\sqrt{\lambda_{1}^{2}-4\varepsilon}$.
\begin{lemma}\label{lm:4.3}
The bilinear map $L$ is isotropic in the sense that
\begin{align}\label{eqn:4.4}
h(L(v,v),L(v,v))=\tfrac{1}{2}\lambda_1\eta h(v,v)^2,\ \
\forall\,v\in \mathcal{D}_2.
\end{align}
Moreover, linearizing \eqref{eqn:4.4}, it follows for arbitrary
$v_1,v_2,v_3,v_4\in \mathcal{D}_2$ that
\begin{align}\label{eqn:4.5}
\begin{split}
&h(L(v_1,v_2),L(v_3,v_4))+h(L(v_1,v_3),L(v_2,v_4))+h(L(v_1,v_4),L(v_2,v_3))\\
&=\tfrac{1}{2}\lambda_1\eta  (h(v_1,v_2)h(v_3,v_4)
+h(v_1,v_3)h(v_2,v_4)+h(v_1,v_4)h(v_2,v_3)).
\end{split}
\end{align}
\end{lemma}
\begin{proof}
We use \eqref{eqn:2.5} and take $X=e_1$ and $Y=v_1,\ Z=v_2,\ U=v_3$
in $\mathcal{D}_2$. By using \eqref{eqn:2.3} and the definition of
$L$, it follows immediately that
\begin{align}\label{eqn:4.6}
\begin{split}
&K(L(v_1,v_2),v_3)+K(L(v_1,v_3),v_2)+K(L(v_2,v_3),v_1)\\&=
\tfrac{1}{2}\lambda_1\eta  (h(v_1,v_2)v_3+h(v_1,v_3)v_2+h(v_2,v_3)v_1).
\end{split}
\end{align}
Taking the product of \eqref{eqn:4.6} with $v_4\in
\mathcal{D}_2$, we can obtain \eqref{eqn:4.5}.
Finally, we choose $v_1=v_2=v_3=v_4=v$ in \eqref{eqn:4.5},
then we get \eqref{eqn:4.4}.
\end{proof}

Since $L:\mathcal{D}_2\times \mathcal{D}_2\rightarrow \mathcal{D}_3$
is isotropic, we see from \eqref{eqn:4.4} that, if $\dim\,
\mathcal{D}_2 \geq1$, then the image space of $L$ has positive
dimension, i.e. $\dim\,({\rm Im}\,L)\geq1$. Moreover, the following
well-known properties hold.

\begin{lemma}[cf. \cite{HLSV,HLV2}]\label{lm:4.4}
If\ \ $\dim\, \mathcal{D}_2\geq1$, for orthonormal vectors
$v_1,v_2,v_3$ and $v_4\in \mathcal{D}_2$, there hold
\begin{equation}\label{eqn:4.7}
h(L(v_1,v_1),L(v_1,v_2))=0,
\end{equation}
\begin{equation}\label{eqn:4.8}
h(L(v_1,v_1),L(v_2,v_2))+2h(L(v_1,v_2),L(v_1,v_2))=\tfrac{1}{2}\lambda_1\eta,
\end{equation}
\begin{equation}\label{eqn:4.9}
h(L(v_1,v_1),L(v_2,v_3))+2h(L(v_1,v_2),L(v_1,v_3))=0,
\end{equation}
\begin{align}\label{eqn:4.10}
\begin{split}
h(L(v_1,v_2),L(v_3,v_4))&+h(L(v_1,v_3),L(v_2,v_4))\\&+h(L(v_1,v_4),L(v_2,v_3))=0.
\end{split}
\end{align}
\end{lemma}

\begin{lemma}\label{lm:4.5}
In Cases $\{\mathfrak{C}_m\}_{2\le m\le n-1}$, if it occurs that
${\rm Im}\,L\neq \mathcal{D}_3$, then for any $v_1$, $v_2\in
\mathcal{D}_2$ and $w\in \mathcal{D}_3$ with $w\perp{\rm Im}\,L$, we
have
\begin{align}\label{eqn:4.11}
K(L(v_1,v_2),w)=\eta\mu h(v_1,v_2)w.
\end{align}
\end{lemma}
\begin{proof}
For every $v\in \mathcal{D}_2$ and $w\perp{\rm Im}(L)$, we apply
(iii) of Lemma \ref{lm:4.2} and \eqref{eqn:2.3} to obtain
$$
\begin{aligned}
K(v,w)&=\sum_{i=2}^{m}h(K(v,w),e_i)e_i=\sum_{i=2}^{m}h(K(v,e_i),w)e_i\\
&=\sum_{i=2}^{m}h(L(v,e_i),w)e_i=0,
\end{aligned}
$$
$$
\begin{aligned}
\hat{R}(e_1,v)w=-K(K(v,w),e_1)+K(v,K(e_1,w))=0.
\end{aligned}
$$
Then, for $v_1,v_2$ and $w$ as in the assumptions, the following equation
$$
\begin{aligned}
\hat{R}(e_1,v_1)K(v_2,w)=K(\hat{R}(e_1,v_1)v_2,w)+K(v_2,\hat{R}(e_1,v_1)w)
\end{aligned}
$$
becomes equivalent to $K(\hat{R}(e_1,v_1)v_2,w)=0$. On the other
hand, direct calculation gives that
$$
\begin{aligned}
\hat{R}(e_1,v_1)v_2&=\varepsilon h(v_1,v_2)e_1
-K(e_1,K(v_1,v_2))+K(v_1,K(e_1,v_2))\\
&=\varepsilon h(v_1,v_2)e_1-K(e_1,L(v_1,v_2)
+\tfrac{1}{2}\lambda_1h(v_1,v_2)e_1)\\
&\ \ \ +\tfrac{1}{2}\lambda_1(L(v_1,v_2)
+\tfrac{1}{2}\lambda_1 h(v_1,v_2)e_1)\\
&=-\eta^2h(v_1,v_2)e_1+\eta L(v_1,v_2).
\end{aligned}
$$
Then \eqref{eqn:4.11} immediately follows.
\end{proof}

\begin{lemma}\label{lm:4.6}
In Cases $\{\mathfrak{C}_m\}_{2\le m\le n-1}$, let $v_1,v_2,v_3,v_4
\in \mathcal{D}_2$ and $\{u_1,\ldots,u_{m-1}\}$ be an orthonormal
basis of $\mathcal{D}_2$, then we have
\begin{align}\label{eqn:4.12}
K(L(v_1,v_2),L(v_3,v_4))=&\mu h(L(v_1,v_2),L(v_3,v_4))e_1
+\mu \eta h(v_1,v_2)L(v_3,v_4)\nonumber \\
&+\sum_{i=1}^{m-1}
h(L(v_1,u_i),L(v_3,v_4))L(u_i,v_2)\nonumber \\
& +\sum_{i=1}^{m-1}
h(L(v_2,u_i),L(v_3,v_4))L(u_i,v_1).
\end{align}
\end{lemma}
\begin{proof}
By \eqref{eqn:2.5}, we have, for $v_1,v_2,v_3,v_4
\in \mathcal{D}_2$, that
\begin{align}\label{eqn:4.13}
&\hat{R}(e_1,v_1)K(v_2,L(v_3,v_4))\nonumber \\
&=K(\hat{R}(e_1,v_1)v_2,
L(v_3,v_4))+K(v_2,\hat{R}(e_1,v_1)L(v_3,v_4)).
\end{align}
Applying
\eqref{eqn:2.3} for $v_1, v_2\in \mathcal{D}_2$, we obtain that
\begin{equation}\label{eqn:4.14}
\hat{R}(e_1,v_1)v_2=-\eta^2h(v_1,v_2)e_1+\eta L(v_1,v_2).
\end{equation}
Similarly, for $v\in \mathcal{D}_2$ and
$w\in \mathcal{D}_3$, we have that
\begin{equation}\label{eqn:4.15}
\hat{R}(e_1,v)w=-\eta K(v,w).
\end{equation}
By Lemma \ref{lm:4.2}, $K(v_2,L(v_3,v_4))\in \mathcal{D}_2$ and
we can write
\begin{equation}\label{eqn:4.16}
K(v_2,L(v_3,v_4))=\sum_{i=1}^{m-1}
h(L(v_2,u_i),L(v_3,v_4))u_i.
\end{equation}

Now, we can compute both sides of \eqref{eqn:4.13} to obtain
\begin{align*}%\label{eqn:4.20}
&{\rm LHS}=-\eta^2h(L(v_1,v_2),L(v_3,v_4))e_1+\eta \sum_{i=1}^{m-1}
h(L(v_2,u_i),L(v_3,v_4))L(u_i,v_1),\\
&{\rm RHS}=-\mu \eta^2 h(v_1,v_2)L(v_3,v_4)
+\eta K(L(v_1,v_2),L(v_3,v_4))\\
&\ \ \ \ \ \ \ \ \
-\tfrac{1}{2}\lambda_1\eta  h(L(v_1,v_2),L(v_3,v_4))e_1\\
&\ \ \ \ \ \ \ \ \ -\eta\sum_{i=1}^{m-1}
h(L(v_1,u_i),L(v_3,v_4))L(u_i,v_2).
\end{align*}
From these computations we immediately get \eqref{eqn:4.12}.
\end{proof}

We note that \eqref{eqn:4.12} has very important consequences which
will be used in sequel sections. For example, we have
\begin{lemma}\label{lm:4.7}
For Case $\mathfrak{C}_m$ with $m\ge3$, let $\{u_1,\ldots,u_{m-1}\}$
be an orthonormal basis of $\mathcal{D}_2$, then for $p\neq j$, we
have
\begin{align}\label{eqn:4.17}
0=&\big(\eta(\eta+\tfrac{1}{2}\lambda_1 )-4h(L(u_j,u_p),
L(u_j,u_p))\big)L(u_j,u_p)\nonumber \\
&-\sum_{i\neq p}
4h(L(u_j,u_i),
L(u_j,u_p))L(u_i,u_j).
\end{align}
In particular, if $L(u_1,u_2)\neq 0$ and $L(u_1,u_i)$
is orthogonal to $L(u_1,u_2)$ for all
$i\neq 2$, then
\begin{align}\label{eqn:4.18}
h(L(u_1,u_2),L(u_1,u_2))=\tfrac{1}{4}\eta(\eta+\tfrac{1}{2}\lambda_1)
=:\tau.
\end{align}
\end{lemma}
\begin{proof}
By \eqref{eqn:4.12}, interchanging the couples of
indices $\{1,2\}$ and $\{3,4\}$ we find
the following condition:
\begin{align}\label{eqn:4.19}
0=&\mu \eta \big(h(v_1,v_2)L(v_3,v_4))-h(v_3,v_4)L(v_1,v_2))
\big)\nonumber \\
&+\sum_{i=1}^{m-1}
h(L(v_1,u_i),L(v_3,v_4))L(u_i,v_2)\nonumber \\
& +\sum_{i=1}^{m-1}
h(L(v_2,u_i),L(v_3,v_4))L(u_i,v_1)\nonumber \\
& -\sum_{i=1}^{m-1}
h(L(v_3,u_i),L(v_1,v_2))L(u_i,v_4)\nonumber \\
&-\sum_{i=1}^{m-1}
h(L(v_4,u_i),L(v_1,v_2))L(u_i,v_3).
\end{align}

If we take $v_2=v_3=v_4=u_j$ and $v_1=u_p$ with $j\not=p$, then by
using also the isotropy condition, \eqref{eqn:4.19} reduces to
\eqref{eqn:4.17}. Taking $j=1$ and $p=2$ in \eqref{eqn:4.17}, we
obtain \eqref{eqn:4.18}.
\end{proof}

\vskip 1mm

\subsection{The mapping $P_v:\mathcal{D}_2\rightarrow \mathcal{D}_2$ with unit vector
$v\in \mathcal{D}_2$}\label{sect:4.4}~

\vskip 1mm

We now define for any given unit vector
$v\in \mathcal{D}_2$ a linear map $P_v:\mathcal{D}_2
\rightarrow \mathcal{D}_2$ by
\begin{equation}\label{eqn:4.20}
P_v\tilde{v}=K_vL(v,\tilde{v}),\ \forall\,\tilde{v}\in
\mathcal{D}_2.
\end{equation}

It is easily seen that $P_v$ is a symmetric linear operator
satisfying
\begin{equation}\label{eqn:4.21}
h(P_v\tilde{v},v')=h(L(v,\tilde{v}),L(v,v'))=h(P_vv',\tilde{v}),
\end{equation}
for any $\tilde{v},\,v'\in \mathcal{D}_2$. Moreover, we have

\begin{lemma}\label{lm:4.8}
For any unit vector $v\in \mathcal{D}_2$, the operator
$P_v:\mathcal{D}_2 \rightarrow \mathcal{D}_2$ has
$\sigma=\tfrac{1}{2}\lambda_1\eta $ as an eigenvalue with
eigenvector $v$. In the orthogonal complement $\{v\}^\perp$ of
$\{v\}$ in $\mathcal{D}_2$ the operator $P_v$ has at most two
eigenvalues, namely $0$ and $\tau$, defined as in \eqref{eqn:4.18}.
\end{lemma}
\begin{proof}
By \eqref{eqn:4.4}, we have
\begin{equation}\label{eqn:4.22}
h(P_vv,v)=h(L(v,v),L(v,v))=\tfrac{1}{2}\lambda_1\eta .
\end{equation}
Taking $v'\perp v$, we get
\begin{equation}\label{eqn:4.23}
h(P_vv,v')=h(L(v,v'),L(v,v))=0.
\end{equation}
\eqref{eqn:4.22} and \eqref{eqn:4.23} imply that
$P_vv=\tfrac{1}{2}\lambda_1\eta v$.

Next, we take an orthonormal basis $\{u_1,\ldots,u_{m-1}\}$ of
$\mathcal{D}_2$ consisting of eigenvectors of $P_v$ such that
$P_vu_i=\sigma_iu_i,\ i=1,\ldots,m-1$, with $u_1=v$ and
$\sigma_1=\sigma$. We take the inner product of \eqref{eqn:4.17}
with $L(u_1,u_p)$ for $j=1$ and any $p\geq 2$. We obtain that
\begin{equation}\label{eqn:4.24}
h(L(u_1,u_p),L(u_1,u_p))\big(\tau-h(L(u_1,u_p),L(u_1,u_p))\big)=0,\
\ p\ge2.
\end{equation}
Here, to derive \eqref{eqn:4.24}, we have used that
\begin{equation*}
h(L(u_1,u_p),L(u_1,u_i))=h(u_p,P_{u_1}u_i)=0,\ i\neq p.
\end{equation*}
From \eqref{eqn:4.24}, we immediately get the remaining assertion.
\end{proof}
In the following we denote by $V_v(0)$ and $V_v(\tau)$
the eigenspaces of $P_v$ (in the orthogonal
complement of $\{v\}$) with respect to the eigenvalues 0
and $\tau$, respectively. Note that in exceptional cases
it can happen that $\sigma=\tau$.
\begin{lemma}\label{lm:4.9}
Let $v,u\in \mathcal{D}_2$ be two unit orthogonal vectors. Then the
following statements are equivalent:
\begin{enumerate}
\item[(i)] $u\in V_v(0)$.
\item[(ii)] $L(u,v)=0$.
\item[(iii)] $L(u,u)=L(v,v)$.
\item[(iv)] $v\in V_u(0)$.
\end{enumerate}
Moreover, any of the previous statements implies that
\begin{enumerate}
\item[(v)] $P_v=P_u$ on $\{u,v\}^{\perp} $.
\end{enumerate}
\end{lemma}
\begin{proof}
As $h(P_vu,u)=h(L(v,u),L(v,u))=h(P_uv,v)$, the equivalence
of (i), (ii) and (iv) follows immediately. As $u$ and $v$ are
orthogonal, \eqref{eqn:4.4} and \eqref{eqn:4.8} imply that
\begin{equation*}
h(L(v,v)-L(u,u),L(v,v)-L(u,u))=4h(L(v,u),L(v,u)).
\end{equation*}
It follows that (ii) is equivalent to (iii).

Now we assume that (i), (ii), (iii) and (iv) are satisfied. In order
to prove (v), we see that the space spanned by $\{u,v\}$ is
invariant by $P_v$ and $P_u$, also its orthogonal complement is
invariant. By taking $v_1,v_2 \in\{u,v\}^\perp$ and using
\eqref{eqn:4.6}, we find
$$
\begin{aligned}
h(v_1,P_vv_2)&=h(L(v,v_1),L(v,v_2))\\
&=-\tfrac{1}{2}h(L(v,v),L(v_1,v_2))+
\tfrac{1}{4}\lambda_1\eta h(v_1,v_2)\\
&=-\tfrac{1}{2}h(L(u,u),L(v_1,v_2))+
\tfrac{1}{4}\lambda_1\eta h(v_1,v_2)\\
&=h(v_1,P_uv_2).
\end{aligned}
$$
This completes the proof.
\end{proof}

\begin{lemma}\label{lm:4.10}
Let $v,\tilde{v}\in \mathcal{D}_2$ be two unit orthogonal vectors, then
\begin{equation}\label{eqn:4.25}
\begin{aligned}
h(L(v,\tilde{v}),L(v,\tilde{v}))=\tau
\end{aligned}
\end{equation}
holds if and only if $\tilde{v}\in V_v(\tau)$. Moreover, if we
assume $u\in V_v(0)$ and the equality in \eqref{eqn:4.25} holds,
then $u\in V_{\tilde{v}}(\tau)$.
\end{lemma}
\begin{proof}
If $\tilde{v}\in V_v(\tau)$, then $h(L(v,\tilde{v}),L(v,\tilde{v}))
=h(\tilde{v},P_v\tilde{v})=\tau$.

Conversely, if $h(L(v,\tilde{v}),L(v,\tilde{v}))=\tau$,
we should consider the following three cases:

(i) $V_v(0)=\emptyset$. From Lemma \ref{lm:4.8}, it is easily seen
that $\tilde{v}\in V_v(\tau)$.

(ii) $V_v(\tau)=\emptyset$. In this case, Lemma \ref{lm:4.8}
implies that $\tilde{v}\in V_v(0)$. By Lemma \ref{lm:4.9}, we have
$h(L(v,\tilde{v}),L(v,\tilde{v}))=0. $
This is a contradiction.

(iii) $V_v(0)\neq\emptyset$ and $V_v(\tau)\neq\emptyset$. We can
write
$$
\begin{aligned}
\tilde{v}=\cos\theta v_0+\sin\theta v_1,\  h(v_0,v_0)=h(v_1,v_1)=1,
\end{aligned}
$$
where $v_0 \in V_v(0)$ and $v_1\in V_v(\tau)$. Then we get
$$
\begin{aligned}
\tau=h(L(v,\tilde{v}),L(v,\tilde{v}))=\sin^2\theta\tau,
\end{aligned}
$$
which means that $\sin\theta=\pm1$ and $\cos\theta=0$.
Therefore, $\tilde{v}\in V_v(\tau)$.

Taking unit vector $u \in V_v(0)$, we have $L(u,u)=L(v,v)$.
Consequently,
$$
\begin{aligned}
&h(L(\tilde{v},u),L(\tilde{v},u))\\
&=-\tfrac{1}{2}h(L(\tilde{v},\tilde{v}),L(u,u))+
\tfrac{1}{4}\lambda_1\eta \\
&=-\tfrac{1}{2}h(L(v,v),L(\tilde{v},\tilde{v}))+
\tfrac{1}{4}\lambda_1\eta\\
&=h(v,P_{\tilde{v}}v)=\tau.
\end{aligned}
$$
Applying the first assertion of Lemma \ref{lm:4.10}, we have $u\in
V_{\tilde{v}}(\tau)$.
\end{proof}

\begin{lemma}\label{lm:4.11}
Let $v_1,v_2,v_3\in \mathcal{D}_2$ be orthonormal vectors satisfying
$v_1,v_2\in V_{v_3}(\tau)$, then for any vector $v\in
\mathcal{D}_2$, we have $h(L(v_1,v_2),L(v,v_3))=0$.
\end{lemma}
\begin{proof}
Using the linearity of the assertion with $v$, we may assume that
$v$ is an eigenvector of $P_{v_3}$. Let $\{u_1,\ldots,u_{m-1}\}$ be
an orthonormal basis of $\mathcal{D}_2$ consisting of eigenvectors
of $P_{v_3}$ such that $u_1=v_1,\ u_2=v_2$ and $u_3=v_3$. We now use
\eqref{eqn:4.19} for $v_3=v_4$ to obtain
\begin{align}\label{eqn:4.26}
0=&-\mu \eta L(v_1,v_2)+\sum_{i=1}^{m-1}
h(L(v_1,u_i),L(v_3,v_3))L(u_i,v_2)\nonumber \\
& +\sum_{i=1}^{m-1}
h(L(v_2,u_i),L(v_3,v_3))L(u_i,v_1)\nonumber \\
& -2\sum_{i=1}^{m-1}
h(L(v_3,u_i),L(v_1,v_2))L(u_i,v_3).
\end{align}

On the other hand, from \eqref{eqn:4.7}\,--\,\eqref{eqn:4.9}, we
have
$$
\begin{aligned}
&h(L(v_1,u_i),L(v_3,v_3))=h(L(v_2,u_j),L(v_3,v_3))=0,\
i\neq1,j\neq2,\\
&h(L(v_1,v_1),L(v_3,v_3))=h(L(v_2,v_2),L(v_3,v_3))=
\tfrac{1}{2}\lambda_1\eta -2\tau .
\end{aligned}
$$

Inserting the above into \eqref{eqn:4.26}, we obtain
\begin{align}\label{eqn:4.27}
0=\sum_{i=1}^{m-1}
h(L(v_3,u_i),L(v_1,v_2))L(u_i,v_3).
\end{align}

Since $h(L(u_i,v_3),L(u_j,v_3))=h(P_{v_3}u_i,u_j)=0$ if $i\neq j$,
the equation \eqref{eqn:4.27} implies that
$h(L(v_3,u_i),L(v_1,v_2))=0$ holds for all $u_i\notin V_{v_3}(0)$.
Combining with Lemma \ref{lm:4.9}, this immediately shows that for
any vector $v\in \mathcal{D}_2$, we have $h(L(v_1,v_2),L(v,v_3))=0$.
\end{proof}

\vskip 1mm

\subsection{Direct sum
decomposition for $\mathcal{D}_2$}\label{sect:4.5}~

\vskip 1mm

For our purpose, a crucial matter is to introduce a direct sum
decomposition for $\mathcal{D}_2$ based on the preceding Lemmas.
First, pick any unit vector $v_1\in\mathcal{D}_2$ and recall that
$\tau=\tfrac{1}{4}\eta(\eta+\tfrac{1}{2}\lambda_1)$, then by Lemma
\ref{lm:4.8}, we have a direct sum decomposition for
$\mathcal{D}_2$:
$$
\mathcal{D}_2=\{v_1\}\oplus V_{v_1}(0)\oplus V_{v_1}(\tau),
$$
where, here and later on, we denote also by $\{\cdot\}$ the vector
space spanned by its elements. If $V_{v_1}(\tau)\neq \emptyset$, we
take an arbitrary unit vector $v_2\in V_{v_1}(\tau)$. Then by Lemma
\ref{lm:4.10} we have:
$$
v_1\in V_{v_2}(\tau),\ \ V_{v_1}(0)\subset V_{v_2}(\tau)\ \ {\rm
and} \ \ V_{v_2}(0)\subset V_{v_1}(\tau).
$$
From this we deduce that
$$
\mathcal{D}_2=\{v_1\}\oplus V_{v_1}(0)\oplus \{v_2\}\oplus
V_{v_2}(0) \oplus\big(V_{v_1}(\tau)\cap V_{v_2}(\tau)\big).
$$

If $V_{v_1}(\tau)\cap V_{v_2}(\tau)\neq \emptyset$, we further pick a
unit vector $v_3\in V_{v_1}(\tau)\cap V_{v_2}(\tau)$. Then
$$
\mathcal{D}_2=\{v_3\}\oplus V_{v_3}(0)\oplus V_{v_3}(\tau),
$$
and by Lemma \ref{lm:4.10} we have
$$
v_1,v_2\in V_{v_3}(\tau);\ V_{v_1}(0),\ V_{v_2}(0)\subset
V_{v_3}(\tau).
$$
It follows that
$$
\begin{aligned}
\mathcal{D}_2=\{v_1\}\oplus V_{v_1}(0)\oplus
\{v_2\}\oplus V_{v_2}(0)&\oplus\{v_3\}\oplus V_{v_3}(0)\\
&\oplus\big(V_{v_1}(\tau)\cap V_{v_2}(\tau)\cap V_{v_3}(\tau)\big).
\end{aligned}
$$

Considering that $\dim\,(\mathcal{D}_2)=m-1$ is finite, by
induction, we get
\begin{proposition}\label{pr:4.1}
In Cases $\{\mathfrak{C}_m\}_{2\le m\le n-1}$, there exists an
integer $k_0$ and unit vectors $v_1,\ \ldots,\ v_{k_0}\in
\mathcal{D}_2$ such that
\begin{align}\label{eqn:4.28}
\mathcal{D}_2=\{v_1\}\oplus V_{v_1}(0)\oplus \cdots
\oplus \{v_{k_0}\}\oplus V_{v_{k_0}}(0).
\end{align}
\end{proposition}

\vskip 1mm

In what follows, we will study the decomposition \eqref{eqn:4.28} in
more details. %We denote $\{v_k\}\oplus V_{v_k}(0)$ by $V_k$.
\begin{lemma}\label{lm:4.12}
\begin{enumerate}
\item[(i)]
For any unit vector $u_1\in\{v_1\}\oplus V_{v_1}(0)$, we have
$$
\{v_1\}\oplus V_{v_1}(0)=\{u_1\}\oplus V_{u_1}(0).
$$
\item[(ii)] For any orthonormal vectors $u_1,\tilde{u}_1\in\{v_1\}
\oplus V_{v_1}(0)$, we
have $L(u_1, \tilde{u}_1)=0$.
\end{enumerate}
\end{lemma}
\begin{proof}
(i)\ \ We first assume the special case that $u_1 \perp v_1$. Then
we have $u_1\in V_{v_1}(0)$ and thus $L(u_1, v_1)=0$, hence $v_1\in
V_{u_1}(0)$. Let $u\in V_{v_1}(0)$ and write $u=x_1u_1+u'$ with
$u'\perp u_1$. By (v) in Lemma \ref{lm:4.9} we have
$P_{u_1}u'=P_{v_1}u'=P_{v_1}(u-x_1u_1)=0$. Therefore, $u'\in
V_{u_1}(0)$ and $\{v_1\}\oplus V_{v_1}(0)\subset \{u_1\}\oplus
V_{u_1}(0)$. Similarly, we obtain $\{u_1\}\oplus
V_{u_1}(0)\subset\{v_1\}\oplus V_{v_1}(0)$.

Next we consider the general case in three subcases. (a) If
$V_{v_1}(0)=\emptyset$, there is nothing to prove. (b) If
$\dim\,(V_{v_1}(0))\geq 2$, we can take a vector $\tilde{u}\in
V_{v_1}(0)$ which is orthogonal to both $u_1$ and $v_1$. Applying
twice the previous result then completes the proof. (c) If
$\dim\,(V_{v_1}(0))=1$, there exists $v_0\in V_{v_1}(0)$ such that
$V_{v_1}(0)=\{v_0\}$. Denote $u_1=\cos\theta v_1+\sin\theta v_0$. By
Lemma \ref{lm:4.9}, we see that
$$
L(\cos\theta v_1+\sin\theta v_0,\cos\theta v_0-\sin\theta v_1)=0,
$$
thus $\cos\theta\, v_0-\sin\theta\, v_1\in V_{u_1}(0)$.
Therefore, $\{v_1\}\oplus V_{v_1}(0)\subset\{u_1\}\oplus
V_{u_1}(0)$. If $\{v_1\}\oplus V_{v_1}(0)\subsetneq\{u_1\}\oplus
V_{u_1}(0)$, we have a unit vector $x\in\{u_1\}\oplus V_{u_1}(0)$
which is orthogonal to both $u_1$ and $v_1$. As $\{v_1\}\oplus
V_{v_1}(0) =\{x\}\oplus V_{x}(0)=\{u_1\}\oplus V_{u_1}(0)$, we get a
contradiction.

(ii)\ \ From (i) we have that $\{v_1\}\oplus
V_{v_1}(0)=\{u_1\}\oplus V_{u_1}(0)$. As $u_1$ and $\tilde{u}_1$ are
orthogonal, this implies that $\tilde{u}_1\in V_{u_1}(0)$.
Consequently, we have $L(u_1, \tilde{u}_1)=0$.
\end{proof}

\begin{lemma}\label{lm:4.13}
In the decomposition \eqref{eqn:4.28}, if we pick a unit vector
$u_2\in V_{v_2}(0)$, then there exists a unique vector $u_1\in
\{v_1\}\oplus V_{v_1}(0)$ such that $L(u_1, v_2) =L(v_1, u_2)$.
Moreover, $u_1$ is a unit vector in $V_{v_1}(0)$ and $L(v_1,
v_2)=-L(u_1, u_2)$.
\end{lemma}
\begin{proof}
Let $u_1^l,\ldots,u_{p_l}^l$ be an orthonormal basis of
$V_{v_l}(0),\,1\leq l\leq k_0$, such that $u_1^2=u_2$. Then
$$
\{v_1,\ldots,v_{k_0},u_1^1,\ldots,u_{p_1}^1,\ldots,
u_1^{k_0},\ldots,u_{p_{k_0}}^{k_0}\}=
:\{\tilde{u}_i\}_{1\leq i\leq m-1}
$$
forms an orthonormal basis of $\mathcal{D}_2$. Now we
use \eqref{eqn:4.12} with the vectors $v_2, u_2, v_1, v_2$.
As by Lemma \ref{lm:4.9} $L(v_2,u_2)=0$ and by our
decomposition $v_1\in V_{v_2}(\tau)$, we obtain
$$
\begin{aligned}
0&=K(L(v_2,u_2),L(v_1,v_2)) \\
&=\mu h(L(v_2,u_2),L(v_1,v_2))e_1+\sum_{i=1}^{m-1}
h(L(v_2,\tilde{u}_i),L(v_1,v_2))L(\tilde{u}_i,u_2) \\
&\ \ \  +\sum_{i=1}^{m-1}
h(L(u_2,\tilde{u}_i),L(v_1,v_2))L(v_2,\tilde{u}_i)\\
&=\tau L(v_1,u_2)+\sum_{i=1}^{m-1}
h(L(u_2,\tilde{u}_i),L(v_1,v_2))L(v_2,\tilde{u}_i).
\end{aligned}
$$
Let us take
$$
u_1=-\frac{1}{\tau}\sum_{i=1}^{m-1}
h(L(u_2,\tilde{u}_i),L(v_1,v_2))\tilde{u}_i.
$$
By Lemma \ref{lm:4.11}, we have
\begin{align}\label{eqn:4.29}
h(L(u_2,\tilde{u}_i),L(v_1,v_2))=0,\ \ \tilde{u}_i \notin
\{v_1\}\oplus V_{v_1}(0)\oplus\{v_2\}\oplus V_{v_2}(0).
\end{align}
Applying \eqref{eqn:4.7} and Lemma \ref{lm:4.9}, we get
\begin{align}\label{eqn:4.30}
h(L(u_2,\tilde{u}_i),L(v_1,v_2))=0,\ \ \tilde{u}_i \in \{v_2\}\oplus
V_{v_2}(0).
\end{align}
Moreover, note that $v_2\in V_{v_1}(\tau)$, thus we have
\begin{align}\label{eqn:4.31}
h(L(u_2,v_1),L(v_1,v_2))=0.
\end{align}
It follows from \eqref{eqn:4.29}, \eqref{eqn:4.30} and
\eqref{eqn:4.31} that $u_1\in V_{v_1}(0)$.

In order to prove the uniqueness of $u_1\in \{v_1\}\oplus V_{v_1}(0)$,
suppose that $\tilde{u}_1\in \{v_1\}\oplus V_{v_1}(0)$ such
that $L(\tilde{u}_1, v_2)=L(v_1, u_2)$, then we have
$L(u_1-\tilde{u}_1, v_2)=0$. It follows from the Lemma \ref{lm:4.9}
that $u_1-\tilde{u}_1\in V_{v_2}(0)$. On the other hand,
we also have $u_1-\tilde{u}_1\in \{v_1\}\oplus V_{v_1}(0)$; so we
must have $u_1=\tilde{u}_1$.

From the following fact
$$
V_{v_1}(0)\subset V_{v_2}(\tau),\ \ V_{v_2}(0)\subset V_{v_1}(\tau)
$$
we have $h(u_1,u_1)\tau=h(L(u_1, v_2),L(u_1, v_2)) =h(L(v_1,
u_2),L(v_1, u_2))=\tau$.
Hence, $u_1$ is a unit vector.

In order to prove the fact that $L(u_1, v_2)=L(v_1, u_2)$ and
$L(v_1, v_2)=-L(u_1, u_2)$ are equivalent, we use \eqref{eqn:4.5}
and the Cauchy-Schwarz inequality. In fact, if we first suppose
$L(u_1, v_2)=L(v_1, u_2)$, then applying \eqref{eqn:4.5} we get
$$
h(L(v_1, v_2),-L(u_1, u_2)) =h(L(v_1, u_2),L(v_2, u_1))=h(L(v_2,
u_1),L(v_2, u_1))=\tau.
$$
On the other hand, Lemma \ref{lm:4.12} implies that $v_1,u_1\in
V_{v_2}(\tau)=V_{u_2}(\tau)$ and thus
$$
h(L(v_1, v_2),L(v_1,
v_2)) =h(L(u_1, u_2),L(u_1, u_2))=\tau.
$$
Then, by Cauchy-Schwarz inequality we immediately have $L(v_1,
v_2)=-L(u_1, u_2)$.

The converse can be proved in a similar way.
\end{proof}

\vskip 2mm

To state the next lemma, we denote $V_l= \{v_l\}\oplus V_{v_l}(0)$
in the decomposition \eqref{eqn:4.28} for each $1\leq l\leq k_0$.
Then we have

\begin{lemma}\label{lm:4.14}
With respect to the decomposition \eqref{eqn:4.28}, the following
hold.
\begin{enumerate}
\item[(1)] For any unit vector $a\in V_j$,
\begin{align}\label{eqn:4.32}
K(L(a,a),L(a,a))=\tfrac{1}{2}\lambda_1\mu\eta e_1
+\eta(\mu+\lambda_1)L(a,a).
\end{align}
\item[(2)] For $j\neq l$ and any unit vector $a\in V_j,\ b\in V_l$,
\begin{align}\label{eqn:4.33}
K(L(a,a),L(a,b))=\tfrac{1}{2}\eta(\mu+\lambda_1)L(a,b),
\end{align}
\begin{align}\label{eqn:4.34}
K(L(a,a),L(b,b))=\tfrac{1}{2}\eta\mu^2e_1
+\eta\mu(L(a,a)+L(b,b)),
\end{align}
\begin{align}\label{eqn:4.35}
K(L(a,b),L(a,b))=\mu \tau e_1
+\tau(L(a,a)+L(b,b)).
\end{align}
\item[(3)] For distinct $j,\,l,\,q,\,s$ and any unit vector $a\in V_j,\ b,\,b'\in V_l,\
c\in V_q,\ d\in V_s$, where $b$ and $b'$ are orthogonal, the
following relations hold
\begin{align}\label{eqn:4.36}
K(L(a,b),L(a,c))=\tau L(b,c),
\end{align}
\begin{align}\label{eqn:4.37}
K(L(a,a),L(b,c))=\eta\mu L(b,c),
\end{align}
\begin{align}\label{eqn:4.38}
K(L(a,b),L(a,b'))=0,
\end{align}
\begin{align}\label{eqn:4.39}
K(L(a,b),L(c,d))=0.
\end{align}
\item[(4)]For distinct $j,\,l,\,q$ and orthogonal unit vector $a_1,\,a_2\in V_j$
and unit vectors $b\in V_l,\,c\in V_q$, it holds
\begin{align}\label{eqn:4.40}
K(L(a_1,b),L(a_2,c))=\tau L(b,c'),
\end{align}
where $c'\in V_q$ is the unique unit vector
satisfying $L(a_1,c')=L(a_2,c)$.
\end{enumerate}
\end{lemma}
\begin{proof}
Take an orthonormal basis of $\mathcal{D}_2$ such that it consists
of the orthonormal basis of all $V_l,\ 1\leq l\leq k_0$, the
assertions are direct consequences of Lemma \ref{lm:4.6}. Take for
example, from the fact $h(L(a,b),L(a,c))=h(P_ab,c)=0$, eq.
\eqref{eqn:4.6} and Lemma \ref{lm:4.6}, we immediately get
\eqref{eqn:4.36}. As another example, from \eqref{eqn:4.36}, Lemmas
\ref{lm:4.12} and \ref{lm:4.13} we can get \eqref{eqn:4.40}.
\end{proof}

\begin{proposition}\label{pr:4.2}
In the decomposition \eqref{eqn:4.28}, if $k_0=1$, then $\dim\,({\rm
Im}\,L)=1$. If $k_0\geq 2$, then $\dim\, V_{v_1}(0)=\cdots=\dim\,
V_{v_{k_0}}(0)$ and the dimension which we denote by $\mathfrak{p}$
can only be equal to 0, 1, 3 or 7.
\end{proposition}
\begin{proof}
If $k_0=1$, from Lemmas \ref{lm:4.9} and \ref{lm:4.12} we see that $L(v_1, v_1)$ is a
basis of the image ${\rm Im}\,L$, so we have $\dim\,({\rm
Im}\,L)=1$. As a direct consequence of Lemma \ref{lm:4.13}, for any
$j\neq l$, we can define a one-to-one linear map from $V_{v_j}(0)$
to $V_{v_l}(0)$, which preserves the length of vectors. Hence
$V_{v_j}(0)$ and $V_{v_l}(0)$ are isomorphic and have the same
dimension which we denote by $\mathfrak{p}$. To make the following
discussion meaningful, we now assume $\mathfrak{p}\geq1$.

Let $\{v_l,u_1^l,\ldots,u_\mathfrak{p}^l\}$ be an orthonormal basis
of $V_l$. For each $j=1,\ldots,\mathfrak{p}$, Lemmas \ref{lm:4.12}
and \ref{lm:4.13} show that we can define a linear map
$\mathfrak{T}_j:V_1 \rightarrow V_1$ such that, for any unit vector
$v$, the image $\mathfrak{T}_j(v)$ satisfies
\begin{align}\label{eqn:4.41}
L(v,u_j^2)=L(v_2,\mathfrak{T}_j(v)).
\end{align}

The linear map $\mathfrak{T}_j:\ V_1 \rightarrow V_1$ has the
following properties:
\begin{enumerate}
\item[(P1)] For any $v\in V_1$, $h(\mathfrak{T}_j(v),\mathfrak{T}_j(v))=h(v,v)$,
i.e., $\mathfrak{T}_j$ preserves the length of vectors.
\item[(P2)] For all $v\in V_1$, we have $\mathfrak{T}_j(v)\perp v$.
\item[(P3)] $\mathfrak{T}_j^2=-{\rm id}$.
\item[(P4)] For any $j\neq l$,
we have $h(\mathfrak{T}_j(v),\mathfrak{T}_l(v))=0$ for all $v\in
V_1$.

\end{enumerate}

(P1) and (P2) can be easily seen from Lemma \ref{lm:4.13} and the
definition of $\mathfrak{T}_j(v)$. We now verify (P3) and (P4). For
any unit vector $v\in V_1$, we have
\begin{align}\label{eqn:4.42}
L(v_2,\mathfrak{T}_j^2(v))=L(u_j^2,\mathfrak{T}_j(v)).
\end{align}
Using the fact $\{\mathfrak{T}_j(v)\}\oplus
V_{\mathfrak{T}_j(v)}(0)=V_1$ and $u_j^2\in V_{v_2}(0)\subset
V_{\mathfrak{T}_j(v)}(\tau)$, we have
$$
\begin{aligned}
h(L(u_j^2,\mathfrak{T}_j(v)),L(u_j^2,\mathfrak{T}_j(v)))
&=h(L(v_2,\mathfrak{T}_j(v)),L(v_2,\mathfrak{T}_j(v)))\\
&=h(L(v,v_2),L(v,v_2))=\tau.
\end{aligned}
$$
Since $v,\,\mathfrak{T}_j(v),\,v_2,\,u_j^2$ are orthonormal vectors,
by \eqref{eqn:4.10}, \eqref{eqn:4.41} and $L(v_2,u_j^2)=0$, we see
that
$$
\begin{aligned}
0&=h(L(v,v_2),L(u_j^2,\mathfrak{T}_j(v)))
+h(L(v,\mathfrak{T}_j(v)),L(v_2,u_j^2))\\
&+h(L(v,u_j^2),L(\mathfrak{T}_j(v),v_2))\\
&=h(L(v,v_2),L(u_j^2,\ \mathfrak{T}_j(v)))
+h(L(v_2,\mathfrak{T}_j(v)),L(v_2,\mathfrak{T}_j(v))).
\end{aligned}
$$
Applying the Cauchy-Schwarz inequality we deduce
\begin{align}\label{eqn:4.43}
L(u_j^2,\mathfrak{T}_j(v))=-L(v,v_2).
\end{align}

Combining \eqref{eqn:4.42} and \eqref{eqn:4.43},
we get $L(v_2,\mathfrak{T}_j^2(v)+v)=0$, which implies that
$\mathfrak{T}_j^2(v)+v\in V_{v_2}(0)$. As $\mathfrak{T}_j^2(v)+
v\in V_1\subset V_{v_2}(\tau)$, it follows that
$\mathfrak{T}_j^2(v)=-v$ for a
unit vector $v$ and then by linearity for all $v\in V_1$,
as claimed by (P3).

To verify (P4), we note that, if $j\neq l$, and
$\mathfrak{T}_j(v),\,\mathfrak{T}_l(v)\in V_{v}(0)$, then by
definition
$$
\begin{aligned}
L(v_2,\mathfrak{T}_j(v))=L(v,u_j^2)
\perp L(v,u_l^2)=L(v_2,\mathfrak{T}_l(v)).
\end{aligned}
$$
If we assume $\mathfrak{T}_l(v)=a\mathfrak{T}_j(v)+x$,
where $x\perp \mathfrak{T}_j(v)$, then
$$
\begin{aligned}
0&=h(L(v_2,\mathfrak{T}_j(v)),L(v_2,\mathfrak{T}_l(v)))\\
&=h(L(v_2,\mathfrak{T}_j(v)),aL(v_2,\mathfrak{T}_j(v))+L(v_2,x))\\
&=a\tau.
\end{aligned}
$$
Hence, $a=0$ and $\mathfrak{T}_l(v)\perp \mathfrak{T}_j(v)$.

We look at the unit hypersphere $S^\mathfrak{p}(1)\subset V_1$, the
above properties (P1)\,--\,(P4) show that at $v\in
S^\mathfrak{p}(1)$ one has
$$
T_vS^\mathfrak{p}(1)={\rm
span}\,\{\mathfrak{T}_1(v),\ldots,\mathfrak{T}_\mathfrak{p}(v)\}.
$$

Hence, by the properties (P1)\,--\,(P4), the
$\mathfrak{p}$-dimensional sphere $\mathbb{S}^\mathfrak{p}(1)$ is
parallelizable. Then, according to R. Bott and J. Milnor \cite{BM}
and M. Kervaire \cite{Ke}, the dimension $\mathfrak{p}$ can only be
equal to $1$, $3$ or $7$.
\end{proof}

%%%%%%%%%%%%%%%%%%%%%%%%%%%%%%%%%%%%%%%%%%%%%%%%%%%%%%%%%%%%%%%%
\numberwithin{equation}{section}

\section{The exceptional case $\mathfrak{B}$}\label{sect:5}

In this section, we shall study an $n$-dimensional ($n\geq2$)
locally strongly convex centroaffine hypersurface $M^n$ which has
parallel cubic form, such that Case $\mathfrak{B}$ occurs. The main
result of this section is the following theorem.

\begin{theorem}\label{th:5.1}
Let $x: M^n\to\mathbb{R}^{n+1}$ ($n\geq2$) be a locally strongly
convex centroaffine hypersurface which has parallel cubic form. If
Case $\mathfrak{B}$ occurs, then $M^n$ is locally centroaffinely
equivalent to the hypersurface:
\begin{equation}\label{eqn:5.2}
x_{n+1}=\tfrac{1}{2x_{1}}\sum_{k=2}^{n}x_{k}^{2} + x_{1}\ln x_{1}.
\end{equation}
\end{theorem}

To begin with, we prove the following lemma.

\begin{lemma}\label{lm:5.1}
In Case $\mathfrak{B}$, there exists an orthonormal basis
$\{e_1,\ldots,e_n\}$ of $T_pM$ such that the difference tensor $K$
satisfies
\begin{equation}\label{eqn:5.2}
K_{e_1}e_1=2e_1,\ K_{e_1}e_i=e_i,\ K_{e_i}e_j=\delta_{ij} e_1,\ \ i,
j=2,\ldots,n.
\end{equation}
\end{lemma}
\begin{proof}
Let $\{e_1,\ldots,e_n\}$ be the orthonormal basis determined in
Lemma \ref{lm:4.1}. By assumption, $\lambda_{1}^{2}-4\varepsilon=0$,
we have
\begin{equation}\label{eqn:5.3}
\varepsilon=1,\ \ \lambda_{1}=2.
\end{equation}

Similar to the proof of \eqref{eqn:4.2}, we now have
\begin{equation}\label{eqn:5.4}
h(K_{e_i}e_j,e_k)=0,\ \ 2\leq i, j, k\leq n.
\end{equation}

From these results we easily get the assertion of Lemma
\ref{lm:5.1}.
\end{proof}

Next, as an extension of Lemma \ref{lm:5.1} we can prove the
following lemma.

\begin{lemma}\label{lm:5.2}
If Case $\mathfrak{B}$ occurs, then around $p$ there exists a local
orthonormal basis $\{E_{1},\ldots,E_{n}\}$ such that
$\hat{\nabla}_{X}E_1=0$ for all $X \in TM^n$, and
\begin{equation}\label{eqn:5.5}
K_{E_1}E_1=2 E_1,\ K_{E_1}E_i= E_i,\ K_{E_i}E_j=\delta_{ij} E_1,\ \
i,j=2,\ldots,n.
\end{equation}
Moreover, $(M^n,h)$ is locally isometric to the Euclidean space
$\mathbb{R}^n$.
\end{lemma}
\begin{proof}
Let $\{e_{1},\ldots,e_{n}\}$ be the orthonormal basis of $T_pM$,
given by Lemma \ref{lm:5.1}. By parallel translation of
$\{e_i\}_{i=1}^n$ along geodesics through $p$, we can obtain an
$h$-orthonormal basis, denoted by $\{E_{1},\ldots,E_{n}\}$, in a
normal neighbourhood around $p$. Since $\hat{\nabla}K=0$, the
difference tensor $K$ takes the form of \eqref{eqn:5.5}.

It follows from \eqref{eqn:2.3}, \eqref{eqn:5.3} and \eqref{eqn:5.5}
that $(M^n,h)$ satisfies $\hat{R}(E_i, E_j)E_j=0$ for any $i,j$,
i.e., $(M^n,h)$ is flat and it is locally isometric to the Euclidean
space $\mathbb{R}^n$.

To show that $\hat{\nabla}_{X}E_1=0$ for any $X\in TM^n$, we denote
$\hat{\nabla}_{E_j}E_i=\sum_k\Gamma_{ij}^kE_k$, where
$\Gamma_{ij}^k=-\Gamma_{kj}^i$, $1\leq i,j,k\leq n$. By using
$\hat\nabla K=0$ and \eqref{eqn:5.5}, straightforward calculations
of the equations
$$
0=(\hat{\nabla}_{E_i}K)(E_i, E_i) =(\hat{\nabla}_{E_1}K)(E_i, E_i),\
i\not=1
$$
give that $\Gamma_{ij}^1=0$ for $1\leq i,j\leq n$.
It follows that
\begin{equation}\label{eqn:5.6}
\hat{\nabla}_{E_i}E_{1}=0,\ \ 1\le i\le n.
\end{equation}
This completes the proof of Lemma \ref{lm:5.2}.
\end{proof}

Now we will prove Theorem \ref{th:5.1}.

\begin{proof}[Proof of Theorem \ref{th:5.1}.]
As proved in Lemma \ref{lm:5.2}, $\hat{\nabla}_{X}E_1=0$ and
$(M,h)$ is locally isometric to $\mathbb{R}^{n}$, we may choose
local coordinates $(u_1,u_2,\ldots,u_n)$ on $M^n$ such that the
metric $h$ has the following expression:
\begin{equation}\label{eqn:5.7}
h=du_1^2+du_2^2+du_3^2+\cdots+du_n^2,
\end{equation}
and that $\tfrac{\partial}{\partial u_1}=E_1$. It follows from
\eqref{eqn:5.7} that
\begin{equation}\label{eqn:5.8}
\hat{\nabla}_{\partial u_i}\partial u_j=0,\ \ 1\leq i,j\leq n,
\end{equation}
where, and also later on, we use the notations $\partial u_k
=\tfrac{\partial\ }{\partial u_k},\ k=1,\ldots,n$.

By using \eqref{eqn:5.5}, we get that
\begin{equation}\label{eqn:5.9}
K_{\partial u_1}X=X,\ \ K_X Y= h(X,Y)\partial u_1,\ \ X,Y\in
\{\partial u_1\}^\perp.
\end{equation}
By using \eqref{eqn:5.5}, \eqref{eqn:5.7} and \eqref{eqn:5.9}, we get that
\begin{equation}\label{eqn:5.10}
\left\{
\begin{aligned}
&K_{\partial u_1}\partial u_1=2\partial u_1,\ \ K_{\partial
u_1}\partial u_k=
\partial u_k,\ \ 2\leq k\leq n,\\
& K_{\partial u_k}\partial u_j=\delta_{kj}\partial u_1,\ \ 2\leq
j,k\leq n.
\end{aligned}
\right.
\end{equation}

Write $x=x(u_1,\ldots,u_n)\in\mathbb{R}^{n+1}$. From
\eqref{eqn:5.10}, \eqref{eqn:5.8}, and using \eqref{eqn:1.1} with
the fact $\varepsilon=1$, we have
\begin{equation}\label{eqn:5.11}
x_{u_1u_1}=2x_{u_1}-x,
\end{equation}
\begin{equation}\label{eqn:5.12}
x_{u_1u_k}=x_{u_k},\ \ 2\leq k\leq n,
\end{equation}
\begin{equation}\label{eqn:5.13}
x_{u_ku_k}=x_{u_1}-x,\ \ 2\leq k\leq n,
\end{equation}
\begin{align}\label{eqn:5.14}
x_{u_ku_j}=0,\ \ 2\leq j,k\leq n\ {\rm and}\ j\neq k.
\end{align}

First of all, we can solve \eqref{eqn:5.11} to obtain that
\begin{align}\label{eqn:5.15}
x=&P_1(u_2,\ldots,u_n)e^{u_1}+P_2(u_2,\ldots,u_n) u_1e^{u_1},
\end{align}
where $P_1(u_2,\ldots,u_n)$ and $P_2 (u_2,\ldots,u_n)$ are
$\mathbb{R}^{n+1}$-valued functions.

Inserting \eqref{eqn:5.15} into \eqref{eqn:5.12}, we obtain
$\tfrac{\partial P_2}{\partial u_k}=0$, $2\le k\le n$, which shows
that $P_2(u_2,\ldots,u_n)$ is a constant vector denoted by $A_1$.
Hence, we have
\begin{align}\label{eqn:5.16}
x=P_1(u_2,\ldots,u_n)e^{u_1} +A_1u_1e^{u_1}.
\end{align}

Putting \eqref{eqn:5.16} into \eqref{eqn:5.13} for $k=2$, we
further obtain that
\begin{align}\label{eqn:5.17}
\tfrac{\partial^2P_1}{\partial u_2\partial u_2}=A_1.
\end{align}
Thus, we can write
\begin{align}\label{eqn:5.18}
x=\Big(\tfrac{1}{2}u_2^2A_1
+P_3(u_3,\ldots,u_n)u_2+P_4(u_3,\ldots,u_n)\Big)e^{u_1}
+u_1e^{u_1}A_1.
\end{align}

From \eqref{eqn:5.14} and \eqref{eqn:5.18}, we can derive that
$P_3(u_3,\ldots,u_n)$ is a constant vector denoted by $A_2$. Hence,
we have
$$
x=\left(\tfrac{1}{2}u_2^2A_1
+u_2A_2+P_4(u_3,\ldots,u_n)\right)e^{u_1} +u_1e^{u_1}A_1.
$$

If we carry out such procedure by induction for other $u_k$ with
$k\ge3$, we can finally obtain constant vectors
$\{A_1,A_2,\ldots,A_{n+1}\}$ such that $x(u_1,\ldots,u_n)$ has the
following expression:
\begin{align}\label{eqn:5.19}
x=\Big(\tfrac{1}{2}\sum_{k=2}^n u_k^2+u_1\Big)e^{u_1}A_1
+\sum_{k=2}^nu_ke^{u_1}A_k+e^{u_1}A_{n+1}.
\end{align}

The nondegeneracy of $x$ implies that it lies linearly full in
$\mathbb{R}^{n+1}$ and thus $A_{1}$, $\ldots$, $A_{n+1}$ are
linearly independent vectors. Thus, up to a centroaffine
transformation, $x$ can be written as
$$
x=\Big(e^{u_1},u_2e^{u_1},\ldots,u_ne^{u_1},
\big(\tfrac{1}{2}\sum_{k=2}^n u_k^2+u_1\big)e^{u_1}\Big),
$$
which is easily seen to be locally centroaffinely equivalent to the
hypersurface given in Theorem \ref{th:5.1}.

We have completed the proof of Theorem \ref{th:5.1}.
\end{proof}

%%%%%%%%%%%%%%%%%%%%%%%%%%%%%%%%%%%%%%%%%%%%%%%%%%%%%%%%%%%%%%%%%%%
\numberwithin{equation}{section}

\section{Centroaffine surfaces in
$\mathbb{R}^3$ with $\tilde\nabla C=0$}\label{sect:6}

Although Theorem \ref{th:1.1} gives a complete classification of
locally strongly convex centroaffine hypersurfaces in
$\mathbb{R}^{n+1}$ with parallel cubic form, its statement involving
the Calabi product constructions actually makes use of the induction
procedure. Therefore, in order to guarantee the validity of such
induction procedure, we need first consider the lowest dimension case
(i.e. $n=2$). This problem will be settled by
the following theorem.

\begin{theorem}\label{th:6.1}
Let $x:M^2\to\mathbb{R}^3$ be a locally strongly convex centroaffine
surface which has parallel cubic form. Then $x$ is locally
centroaffinely equivalent to one of the following hypersurfaces:
\begin{enumerate}
\item[(i)] quadrics (C=0);

\vskip 1mm

\item[(ii)] $x_{1}^{\alpha_{1}}x_{2}^{\alpha_{2}}
x_{3}^{\alpha_{3}}=1$, where $\{\alpha_i\}$ are real numbers which
satisfy
$$
\alpha_{i}>0,\ i=1,2,3;\ {\rm or}\ \alpha_{1}<0,\ \alpha_{2},
\alpha_{3}>0,\ \alpha_{1}+\alpha_{2}+\alpha_{3}<0;
$$

\vskip 1mm

\item[(iii)] $x_{1}^{\alpha_{1}}
(x_{2}^2+x_{3}^2)^{\alpha_{2}} \exp (\alpha
_{3}\arctan\tfrac{x_2}{x_3})=1,\ \ \alpha_{1}<0,\
\alpha_{1}+2\alpha_{2}>0$;

\vskip 1mm

\item[(iv)] $x_{3}= x_{1}(\ln x_{1}
-\alpha_{2}\ln x_{2}),\ \ 0<\alpha_{2}<1;$

\vskip 1mm

\item[(v)] $x_{3}=\tfrac{1}{2x_{1}}x_{2}^{2}
+x_{1}\ln x_{1}$,
\end{enumerate}
where $\alpha_1,\alpha_2, \alpha_3$ are constants and $(x_1,x_2,x_3)$
is the coordinate of $\mathbb{R}^3$.
\end{theorem}
\begin{remark}\label{rm:6.1}
Centroaffine surfaces with parallel cubic form have been studied in
\cite{LW2}, where the authors made use of Theorem 1.3 in \cite{LW}.
Comparing our theorem with the result in \cite{LW2}, one can see
that the surface (v) of Theorem \ref{th:6.1} is missing in
\cite{LW2}. This appearance is because in \cite{LW} the authors only
obtained the classification of canonical centroaffine hypersurfaces
for $N(h)\leq1$, hence in \cite{LW2} the conclusion for the case
$N(h)=2$ is unfortunately not correct stated. Here, the fact that
the surface (v) corresponds to $n=2,\ v=3$ and $N(h)=2$ in corollary
\ref{cr:1.1} should be emphasized.
\end{remark}

In order to prove Theorem \ref{th:6.1}, we first notice that, for
$n=2$, it follows from Theorem \ref{th:5.1} that in Case
$\mathfrak{B}$ the surface $M^2$ is centroaffinely equivalent to the
surface (v). Thus, taking into consideration of Theorem
\ref{th:4.2}, we see that what we need to consider is Case
$\mathfrak{C}_1$ with $n=2$ in a more explicit way, rather than like
the sketchy statement of Theorem \ref{th:4.1}.

\vskip 1mm

To begin with, we state the following lemma which is a direct
consequence of Lemma \ref{lm:4.1}.

\begin{lemma}[cf. Lemma \ref{lm:4.1}]\label{lm:6.1}
If Case $\mathfrak{C}_1$ occurs, then there exists an orthonormal
basis $\{e_{1},e_{2}\}$ of $T_pM^2$ such that the difference tensor
$K$ takes the following form:
\begin{align*}
K_{e_{1}}e_{1}=\lambda _{1}e_{1},\ \ K_{e_{1}}e_{2}=\mu
e_{2},\ \ K_{e_{2}}e_{2}=\mu e_{1}+a_{1}e_{2},\\
\varepsilon-\lambda_{1}\mu+\mu^{2}=0,\ \lambda_{1}>0,\
\lambda_{1}^{2}-4\varepsilon>0,\ \lambda_{1}> 2\mu.
\end{align*}
\end{lemma}

To prove Theorem \ref{th:6.1}, we also need the following lemma.
\begin{lemma}\label{lm:6.2}
If Case $\mathfrak{C}_1$ occurs, then there exists a local
orthonormal basis $\{E_{1},E_{2}\}$ around $p$, such that the
difference tensor takes the following form:
\begin{align}\label{eqn:6.1}
K_{E_{1}}&E_{1}=\lambda _{1}E_{1},\ K_{E_{1}}E_{2}=\mu
E_{2},\ K_{E_{2}}E_{2}=\mu E_{1}+a_{1}E_{2},\\
&\varepsilon-\lambda_{1}\mu+\mu^{2}=0,\ \lambda_{1}>0,\
\lambda_{1}^{2}-4\varepsilon>0,\ \lambda_{1}> 2\mu,\nonumber
\end{align}
where $\lambda_{1}, \mu, a_1$ are constant numbers and
$\hat{\nabla}_{E_i}E_j=0,\ i,j=1,2$. Moreover, $(M^2,h)$ is locally
isometric to the Euclidean space $\mathbb{R}^2$.
\end{lemma}
\begin{proof}
Let $\{e_{1},e_{2}\}$ be the orthonormal basis of $T_pM^2$ given by
Lemma \ref{lm:6.1}. By parallel translation of $\{e_1,e_2\}$ along
geodesics (with respect to $\hat{\nabla}$) through $p$, we can
obtain an $h$-orthonormal basis, denoted by $\{E_{1},E_2\}$, in a
normal neighbourhood around $p$ such that, thanks to
$\hat{\nabla}K=0$, the difference tensor $K$ takes the form stated
in \eqref{eqn:6.1}.

First, from the calculation
$$
0=(\hat{\nabla}_{E_i}K)(E_1, E_1)=\lambda_1\hat{\nabla}_{E_i}
E_1-2K(\hat{\nabla}_{E_i}E_1, E_1),\ \ i=1, 2,
$$
and noting that $\hat{\nabla}_{E_i}E_{1}$ is $h$-orthogonal to
$E_{1}$, we have $\hat{\nabla}_{E_i}E_1=0,\ i=1,2$.

Next, by computation of $0=h((\hat{\nabla}_{E_2}K)(E_1, E_2),E_1)$
we obtain that
\begin{equation}\label{eqn:6.2}
h(\hat{\nabla}_{E_2}E_2,E_1)=0.
\end{equation}
This, together with $h(\hat{\nabla}_{E_i}E_2,E_2)=0$ and
$h(\hat{\nabla}_{E_1}E_2,E_1)=-h(\hat{\nabla}_{E_1}E_1,E_2)=0$,
we will obtain
\begin{equation}\label{eqn:6.3}
\hat{\nabla}_{E_i}E_j=0,\ i,j=1,2.
\end{equation}

It follows that $\hat{R}(E_i, E_j)E_k=0$ and $(M^2,h)$ is locally
isometric to the Euclidean space $\mathbb{R}^2$.
\end{proof}

\begin{proof}[Proof of Theorem \ref{th:6.1}]

According to Lemma \ref{lm:6.2}, we can choose local coordinates
$(u_1,u_2)$ for $M^2$ such that the centroaffine metric $h$ has the
following expression:
\begin{equation}\label{eqn:6.4}
h=du_1^2+du_2^2,
\end{equation}
and $E_i=\tfrac{\partial}{\partial u_i}$ for $i=1,2$. It follows
from \eqref{eqn:6.4} that
\begin{equation}\label{eqn:6.5}
\hat{\nabla}_{\partial u_i}\partial u_j=0,\ 1\leq i,j\leq 2.
\end{equation}

For $x=x(u_1,u_2)\in\mathbb{R}^3$, using \eqref{eqn:6.1},
\eqref{eqn:6.4}, \eqref{eqn:6.5} and \eqref{eqn:1.1} we can
obtain:
\begin{align}
x_{u_1u_1}&=\lambda_1x_{u_1}-\varepsilon x,\label{eqn:6.6}\\
x_{u_1u_2}&=\mu x_{u_2},\label{eqn:6.7}\\
x_{u_2u_2}&=\mu x_{u_1}+a_1x_{u_2}-\varepsilon x.\label{eqn:6.8}
\end{align}

We first solve the equation \eqref{eqn:6.6} to obtain that
\begin{equation}\label{eqn:6.9}
x=P_1(u_2)\exp\{(\lambda_1-\mu)u_1\}+P_2(u_2)\exp(\mu u_1),
\end{equation}
where $P_1(u_2)$ and $P_2 (u_2)$ are $\mathbb{R}^3$-valued
functions.

Inserting \eqref{eqn:6.9} into \eqref{eqn:6.7}, we obtain
$\tfrac{\partial P_1}{\partial u_2}=0$, showing that $P_1(u_2)$ is a
constant vector, denoted by $A_1$. Hence, we have
\begin{align}\label{eqn:6.10}
x=\exp\{(\lambda_1-\mu)u_1\}A_1 +P_2(u_2)\exp(\mu u_1).
\end{align}

Combining \eqref{eqn:6.10} and \eqref{eqn:6.8}, we get
\begin{align}\label{eqn:6.11}
\tfrac{{\rm d}^2P_2}{{\rm d}u_2{\rm d}u_2}=a_1 \tfrac{{\rm
d}P_2}{{\rm d}u_2}+(\mu^2-\varepsilon)P_2.
\end{align}

To solve \eqref{eqn:6.11}, we will consider the following three
cases, separately:
\begin{enumerate}
\item[(a)] $a_1^2+4(\mu^2-\varepsilon)>0$.\vskip 1mm
\item[(b)] $a_1^2+4(\mu^2-\varepsilon)<0$.\vskip 1mm
\item[(c)] $a_1^2+4(\mu^2-\varepsilon)=0$.
\end{enumerate}

(a) In this case, the solution of \eqref{eqn:6.11} is
$$
\begin{aligned}
P_2=\exp\Big\{\tfrac12&\Big(a_1+\sqrt{a_1^2+4(\mu^2-\varepsilon)}\,\Big)u_2
\Big\}A_2\\
&+\exp\Big\{\tfrac{1}{2}\Big(a_1-\sqrt{a_1^2+4(\mu^2-\varepsilon)}\Big)u_2
\Big\}A_3,
\end{aligned}
$$
where $A_2,\,A_3$ are constant vectors.

It follows that, up to a centroaffine transformation, $x$ can be
written as
\begin{align}\label{eqn:6.12}
x=\Big(\exp\big\{(\lambda_1-\mu)u_1&\big\},
\exp\big\{\tfrac{1}{2}\big(a_1+\sqrt{a_1^2+4(\mu^2-\varepsilon)}\,\big)u_2
+\mu u_1\big\},\nonumber\\
&\ \ \
\exp\big\{\tfrac{1}{2}\big(a_1-\sqrt{a_1^2+4(\mu^2-\varepsilon)}\,\big)u_2
+\mu u_1\big\}\Big),
\end{align}
which, due to its locally strongly convexity, is easily seen locally
on the hypersurface (ii) of Theorem \ref{th:6.1}.

\vskip 3mm

(b) In this case, we have $\varepsilon=1$. The solution of
\eqref{eqn:6.11} is given by
$$
\begin{aligned}
P_2=&\cos\Big(\tfrac{1}{2}\sqrt{-a_1^2
-4(\mu^2-1)}u_2\Big)\exp(\tfrac{1}{2}a_1u_2)A_2\\
&+\sin\Big(\tfrac{1}{2}\sqrt{-a_1^2
-4(\mu^2-1)}u_2\Big)\exp(\tfrac{1}{2}a_1u_2)A_3,
\end{aligned}
$$
where $A_2,A_3$ are constant vectors.

It follows that, up to a centroaffine transformation, $x$ can be
written as
\begin{align}\label{eqn:6.13}
x=\Big(\exp\big\{(\lambda_1-\mu)&u_1\big\},
\sin\big(\tfrac{1}{2}\sqrt{-a_1^2
-4(\mu^2-1)}u_2\big)\exp(\tfrac{1}{2}a_1u_2+\mu u_1),\nonumber\\
&\cos\big(\tfrac{1}{2}\sqrt{-a_1^2-4(\mu^2-1)}u_2\big)
\exp(\tfrac{1}{2}a_1u_2 +\mu u_1)\Big),
\end{align}
which, due to its locally strongly convexity, is locally on the
hypersurface (iii) of Theorem \ref{th:6.1}.

\vskip 3mm

(c) In this case, from the fact that $a_1^2+4(\mu^2-\varepsilon)=0$
and Lemma \ref{lm:6.2}, we have
\begin{align}\label{eqn:6.14}
a_1\neq 0,\ \ \varepsilon=1.
\end{align}
The solution of \eqref{eqn:6.11} is given by
$$
\begin{aligned}
P_2=&\exp(\tfrac{1}{2}a_1u_2)A_2+u_2 \exp(\tfrac{1}{2}a_1u_2)A_3,
\end{aligned}
$$
where $A_2,A_3$ are constant vectors.

It follows that, up to a centroaffine transformation, $x$ can be
written as
\begin{align}\label{eqn:6.15}
x=&\Big(\exp(\tfrac{1}{2}a_1u_2+\mu u_1),
\exp\big\{(\lambda_1-\mu)u_1\big\},\tfrac{1}{2}a_1u_2\exp(\tfrac{1}{2}a_1u_2
+\mu u_1)\Big),
\end{align}
which, according to \eqref{eqn:6.14}, \eqref{eqn:6.15} and due to
its locally strongly convexity, is locally on the hypersurface (iv)
of Theorem \ref{th:6.1}.

We have completed the proof of Theorem \ref{th:6.1}.
\end{proof}

%%%%%%%%%%%%%%%%%%%%%%%%%%%%%%%%%%%%%%%%%%%%%%%%%%%%%%%%%%%%%%%%
\numberwithin{equation}{section}

\section{Case
$\{\mathfrak{C}_m\}_{2\le m\le n-1}$ with $k_0=1$}\label{sect:7}

In this section, we consider Case $\mathfrak{C}_m$ ($2\le m\le n-1$)
with the condition that in the decomposition \eqref{eqn:4.28},
$k_0=1$. We will prove the following theorem.
\begin{theorem}\label{th:7.1}
Let $M^n$ be a locally strongly convex centroaffine hypersurface in
$\mathbb{R}^{n+1}$ which has parallel and non-vanishing cubic form.
If $\mathfrak{C}_m$ with $2\le m\le n-1$ occurs and the integer
$k_0$, as defined in subsection \ref{sect:4.5}, satisfies $k_0=1$,
then $M^n$ can be decomposed as the Calabi product of two locally
strongly convex centroaffine hypersurfaces with parallel cubic form,
or the Calabi product of a locally strongly convex centroaffine
hypersurface with parallel cubic form and a point.
\end{theorem}

To prove Theorem \ref{th:7.1}, we first note that if $k_0=1$ then by
Proposition \ref{pr:4.2} we have $\dim\,({\rm Im}\,L)=1$. Moreover,
we can prove the following result.

\begin{lemma}\label{lm:7.1}
If $\dim\,({\rm Im}\,L)=1$, then there is a unit vector $w_1\in {\rm
Im}\,L\subset \mathcal{D}_3$ such that $L$ has the expression
\begin{equation}\label{eqn:7.1}
L(v_1,v_2)=\sqrt{\tfrac{1}{2}\lambda_1\eta }h(v_1,v_2)w_1,\ \
\forall\, v_1,v_2\in\mathcal{D}_2.
\end{equation}
\end{lemma}
\begin{proof}
The fact $\dim\,({\rm Im}\,L)=1$ implies that we have a unit vector
$\bar{w}\in {\rm Im}\,(L)\subset \mathcal{D}_3$ and a symmetric
bilinear form $\alpha$ over $\mathcal{D}_2$ such that
\begin{equation}\label{eqn:7.2}
L(v_1,v_2)=\alpha(v_1,v_2)\bar{w},\ \ \forall\,
v_1,v_2\in\mathcal{D}_2.
\end{equation}
We define $Q:\mathcal{D}_2\rightarrow \mathcal{D}_2$ by
$h(Qv_1,v_2):=\alpha(v_1, v_2)$. From Lemma \ref{lm:4.12} we have
\begin{equation}\label{eqn:7.3}
L(v_1,v_2)=0,\ {\rm if} \ h(v_1,v_2)=0.
\end{equation}
Now we see that $h(Qv_1,v_2)=0$ if $h(v_1,v_2)=0$. Hence,
$Qv=\sqrt{\tfrac{1}{2}\lambda_1\eta}\,\varepsilon(v)v$ for all $v\in
\mathcal{D}_2$ and $\varepsilon(v)=\pm1$. It follows that
\begin{equation}\label{eqn:7.4}
L(v_1,v_2)=\alpha(v_1,v_2)\bar{w}
=\sqrt{\tfrac{1}{2}\lambda_1\eta }\varepsilon(v_1)h(v_1,v_2)\bar{w}.
\end{equation}
This, together with the fact that both $L$ and $h$ are symmetric,
implies that, for any $v_1,v_2\in \mathcal{D}_2$,
$\varepsilon(v_1)=\varepsilon(v_2)$ holds, i.e., $\varepsilon(v)$ is
independent of $v$.

We finally get the assertion by putting $w_1:=
\varepsilon(v_1)\bar{w}$.
\end{proof}

In sequel of this section, we will fix the unit vector $w_1 \in
\mathcal{D}_3$ as in Lemma \ref{th:7.1}. Then, besides
$K_{e_1}w_1=\mu w_1$, the next three lemmas give all informations
about the difference tensor $K$.
\begin{lemma}\label{lm:7.2}
There exists an orthonormal basis $\{v_1,\ldots,v_{m-1}\}$
of $\mathcal{D}_2$ such that
\begin{align}\label{eqn:7.5}
&K(e_1,v_i)=\tfrac{1}{2}\lambda_1v_i,\ K(w_1,v_i)
=\sqrt{\tfrac{1}{2}\lambda_1\eta }v_i,\ \ 1\leq i\leq m-1,\\
&K(v_i,v_j)=\big(\tfrac{1}{2}\lambda_1e_1
+\sqrt{\tfrac{1}{2}\lambda_1\eta }w_1\big)\delta_{ij},\ \ 1\leq
i,j\leq m-1.
\end{align}
\end{lemma}
\begin{proof}
From Lemma \ref{lm:4.2}, we see that $K_{w_1}$ maps $\mathcal{D}_2$
to $\mathcal{D}_2$. Note that $K_{w_1}$ is self-adjoint, then
there exists an orthonormal basis $\{v_1,\ldots, v_{m-1}\}$ of
$\mathcal{D}_2$ such that $K_{w_1}v_i=\alpha_iv_i$ with
eigenvalues $\alpha _i$. As $v_i\in \mathcal{D}_2$,
we have $K_{e_1}v_i=\tfrac{1}{2}\lambda_1v_i$. By Lemma
\ref{lm:7.1} we get
\begin{equation*}
\alpha_i=h(K_{w_1}v_i,v_i)=h(L(v_i,v_i),w_1)=\sqrt{\tfrac{1}{2}\lambda_1\eta }.
\end{equation*}
Since
$$
L(v_i,v_j)=\sqrt{\tfrac12\lambda_1\eta}h(v_i,v_j)w_1
=\sqrt{\tfrac{1}{2}\lambda_1\eta}\,\delta_{ij}w_1,
$$
we get
$$
K(v_i,v_j)=\big(\tfrac{1}{2}\lambda_1
e_1+\sqrt{\tfrac{1}{2}\lambda_1\eta } w_1\big)\delta_{ij}.
$$
This completes the proof of Lemma \ref{lm:7.2}.
\end{proof}

Next, by \eqref{eqn:4.32} and Lemma \ref{lm:7.1} we get the
following result.

\begin{lemma}\label{lm:7.3}
$K(w_1,w_1)=\mu e_1+(\lambda_1+\mu)\sqrt{\tfrac{2\eta}{\lambda_1}}w_1.$
\end{lemma}

Finally, in case $\mathcal{D}_3\neq\mathbb{R}w_1$ and let
$\{w_2,\ldots,w_{n-m}\}$ be an orthonormal basis of
$\mathcal{D}_3\setminus\mathbb{R}w_1$, by Lemmas \ref{lm:4.5} and
\ref{lm:7.1}, we immediately have:
\begin{lemma}\label{lm:7.4}
$K(w_1,w_i)=\mu_i w_i,\ 2\leq i\leq n-m,$
where $\mu_i=\sqrt{\tfrac{2\eta}{\lambda_1}}\mu$.
\end{lemma}

Now, we are ready to complete the proof of Theorem \ref{th:7.1}.

\vskip 2mm

\begin{proof}[Proof of Theorem \ref{th:7.1}]

Based on Lemmas \ref{lm:7.1}, \ref{lm:7.2}, \ref{lm:7.3} and
\ref{lm:7.4}, by putting
$$
t=\sqrt{\tfrac{\lambda_1}{\lambda_1+2\eta}}e_1
+\sqrt{\tfrac{2\eta}{\lambda_1+2\eta}}w_1,\ \
v=-\sqrt{\tfrac{2\eta}{\lambda_1+2\eta}}e_1
+\sqrt{\tfrac{\lambda_1}{\lambda_1+2\eta}}w_1,
$$
we see that if $\mathcal{D}_3=\mathbb{R}w_1$, then
$\{t,v,v_1,\ldots,v_{m-1}\}$ (or, resp. if
$\mathcal{D}_3\neq\mathbb{R}w_1$, then $\{t,v,v_1,\ldots,v_{m-1},
w_2,\ldots,w_{n-m}\}$) forms an orthonormal basis of $T_pM^{n}$,
with respect to which, the difference tensor $K$ takes the following
form:
\begin{equation}\label{eqn:7.7}
\left\{
\begin{aligned}
&K(t,t)=\sigma_1t;\ \ K(t,v)=\sigma_2v; \ \ K(t,v_i)=\sigma_2v_i,\
\  1\leq i\leq m-1;\\[1mm]
&{\rm if}\ \mathcal{D}_3\neq\mathbb{R}w_1,\ \ K(t,w_i)=\sigma_3w_i,\
\ 2\leq i\leq n-m,
\end{aligned}\right.
\end{equation}
where
\begin{align}\label{eqn:7.8}
\sigma_1=\tfrac{\lambda_1^2+2\eta\mu}{\sqrt{\lambda_1(\lambda_1
+2\eta)}},\ \sigma_2=\tfrac{\tfrac{1}{2}\lambda_1^2
+\lambda_1\eta}{\sqrt{\lambda_1(\lambda_1+2\eta)}},\
\sigma_3=\tfrac{\lambda_1\mu+2\eta\mu}{\sqrt{\lambda_1(\lambda_1+2\eta)}}.
\end{align}

It is easy to show that the constants $\sigma_1,\sigma_2$ and
$\sigma_3$ satisfy the relations:
\begin{align}\label{eqn:7.9}
\sigma_1\neq 2\sigma_2,\ \ \sigma_1\neq 2\sigma_3,\ \
\sigma_2\neq\sigma_3.
\end{align}

By parallel translation along geodesics (with respect to
$\hat{\nabla}$) through $p$, we can extend
$\{t,v,v_1,\ldots,v_{m-1}\}$ (if $\mathcal{D}_3=\mathbb{R}w_1$), or,
resp. $\{t,v,v_1,\ldots,v_{m-1}, w_2,\ldots,w_{n-m}\}$ (if
$\mathcal{D}_3\neq\mathbb{R}w_1$) to obtain a local $h$-orthonormal
basis $\{T,V,V_1,\ldots,V_{m-1}\}$, or, resp.
$\{T,V,V_1,\ldots,V_{m-1},W_2,\ldots,W_{n-m}\}$ such that
$$
\left\{
\begin{aligned}
&K(T,T)=\sigma_1T;\ \ K(T,V)=\sigma_2V; \ \ K(T,V_i)=\sigma_2V_i,\ \
1\leq i\leq m-1;\\[1mm]
&{\rm if}\ \mathcal{D}_3\neq\mathbb{R}w_1,\ \
K(T,W_i)=\sigma_3W_i,\ \ 2\leq i\leq n-m.
\end{aligned}
\right.
$$

Now, the above fact implies that, if
$\mathcal{D}_3\neq\mathbb{R}w_1$ we can apply Theorem \ref{th:3.2}
to conclude that $M^n$ is decomposed as the Calabi product of two
locally strongly convex centroaffine hypersurfaces with parallel
cubic form. If $\mathcal{D}_3=\mathbb{R}w_1$, then we can apply
Theorem \ref{th:3.4} to conclude that $M$ can be decomposed as the
Calabi product of a locally strongly convex centroaffine
hypersurface with parallel cubic form and a point.
\end{proof}
%

%%%%%%%%%%%%%%%%%%%%%%%%%%%%%%%%%%%%%%%%%%%%%%%%%%%%%%%%%%%%%%%%%%%%%%%%%%%%%%%%%%%%%%%%%%
\numberwithin{equation}{section}

\section{Case $\{\mathfrak{C}_m\}_{2\le m\le n-1}$
with $k_0\geq 2$ and $\mathfrak{p}=0$}\label{sect:8}

In this section, we will prove the following theorem.
\begin{theorem}\label{th:8.1}
Let $M^n$ be a locally strongly convex centroaffine hypersurface in
$\mathbb{R}^{n+1}$ which has parallel and non-vanishing cubic form.
If $\mathfrak{C}_m$ with $2\le m\le n-1$ occurs and the integers
$k_0$ and $\mathfrak{p}$, as defined in subsection \ref{sect:4.5},
satisfy $k_0\geq 2$ and $\mathfrak{p}=0$, then $n\geq
\frac{1}{2}m(m+1)-1$. Moreover, we have either
\begin{enumerate}
\item[(i)]$n=\frac{1}{2}m(m+1)$, $M^n$ can be decomposed as the Calabi product of
a locally strongly convex centroaffine hypersurface with parallel
cubic form and a point, or

\item[(ii)] $n>\frac{1}{2}m(m+1)$, $M^n$ can
be decomposed the Calabi product of two locally strongly convex
centroaffine hypersurfaces with parallel cubic form, or

\item[(iii)] $n=\frac{1}{2}m(m+1)-1$, $M^n$ is centroaffinely equivalent to the standard
embedding of \
$\mathrm{SL}(m,\mathbb{R})/\mathrm{SO}(m;\mathbb{R})\hookrightarrow
\mathbb{R}^{n+1}$.
\end{enumerate}
\end{theorem}
In the present situation, the decomposition \eqref{eqn:4.28} reduces
to $\mathcal{D}_2=\{v_1\}\oplus\cdots \oplus \{v_{k_0}\}$. Then
$\dim\,\mathcal{D}_2=k_0=m-1,\ m\ge3$, and $\{v_1,\ldots,v_{k_0}\}$
forms an orthonormal basis of $\mathcal{D}_2$.

According to \eqref{eqn:4.5}, Lemma \ref{lm:4.11}
and the fact that for $j\neq l,\ v_j\in V_{v_l}(\tau)$, we have
\begin{equation}\label{eqn:8.1}
h(L(v_j,v_l),L(v_j,v_l))=\tau,\ \ j\neq l,
\end{equation}
\begin{equation}\label{eqn:8.2}
h(L(v_j,v_{l_1}),L(v_j,v_{l_2}))=0,\ \ j, l_1, l_2 \ {\rm distinct},
\end{equation}
\begin{equation}\label{eqn:8.3}
h(L(v_{j_1},v_{j_2}),L(v_{j_3},v_{j_4}))=0,\ \ j_1, j_2, j_3, j_4\
{\rm distinct},
\end{equation}
\begin{equation}\label{eqn:8.4}
h(L(v_j,v_j),L(v_j,v_j))=\tfrac{1}{2}\lambda_1\eta,
\end{equation}
\begin{equation}\label{eqn:8.5}
h(L(v_j,v_j),L(v_l,v_l))=\tfrac{1}{2}\mu\eta,\ \ j\not=l,
\end{equation}
\begin{equation}\label{eqn:8.6}
h(L(v_j,v_j),L(v_j,v_l))=0,\ \ j\not=l,
\end{equation}
\begin{equation}\label{eqn:8.7}
h(L(v_j,v_j),L(v_{l_1},v_{l_2}))=0,\ \ j,\ l_1,\ l_2\ {\rm
distinct}.
\end{equation}

Denote $L_j:=L(v_1,v_1)+\cdots+L(v_j,v_j)-jL(v_{j+1},v_{j+1}), \
1\leq j\leq k_0-1$. Then it is easy to check
$h(L_j,L_j)=2j(j+1)\tau\neq 0$, and that
\begin{equation}\label{eqn:8.8}
\left\{
\begin{aligned}
&w_j=\tfrac{1}{\sqrt{2j(j+1)\tau}}L_j,\ 1\leq j\leq k_0-1,\\
&w_{kl}=\tfrac{1}{\sqrt{\tau}}L(v_k,v_l),\ 1\leq k<l\leq k_0
\end{aligned}
\right.
\end{equation}
give $\tfrac{1}{2}(m+1)(m-2)$ orthonormal vectors in ${\rm
Im}\,L\subset \mathcal{D}_3$. Thus, we have the estimate of the
dimension
\begin{align}\label{eqn:8.9}
\begin{split}
n&=1+\dim\,(\mathcal{D}_2)+\dim\,(\mathcal{D}_3)\\
&\geq1+m-1+\tfrac{1}{2}(m+1)(m-2)=\tfrac{1}{2}m(m+1)-1.
\end{split}
\end{align}

Direct computations show that ${\rm Tr}\,L=L(v_1,v_1)+\cdots
+L(v_{k_0},v_{k_0})$ is orthogonal to all vectors in
\eqref{eqn:8.8}, and by using \eqref{eqn:4.4}, \eqref{eqn:4.8} and
the fact that $v_i\in V_{v_j}(\tau),\ i\neq j$, we get
\begin{align}\label{eqn:8.10}
h({\rm Tr}\,L,{\rm Tr}\,L)&=\tfrac{1}{2}k_0\eta(\lambda_1+(k_0-1)\mu)\nonumber\\
&=\tfrac18(m-1) \sqrt{\lambda_1^2-4\varepsilon}\big(m\lambda_1-(m-2)
\sqrt{\lambda_1^2-4\varepsilon}\,\big)\\
&=:\rho^2,\nonumber
\end{align}
where $\rho\geq0$. From \eqref{eqn:8.10} and that
$\lambda_1^2-4\varepsilon>0$, the following result is obvious.

\begin{lemma}\label{lm:8.1}
${\rm Tr}\,L=0$ if and only if $\lambda_1=\tfrac{m-2}{\sqrt{m-1}}$
and $\varepsilon=-1$.
\end{lemma}

On the other hand, an implicit fact can be said about the statement
${\rm Tr}\,L=0$.
\begin{lemma}\label{lm:8.2}
${\rm Tr}\,L=0$ if and only if $n=\tfrac{1}{2}m(m+1)-1$.
\end{lemma}
\begin{proof}
If ${\rm Tr}\,L=0$, then we claim that $\mathcal{D}_3={\rm Im}\,L$.
In fact, if $\mathcal{D}_3\not={\rm Im}\,L$, we have a unit vector
$w\in \mathcal{D}_3$ which is orthogonal to ${\rm Im}\,L$. Then by
Lemma \ref{lm:4.5} we get the contradiction
\begin{equation}\label{eqn:8.11}
0=K({\rm Tr}\,L,w)=k_0\eta\mu w.
\end{equation}
Thus, according to this claim and \eqref{eqn:8.9}, we have
$n=\tfrac{1}{2}m(m+1)-1$, provided that ${\rm Tr}\,L=0$.

Conversely, if $n=\tfrac{1}{2}m(m+1)-1$, then by \eqref{eqn:8.9} we
have
$$
\dim\,(\mathcal{D}_3)=\tfrac{1}{2}(m+1)(m-2)
$$
which implies that $\mathcal{D}_3={\rm Im}\,L$. This further implies
that ${\rm Tr}\,L=0$ due to the fact that the vector ${\rm Tr}\,L$,
which belongs to $\mathcal{D}_3$, is orthogonal to all vectors in
\eqref{eqn:8.8}.
\end{proof}

Now, we are ready to complete the proof of Theorem \ref{th:8.1}.

\vskip 2mm

\begin{proof}[Proof of Theorem \ref{th:8.1}] We need to
consider three cases:

\vskip 1mm

{\bf Case\ \ (i)}\ \ \ $n=\tfrac{1}{2}m(m+1)$.

\vskip 1mm

{\bf Case\ \,(ii)}\ \ \ $n>\tfrac{1}{2}m(m+1)$.

\vskip 1mm

{\bf Case (iii)}\ \ \ $n=\tfrac{1}{2}m(m+1)-1$.

\vskip 2mm

For Cases (i) and (ii), as ${\rm Tr}\,L\not=0$, we can define a unit
vector $t:=\frac1{\rho}{\rm Tr}\,L$.

In Case (i), from the previous discussions we see that
$$
\{t, \ w_j\mid_{ 1\leq j\leq k_0-1},\ w_{kl}\mid_{1\leq k<l\leq
k_0}\}
$$
forms an orthonormal basis of ${\rm Im}\,L=\mathcal{D}_3$. By direct
calculations with the use of Lemmas \ref{lm:4.2}, \ref{lm:4.14} and
\eqref{eqn:8.1}--\eqref{eqn:8.7}, we have the following fact which
we state as

\begin{lemma}\label{lm:8.3}
In Case (i), the difference tensor $K$ satisfies
\begin{equation}\label{eqn:8.12}
\left\{
\begin{aligned}
&K(t,e_1)=\mu t,\ \ K(t,v_i)=\tfrac{\rho}{k_0}v_i,\ \ 1\le i\le
m-1,\\
&K(t,w_j)=\tfrac{2\rho}{k_0}w_j,\ \ 1\leq j\leq k_0-1,\\
&K(t,w_{kl})=\tfrac{2\rho}{k_0}w_{kl},\ \ 1\leq k<l\leq
k_0,\\
&K(t,t)=\mu
e_1+\left(\tfrac{2\rho}{k_0}+\tfrac{k_0\mu\eta}{\rho}\right)t.
\end{aligned}\right.
\end{equation}
\end{lemma}

Put
\begin{equation}\label{eqn:8.13}
T=\tfrac{\rho}{\sqrt{\rho^2+k_0^2\eta^2}}e_1
+\tfrac{k_0\eta}{\sqrt{\rho^2+k_0^2\eta^2}}t,\ \
T^*=-\tfrac{k_0\eta}{\sqrt{\rho^2+k_0^2\eta^2}}e_1
+\tfrac{\rho}{\sqrt{\rho^2+k_0^2\eta^2}}t.
\end{equation}

It is easily to see that $\{T,\ T^*,\ v_j|_{1\le j\le k_0}, \
w_j\mid_{ 1\leq j\leq k_0-1},\ w_{kl}\mid_{1\leq k<l\leq k_0}\}$ is
an orthonormal basis of $T_pM$. By Lemmas \ref{lm:4.2} and \ref{lm:8.3} we have the
following lemma.

\begin{lemma}\label{lm:8.4}
In Case (i), under the above notations, we have
\begin{equation}\label{eqn:8.14}
\left\{
\begin{aligned}
&K(T,T)=\sigma_1T,\ \ \ \ \ \
K(T,v_j)=\sigma_2v_j,\ \ \ 1\le j\le k_0;\\
&K(T,T^*)=\sigma_2T^*,\ \ \ K(T,w_j)=\sigma_2w_j,\ \ 1\le j\le k_0-1;\\
&K(T,w_{kl})=\sigma_2w_{kl},\ \ 1\leq k<l\leq k_0,
\end{aligned}\right.
\end{equation}
where $\sigma_1$ and $\sigma_2$ are defined by
\begin{equation}\label{eqn:8.15}
\sigma_1=\tfrac{\rho^2\lambda_1+k_0^2\eta^2\mu}
{\rho\sqrt{\rho^2+k_0^2\eta^2}},\ \
\sigma_2=\tfrac{(\tfrac{1}{2}\lambda_1+\eta)\rho}
{\sqrt{\rho^2+k_0^2\eta^2}},
\end{equation}
which satisfy $\sigma_1\neq 2\sigma_2$.
\end{lemma}

Given the parallelism of the difference tensor $K$, Lemma
\ref{lm:8.4} and Theorem \ref{th:3.4}, we conclude that in Case (i),
$M$ is locally the Calabi product of a lower-dimensional locally
strongly convex centroaffine hypersurface with parallel cubic form
with a point.

\vskip 2mm

In Case (ii), we proceed in the same way as in Case (i). We first
see that
$$
\{t,\ w_{kl}\mid_{1\leq k<l\leq k_0},
\ w_j\mid_{ 1\leq j\leq k_0-1} \}
$$
is still an orthonormal basis of ${\rm Im}\,L$, even though we have
${\rm Im}\,L\varsubsetneq\mathcal{D}_3$.

Denote $\tilde{n}=n-\frac{1}{2}m(m+1)$ and choose
$\tilde{w}_1,\ldots,\tilde{w}_{\tilde{n}}$ in the orthogonal
complement of ${\rm Im}\,L$ in $\mathcal{D}_3$ such that
$$
\{t,\ w_{kl}\mid_{1\leq k<l\leq k_0},
\ w_j\mid_{ 1\leq j\leq k_0-1},
\ \tilde{w}_r\mid_{ 1\leq r\leq \tilde{n}} \}
$$
is an orthonormal basis of $\mathcal{D}_3$. By
Lemma \ref{lm:4.5}, we obtain that
\begin{align}\label{eqn:8.16}
\begin{split}
 K(t,\tilde{w}_r)=k_0\eta\mu\rho^{-1}\tilde{w}_r.
\end{split}
\end{align}

We define $T$ and $T^*$ as in \eqref{eqn:8.13}. Then
$$
\{T,\ T^*,\ v_j|_{1\le j\le k_0},\ w_{kl}\mid_{1\leq k<l\leq k_0}, \
w_j\mid_{ 1\leq j\leq k_0-1},\ \tilde{w}_r|_{1\le r\le\tilde{n}}\}
$$
is an orthonormal basis of $T_pM$. Similar to Lemma \ref{lm:8.4}, we
can easily show the following
\begin{lemma}\label{lm:8.5}
In Case (ii), under the previous notations, we have
\begin{equation}\label{eqn:8.17}
\left\{
\begin{aligned}
&K(T,T)=\sigma_1T,\ \ \ \ \ \
K(T,v_j)=\sigma_2v_j,\ \ \ 1\le j\le k_0;\\
&K(T,T^*)=\sigma_2T^*,\ \ \ K(T,w_j)=\sigma_2w_j,\ \ 1\le j\le k_0-1;\\
&K(T,w_{kl})=\sigma_2w_{kl},\ \ 1\leq k<l\leq k_0;\\
&K(T,\tilde{w}_r)=\sigma_3\tilde{w}_r,\ \ \ \ 1\leq r\leq \tilde{n},
\end{aligned}\right.
\end{equation}
where $\sigma_1$ and $\sigma_2$ are defined by \eqref{eqn:8.15},
and
\begin{equation}\label{eqn:8.18}
\sigma_3={\rho}^{-1}\mu\sqrt{\rho^2+k_0^2\eta^2},
\end{equation}
which satisfy the relations $\sigma_1\neq 2\sigma_2,\ \sigma_1\neq
2\sigma_3$ and $\sigma_2\neq\sigma_3$.
\end{lemma}

Given the parallelism of the difference tensor $K$, Lemma
\ref{lm:8.5} and Theorem \ref{th:3.2}, we conclude that in Case
(ii), $M$ is locally the Calabi product of two lower-dimensional
locally strongly convex centroaffine hypersurfaces with parallel
cubic form.

\vskip 2mm

In Case (iii), we take the following basis of $T_pM$:
\begin{equation}\label{eqn:8.19}
\{e_1,\ v_i|_{1\le i\le k_0},\ w_j\mid_{1\leq j\leq k_0-1},\
w_{jk}\mid_{1\leq j<k\leq k_0-1}\}.
\end{equation}

By Lemmas \ref{lm:8.1}, \ref{lm:8.2}, \ref{lm:4.14}
and a direct computation, we obtain that
\begin{equation}\label{eqn:8.20}
K(e_1,e_1)+\sum_{j=1}^{k_0}K(v_j,v_j)+\sum_{j=1}^{k_0-1}K(w_j,w_j)
+\sum_{1\leq i<j\leq k_0}K(w_{ij},w_{ij})=0.
\end{equation}
This implies that in Case (iii) it holds ${\rm Tr}\,K_X=0$ for any
vector $X$. Thus $M$ is a proper affine hypersphere. Then, according
to previous computations and the proof of Theorem 5.1 in
\cite{HLV2}, we can easily show that in Case (iii) $M^n$ is
centroaffinely equivalent to the standard embedding $\mathrm{SL}(m,
\mathbb{R})/\mathrm{SO}(m;\mathbb{R})\hookrightarrow
\mathbb{R}^{n+1}$.

\vskip 1mm

The combination of the preceding three cases' discussion then
completes the proof of Theorem \ref{th:8.1}.
\end{proof}

%%%%%%%%%%%%%%%%%%%%%%%%%%%%%%%%%%%%%%%%%%%%%%%%%%%%%%%%%%%%%%%%%%%%%%%%%%%%%%%%%%%%%%%%%
\numberwithin{equation}{section}

\section{Case $\{\mathfrak{C}_m\}_{2\le m\le n-1}$
with $k_0\geq 2$ and $\mathfrak{p}=1$}\label{sect:9}

In this section, we will prove the following theorem.
\begin{theorem}\label{th:9.1}
Let $M^n$ be a locally strongly convex centroaffine hypersurface in
$\mathbb{R}^{n+1}$ which has parallel and non-vanishing cubic form.
If $\mathfrak{C}_m$ with $2\le m\le n-1$ occurs and the integers
$k_0$ and $\mathfrak{p}$, as defined in subsection \ref{sect:4.5},
satisfy $k_0\geq 2$ and $\mathfrak{p}=1$, then $n\geq
\tfrac{1}{4}(m+1)^2-1$. Moreover, we have either
\begin{enumerate}
\item[(i)] $n=\tfrac{1}{4}(m+1)^2$, $M^n$ can be decomposed
as the Calabi product of a locally strongly convex centroaffine
hypersurface with parallel cubic form and a point, or
\item[(ii)] $n>\tfrac{1}{4}(m+1)^2$, $M^n$ can
be decomposed as the Calabi product of two locally strongly convex
centroaffine hypersurfaces with parallel cubic form, or
\item[(iii)] $n=\tfrac{1}{4}(m+1)^2-1$,
$M^n$ is centroaffinely equivalent to the standard embedding
$\mathrm{SL}(\tfrac{m+1}{2},\mathbb{C})/\mathrm{SU}(\tfrac{m+1}{2})
\hookrightarrow\mathbb{R}^{n+1}$.
\end{enumerate}
\end{theorem}

Now we have $\dim\,\mathcal{D}_2=m-1=2k_0$ and $m\ge5$. Similar to
Lemma 6.1 of \cite{HLV2}, we will prove the following
\begin{lemma}\label{lm:9.1}
In the decomposition \eqref{eqn:4.28}, if we have $k_0\geq 2$ and
$\mathfrak{p}=1$, then there exist unit vectors $u_j\in V_{v_j}(0)\
(1\leq j\leq k_0)$ such that the orthonormal basis
$\{v_1,u_1,\ldots,v_{k_0}, u_{k_0}\}$ of $\mathcal{D}_2$ satisfies
the relations
\begin{align}\label{eqn:9.1}
L(u_l, v_j)=-L(v_l, u_j),\ L(v_l, v_j)=L(u_l, u_j),\
1\leq j,\ l\leq k_0.
\end{align}
\end{lemma}

\begin{proof}
As for each $1\leq j\leq k_0$ it holds $\dim\,(V_{v_j}(0))=1$, we
assume $V_{v_2}(0)=\{u_2\}$ for a unit vector $u_2$. Then, for each
$j\neq 2$, by Lemma \ref{lm:4.13}, we have a unique unit vector
$u_j\in V_{v_j}(0)$ satisfying
\begin{align}\label{eqn:9.2}
L(u_2, v_j)=-L(v_2, u_j),\ L(v_2, v_j)=L(u_2, u_j),\
1\leq j\leq k_0,\ j\neq 2.
\end{align}

Moreover, Lemma \ref{lm:4.9} implies that \eqref{eqn:9.2} also
holds for $j=2$. Next, we state

\vskip 1mm

{\bf Claim 1}. $L(u_l, v_j)=-L(v_l, u_j),\ L(v_l, v_j)=L(u_l, u_j),\
1\leq j,l\leq k_0,\ j, l\neq2.$

\vskip 1mm

To verify the claim, as $u_j\in V_{v_j}(0)$, we first see by Lemma
\ref{lm:4.9} that $L(u_j, v_j)=0$ and $L(v_j, v_j)=L(u_j, u_j)$.
Hence the claim is true for $j=l$.

Now we fix $j\neq l$ such that $j, l\neq2$. By Lemma \ref{lm:4.13},
there exists a unique unit vector $u_j^{(l)}\in V_{v_j}(0)$, such
that
\begin{align}\label{eqn:9.3}
L(u_l, v_j)=-L(v_l, u_j^{(l)}),\ L(v_l, v_j)=L(u_l, u_j^{(l)}).
\end{align}

Noting that $\dim\,(V_{v_j}(0))=1$ and $u_j^{(l)},\,u_j\in
V_{v_j}(0)$ are unit vectors, we have $u_j^{(l)}=\epsilon u_j$ with
$\epsilon=\pm 1$. Hence from \eqref{eqn:9.3} we have
\begin{align}\label{eqn:9.4}
L(u_l, v_j)=-\epsilon L(v_l, u_j),\ L(v_l, v_j)=\epsilon L(u_l, u_j).
\end{align}

On the other hand, by using \eqref{eqn:4.36}, \eqref{eqn:9.2} and
\eqref{eqn:9.4}, we get
$$
\begin{aligned}
&K(L(v_j,v_l),L(v_2,u_j))=K(L(v_j,v_l),-L(v_j,u_2))=-\tau L(v_l,u_2),\\
&K(L(v_j,v_l),L(v_2,u_j))=K(\epsilon
L(u_j,u_l),L(v_2,u_j))=-\epsilon\tau L(v_l,u_2).
\end{aligned}
$$
From the comparison of the above two equations we get $\epsilon=1$.

From \eqref{eqn:9.4} we have verified Claim 1 and the proof of Lemma
\ref{lm:9.1} is fulfilled.
\end{proof}

To continue the proof of Theorem \ref{th:9.1}, we now assume that
$k_0\ge2$ and let $\{v_1,u_1,\ldots,v_{k_0}, u_{k_0}\}$ be the
orthonormal basis of $\mathcal{D}_2$ as constructed in Lemma
\ref{lm:9.1}.

Given \eqref{eqn:4.5}, Lemmas \ref{lm:4.9}, \ref{lm:4.11} and that for $j\neq l,
v_j, u_j \in V_{v_l}(\tau)=V_{u_l}(\tau)$, we have the following
calculations:
\begin{equation}\label{eqn:9.5}
h(L(v_j,u_l),L(v_j,u_l))=h(L(v_j,v_l),L(v_j,v_l))=\tau,\ j\neq l,
\end{equation}
\begin{align}\label{eqn:9.6}
h(L(u_j,&v_{l_1}),L(u_j,v_{l_2}))=h(L(v_j,u_{l_1}),L(v_j,u_{l_2}))\nonumber\\
&=h(L(v_j,v_{l_1}),L(v_j,v_{l_2}))=0,\ \ j, l_1, l_2 \ {\rm
distinct},
\end{align}
\begin{equation}\label{eqn:9.7}
h(L(v_{j_1},v_{j_2}),L(v_{j_3},v_{j_4}))=0,\ \ j_1, j_2, j_3, j_4\
{\rm distinct},
\end{equation}
\begin{equation}\label{eqn:9.8}
h(L(v_j,v_l),L(v_{j_1},u_{l_1}))=0,\ \ j\neq l\ {\rm and}\ j_1\neq
l_1,
\end{equation}
\begin{equation}\label{eqn:9.9}
h(L(v_j,v_j),L(v_j,v_j))=\tfrac{1}{2}\lambda_1\eta,\ \ 1\leq j\leq
k_0,
\end{equation}
\begin{equation}\label{eqn:9.10}
h(L(v_j,v_j),L(v_l,v_l))=\tfrac{1}{2}\mu\eta,\ \ 1\leq j\neq l\leq
k_0,
\end{equation}
\begin{align}\label{eqn:9.11}
h(L(v_j,v_j),L(v_j,v_l))&=h(L(v_j,v_j),L(v_j,u_l))\nonumber\\
&=h(L(v_j,v_j),L(v_l,u_j))=0,\ \ 1\leq
j\not=l\leq k_0;
\end{align}
\begin{align}\label{eqn:9.12}
h(L(v_j,v_j),L(v_{l_1},v_{l_2}))=h(&L(v_j,v_j),L(v_{l_1},u_{l_2}))=0,\\
&1\leq j, l_1, l_2\ {\rm distinct}\leq k_0.\nonumber
\end{align}

Similar as in the proof of Theorem \ref{th:8.1}, we denote
$$
L_j:=L(v_1,v_1)+\cdots+L(v_j,v_j)-jL(v_{j+1},v_{j+1}), \ 1\leq j\leq
k_0-1.
$$
Then direct calculations show that $h(L_j,L_j)=2j(j+1)\tau\neq 0$
for each $j$, and
\begin{equation}\label{eqn:9.13}
\left\{
\begin{aligned}
&w_j=\tfrac{1}{\sqrt{2j(j+1)\tau}}L_j,\ \ 1\leq j\leq k_0-1,\\
&w_{kl}=\tfrac{1}{\sqrt{\tau}}L(v_k,v_l),\ \ 1\leq k<l\leq k_0,\\
&w_{kl}'=\tfrac{1}{\sqrt{\tau}}L(v_k,u_l),\ \ 1\leq k<l\leq k_0
\end{aligned}
\right.
\end{equation}
give $\tfrac{1}{4}(m+1)(m-3)$ mutually orthogonal unit vectors in
${\rm Im}\,L\subset\mathcal{D}_3$. Thus, we have the estimate of the
dimension
\begin{align}\label{eqn:9.14}
\begin{split}
n&=1+\dim\,(\mathcal{D}_2)+\dim\,(\mathcal{D}_3)\\
&\geq1+m-1+\tfrac{1}{4}(m+1)(m-3)=\tfrac{1}{4}(m+1)^2-1.
\end{split}
\end{align}

Moreover, direct computations show that ${\rm
Tr}\,L=2[L(v_1,v_1)+\cdots +L(v_{k_0},v_{k_0})]$ is orthogonal to
all vectors in \eqref{eqn:9.13}, and by using \eqref{eqn:4.4},
\eqref{eqn:4.8} and the fact that $v_i\in V_{v_j}(\tau)$ for $i\neq
j$, we get
\begin{align}\label{eqn:9.15}
\tfrac14h({\rm Tr}\,L,{\rm Tr}\,L)&=\tfrac12k_0\eta(\lambda_1+(k_0-1)\mu)\nonumber\\
&=\tfrac1{32}(m-1)
\sqrt{\lambda_1^2-4\varepsilon}\Big[(m+1)\lambda_1-(m-3)
\sqrt{\lambda_1^2-4\varepsilon}\,\Big]\\
&=:\rho^2\nonumber
\end{align}
for $\rho\geq0$. From \eqref{eqn:9.15} and that
$\lambda_1^2-4\varepsilon>0$, the following result is obvious.

\begin{lemma}\label{lm:9.2}
${\rm Tr}L=0$ if and only if $\lambda_1=\tfrac{m-3}{\sqrt{2(m-1)}}$
and $\varepsilon=-1$.
\end{lemma}

On the other hand, the statement ${\rm Tr}\,L=0$ has an implicit
characterization with a proof totally similar to that of Lemma
\ref{lm:8.2}.

\begin{lemma}\label{lm:9.3}
${\rm Tr}\,L=0$ if and only if $n=\tfrac{1}{4}(m+1)^2-1$.
\end{lemma}

\vskip 1mm

Now, we are ready to complete the proof of Theorem \ref{th:9.1}.

\begin{proof}[Proof of Theorem \ref{th:9.1}]

First, if $n\neq\tfrac{1}{4}(m+1)^2-1$, we define a unit vector
$t=\frac{1}{2\rho}{\rm Tr}\,L$. We separate the discussions into
three cases:

(i) If $n=\tfrac{1}{4}(m+1)^2$, the previous results show that
$$\{t,\ w_j\mid_{ 1\leq j\leq k_0-1},\ w_{kl}\mid_{1\leq k<l\leq k_0},
\ w_{kl}'\mid_{1\leq k<l\leq k_0} \}$$ is an orthonormal basis of
${\rm Im}\,L=\mathcal{D}_3$.

(ii) If $n>\tfrac{1}{4}(m+1)^2$, we still have that $\{t,\
w_{kl}\mid_{1\leq k<l\leq k_0}, \ w_j\mid_{ 1\leq j\leq k_0-1} \}$
is an orthonormal basis of ${\rm Im}\,L$, but now ${\rm Im}\,L
\varsubsetneq \mathcal{D}_3$. Denote
$\tilde{n}=n-\tfrac{1}{4}(m+1)^2$ and let
$\{\tilde{w}_1,\ldots,\tilde{w}_{\tilde{n}}\}$ be an orthonormal
basis of $\mathcal{D}_3\setminus{\rm Im}\,L$ such that
$$
\{t,\ w_j\mid_{ 1\leq j\leq k_0-1},\ w_{kl}\mid_{1\leq k<l\leq
k_0},\ w_{kl}'\mid_{1\leq k<l\leq k_0}, \ \tilde{w}_r\mid_{ 1\leq
r\leq \tilde{n}}\}
$$
is an orthonormal basis of $\mathcal{D}_3$.

(iii) If $n=\tfrac{1}{4}(m+1)^2-1$, then an orthonormal basis of
${\rm Im}\,L=\mathcal{D}_3$ is given by
$$
\{w_j\mid_{ 1\leq j\leq k_0-1} ,\ w_{kl}\mid_{1\leq k<l\leq k_0}, \
w_{kl}'\mid_{1\leq k<l\leq k_0}\}.
$$

Now, following the proof of Theorem 6.1 in \cite{HLV2}, we can
proceed in the same way as in the proof of Theorem \ref{th:8.1} to
obtain the following conclusions:

If $n=\tfrac{1}{4}(m+1)^2$, we can apply Theorem \ref{th:3.4} to
conclude that $M^n$ can be decomposed as the Calabi product of a
locally strongly convex centroaffine hypersurface with parallel
cubic form and a point.

If $n>\tfrac{1}{4}(m+1)^2$, we can apply Theorem \ref{th:3.2} to
conclude that $M^n$ can be decomposed as the Calabi product of two
locally strongly convex centroaffine hypersurfaces with parallel
cubic form.

If $n=\tfrac{1}{4}(m+1)^2-1$, then $M^n$ is centroaffinely
equivalent to the standard embedding
$\mathrm{SL}(\tfrac{m+1}{2},\mathbb{C})/\mathrm{SU}(\tfrac{m+1}{2})
\hookrightarrow\mathbb{R}^{n+1}$.
\end{proof}

%%%%%%%%%%%%%%%%%%%%%%%%%%%%%%%%%%%%%%%%%%%%%%%%%%%%%%%%%%%%%%%%%%
\numberwithin{equation}{section}

\section{Case $\{\mathfrak{C}_m\}_{2\le m\le n-1}$
with $k_0\geq 2$ and $\mathfrak{p}=3$}\label{sect:10}

In this section, we will prove the following theorem.
\begin{theorem}\label{th:10.1}
Let $M^n$ be a locally strongly convex centroaffine hypersurface in
$\mathbb{R}^{n+1}$ which has parallel and non-vanishing cubic form.
If $\mathfrak{C}_m$ with $2\le m\le n-1$ occurs and the integers
$k_0$ and $\mathfrak{p}$, as defined in subsection \ref{sect:4.5},
satisfy $k_0\geq 2$ and $\mathfrak{p}=3$, then $n\geq
\tfrac{1}{8}(m+1)(m+3)-1$. Moreover, we have either
\begin{enumerate}
\item[(i)] $n=\tfrac{1}{8}(m+1)(m+3)$, $M^n$ can be decomposed as
the Calabi product of a locally strongly convex centroaffine
hypersurface with parallel cubic form and a point, or
\item[(ii)] $n>\tfrac{1}{8}(m+1)(m+3)$, $M^n$ can
be decomposed the Calabi product of two locally strongly convex
centroaffine hypersurfaces with parallel cubic form, or
\item[(iii)] $n=\tfrac{1}{8}(m+1)(m+3)-1$,
$M^n$ is centroaffinely equivalent to the standard embedding
$\mathrm{SU}^*(\tfrac{m+3}{2})/\mathrm{Sp}(\tfrac{m+3}{4})\hookrightarrow
\mathbb{R}^{n+1}$.
\end{enumerate}
\end{theorem}

Now we have $\dim\,\mathcal{D}_2=m-1=4k_0$ and $m\ge9$. Similar to
Lemma 7.1 of \cite{HLV2}, we will prove the following lemma.
\begin{lemma}\label{lm:10.1}
In the decomposition \eqref{eqn:4.28}, if we have $k_0\geq 2$ and
$\mathfrak{p}=3$, then there exist unit vectors $x_j,y_j,z_j\in
V_{v_j}(0)\ (1\leq j\leq k_0)$ such that the orthonormal basis
$\{v_1,x_1,y_1, z_1;\ldots;v_{k_0},x_{k_0},y_{k_0}, z_{k_0}\}$ of
$\mathcal{D}_2$ satisfies the relations
\begin{equation}\label{eqn:10.1}
\left\{
\begin{aligned}
&L(v_j,v_l)=L(x_j,x_l)=L(y_j,y_l)=L(z_j,z_l),\\
&L(v_j,x_l)=-L(x_j,v_l)=-L(y_j,z_l)=L(z_j,y_l),\\
&L(v_j,y_l)=-L(y_j,v_l)=-L(z_j,x_l)=L(x_j,z_l),\\
&L(v_j,z_l)=-L(z_j,v_l)=-L(x_j,y_l)=L(y_j,x_l),
\end{aligned}\right.\ \ \ 1\leq j,\ l\leq k_0.
\end{equation}
\end{lemma}
\begin{proof}
As doing before, we denote $V_j=\{v_j\}\oplus V_{v_j}(0),\ 1\leq
l\leq k_0$. Let us fix two orthogonal unit vectors $x_1,y_1\in
V_{v_1}(0)$. By using Lemmas \ref{lm:4.12} and \ref{lm:4.13}, for
each $j\not=1$, we have two unit vectors $x_j,\,y_j \in V_{v_j}(0)$
such that
\begin{equation}\label{eqn:10.2}
\left\{
\begin{aligned}
&L(v_j,v_1)=L(x_j,x_1)=L(y_j,y_1),\\
&L(v_j,x_1)=-L(x_j,v_1),\ \ L(v_j,y_1)=-L(y_j,v_1).
\end{aligned}\right.
\end{equation}

Then, according to Lemma \ref{lm:4.13}, we further have unit vectors
$z_1^j\in V_{x_1}(0)$ and $z_j\in V_{x_j}(0)$ such that
\begin{equation}\label{eqn:10.3}
\left\{
\begin{aligned}
&L(v_j,z_1^j)=L(y_j,x_1),\ L(v_j,x_1)=-L(y_j,z_1^j), \\
&L(z_j,v_1)=L(x_j,y_1),\ L(z_j,y_1)=-L(x_j,v_1).
\end{aligned}\right.
\end{equation}

The important is that we have the following

\vskip 1mm

{\bf Claim 1}. {\it For each $j\neq2$, $\{x_1, y_1, z_1^j\}$ is an
orthonormal basis of $V_{v_1}(0)$ and $\{x_j, y_j,z_j\}$ is an
orthonormal basis of $V_{v_j}(0)$.}

\vskip 1mm

To verify this claim, it suffices to show that
$$
z_1^j \perp v_1,\ \ z_1^j \perp y_1,\ \ x_j\perp y_j,\ \ z_j\perp
y_j,\ \ z_j\perp v_j.
$$

In fact, by using \eqref{eqn:10.2} and \eqref{eqn:10.3}, we obtain
that
$$
\tau
h(z_1^j,v_1)=h(L(z_1^j,v_j),L(v_j,v_1))=h(L(y_j,x_1),L(y_j,y_1))=0,
$$
$$
\tau h(z_1^j,y_1)=h(L(z_1^j,v_j),L(v_j,y_1))=h(L(y_j,x_1),-L(y_j,v_1))=0,
$$
$$
\tau h(x_j,y_j)=h(L(x_j,v_1),L(y_j,v_1))=h(L(v_j,x_1),L(v_j,y_1))=0,
$$
$$
\tau h(z_j,y_j)=h(L(z_j,v_1),L(y_j,v_1))=h(L(x_j,y_1),-L(v_j,y_1))=0,
$$
$$
\tau h(z_j,v_j)=h(L(z_j,v_1),L(v_j,v_1))=h(L(x_j,y_1),L(x_j,x_1))=0.
$$

From these relations, we immediately get the claim.

\vskip 1mm

Next, by using Lemmas \ref{lm:4.12} and \ref{lm:4.13},
\eqref{eqn:10.2} and \eqref{eqn:10.3}, we have
\begin{equation}\label{eqn:10.4}
\left\{
\begin{aligned}
&L(v_j,v_1)=L(x_j,x_1)=L(y_j,y_1)=L(z_j,z_1^j),\\
&L(v_j,x_1)=-L(x_j,v_1)=-L(y_j,z_1^j)=L(z_j,y_1),\\
&L(v_j,y_1)=-L(y_j,v_1)=-L(z_j,x_1)=L(x_j,z_1^j),\\
&L(v_j,z_1^j)=-L(z_j,v_1)=-L(x_j,y_1)=L(y_j,x_1),
\end{aligned}\right.\ \ \ 2\leq j\leq k_0.
\end{equation}

From these relations we can prove the following assertion:

\vskip 1mm

{\bf Claim 2}. {\it $z_1^2=\cdots=z_1^{k_0}=:z_1$.}~

\vskip 1mm

In fact, by Claim 1, we know that for $j\neq l\ (j, l\geq 2)$ we
have $z_1^j=\varepsilon_{jl}z_1^l$ with $\varepsilon_{jl}=\pm1$.
From Lemma \ref{lm:4.14} and \eqref{eqn:10.4} we get
\begin{equation}\label{eqn:10.5}
\begin{aligned}
\varepsilon_{jl}\tau L(v_j,v_l)&=K(L(z_1^j,v_j),L(z_1^l,v_l))\\
&=K(L(y_j,x_1),L(y_l,x_1))=\tau L(y_j,y_l).
\end{aligned}
\end{equation}

Similarly, we get
\begin{align}\label{eqn:10.6}
\varepsilon_{jl}L(x_j,x_l)=L(y_j,y_l)=L(z_j,z_l)=L(v_j,v_l).
\end{align}

From \eqref{eqn:10.5} and \eqref{eqn:10.6} we have
$\varepsilon_{jl}=1$. Thus Claim 2 is verified.

Moreover, the following relations hold
\begin{align}\label{eqn:10.7}
\begin{split}
L(v_j,v_l)=L(x_j,x_l)=L(y_j,y_l)=L(z_j,z_l),\ j\neq l,\ j,\ l\geq 2.
\end{split}
\end{align}

From \eqref{eqn:10.4} and apply Lemma \ref{lm:4.14}, we get
\begin{align}\label{eqn:10.8}
\tau L(x_j,y_l)&=K(L(y_1,x_j),L(y_1,y_l))\nonumber \\
&=K(L(z_j,v_1),L(v_1,v_l))=\tau L(z_j,v_l).
\end{align}

Similarly, we have the following relations:
\begin{align}\label{eqn:10.9}
L(z_j,x_l)=L(y_j,v_l),\ \ L(y_j,z_l)=L(x_j,v_l).
\end{align}

Combination of \eqref{eqn:10.4}, Claim 2 and
\eqref{eqn:10.7}\,--\,\eqref{eqn:10.9}, we get \eqref{eqn:10.1}
immediately.
\end{proof}

\vskip 1mm

To continue the proof of Theorem \ref{th:10.1}, we now assume that
$k_0\ge2$ and let $\{v_1, x_1, y_1, z_1;\,\ldots;\,v_{k_0}, x_{k_0},
y_{k_0}, z_{k_0}\}$ be the orthonormal basis of $\mathcal{D}_2$ as
constructed in Lemmas \ref{lm:4.9} and \ref{lm:10.1}. According to \eqref{eqn:4.5},
Lemma \ref{lm:4.11} and the fact that for $j\neq l$, $v_j,\, x_j,\,
y_j,\, z_j\in
V_{v_l}(\tau)=V_{x_l}(\tau)=V_{y_l}(\tau)=V_{z_l}(\tau)$, we have
\begin{equation}\label{eqn:10.10}
\begin{aligned}
h(L(v_j,x_l),L(v_j,x_l))&=h(L(v_j,y_l),L(v_j,y_l))=h(L(v_j,z_l),L(v_j,z_l))\\
&=h(L(v_j,v_l),L(v_j,v_l))=\tau,\ \ \ j\neq l,
\end{aligned}
\end{equation}
\begin{equation}\label{eqn:10.11}
\begin{aligned}
h(L(v_j,v_{l_1}),&L(v_j,v_{l_2}))=h(L(v_j,x_{l_1}),L(v_j,x_{l_2}))\\
&=h(L(x_j,v_{l_1}),L(x_j,v_{l_2}))=h(L(y_j,v_{l_1}),L(y_j,v_{l_2}))\\
&=h(L(v_j,y_{l_1}),L(v_j,y_{l_2}))=h(L(z_j,v_{l_1}),L(z_j,v_{l_2}))\\
&=h(L(v_j,z_{l_1}),L(v_j,z_{l_2}))=0,\ \ \ j, l_1, l_2\ {\rm
distinct},
\end{aligned}
\end{equation}
\begin{equation}\label{eqn:10.12}
\begin{aligned}
h(L(v_{j_1},v_{j_2}),&L(v_{j_3},v_{j_4}))
=h(L(v_{j_1},x_{j_2}),L(v_{j_3},x_{j_4}))\\
&=h(L(v_{j_1},y_{j_2}),L(v_{j_3},y_{j_4}))=h(L(v_{j_1},z_{j_2}),L(v_{j_3},z_{j_4}))\\
&=0,\ \ \ j_1, j_2, j_3, j_4\ {\rm distinct},
\end{aligned}
\end{equation}
\begin{equation}\label{eqn:10.13}
\begin{aligned}
h(L(v_j,v_l),&L(v_{j_1},x_{l_1}))=h(L(v_j,v_l),L(v_{j_1},y_{l_1}))\\
&=h(L(v_j,v_l),L(v_{j_1},z_{l_1}))=0,\ \ \ j\neq l\ {\rm and}\
j_1\neq l_1,
\end{aligned}
\end{equation}
\begin{equation}\label{eqn:10.14}
h(L(v_j,v_j),L(v_j,v_j))=\tfrac{1}{2}\lambda_1\eta,\ \ \ 1\leq j\leq
k_0,
\end{equation}
\begin{equation}\label{eqn:10.15}
h(L(v_j,v_j),L(v_l,v_l))=\tfrac{1}{2}\mu\eta,\ \ \ j\neq l,
\end{equation}
\begin{equation}\label{eqn:10.16}
\begin{aligned}
h(L(v_j,v_j),L(v_j,v_l))&=h(L(v_j,v_j),L(v_j,x_l))=h(L(v_j,v_j),L(v_j,y_l))\\
&=h(L(v_j,v_j),L(v_j,z_l))=h(L(v_j,v_j),L(v_l,x_j))\\
&=h(L(v_j,v_j),L(v_l,y_j))=h(L(v_j,v_j),L(v_l,z_j))\\
&=0,\ \ \ j\neq l,
\end{aligned}
\end{equation}
\begin{equation}\label{eqn:10.17}
\begin{aligned}
h(L(v_j,v_j),&L(v_{l_1},v_{l_2}))=h(L(v_j,v_j),L(v_{l_1},x_{l_2}))\\
&=h(L(v_j,v_j),L(v_{l_1},y_{l_2}))=h(L(v_j,v_j),L(v_{l_1},z_{l_2}))\\
&=0,\ \ \ j, l_1, l_2\ {\rm distinct}.
\end{aligned}
\end{equation}

As in preceding sections we denote
$$
L_j:=L(v_1,v_1)+\cdots+L(v_j,v_j)-jL(v_{j+1},v_{j+1}),\ \ 1\leq
j\leq k_0-1.
$$
Then we have $h(L_j,L_j)=2j(j+1)\tau\neq 0$ for each $j$. Moreover,
\begin{equation}\label{eqn:10.18}
\left\{
\begin{aligned}
&w_j=\tfrac{1}{\sqrt{2j(j+1)\tau}}L_j,\ \ 1\leq j\leq k_0-1,\\
&w_{kl}=\tfrac{1}{\sqrt{\tau}}L(v_k,v_l),\ \ 1\leq k<l\leq k_0,\\
&w_{kl}'=\tfrac{1}{\sqrt{\tau}}L(v_k,x_l),\ \ 1\leq k<l\leq k_0,\\
&w_{kl}''=\tfrac{1}{\sqrt{\tau}}L(v_k,y_l),\ \ 1\leq k<l\leq k_0,\\
&w_{kl}'''=\tfrac{1}{\sqrt{\tau}}L(v_k,z_l),\ \ 1\leq k<l\leq k_0
\end{aligned}
\right.
\end{equation}
give $\tfrac{1}{8}(m+1)(m-5)$ mutually orthogonal unit vectors in
${\rm Im}\,L\subset \mathcal{D}_3$. Thus we have the estimate of the
dimension
\begin{equation}\label{eqn:10.19}
\begin{aligned}
n&=1+\dim\,(\mathcal{D}_2)+\dim\,(\mathcal{D}_3)\\
&\geq1+m-1+\tfrac{1}{8}(m+1)(m-5)=\tfrac{1}{8}(m+1)(m+3)-1.
\end{aligned}
\end{equation}

Further direct computations show that ${\rm
Tr}\,L=4[L(v_1,v_1)+\cdots +L(v_{k_0},v_{k_0})]$ is orthogonal to
all vectors in \eqref{eqn:10.18}, and by using the fact that $v_i\in
V_{v_j}(\tau)\ (i\neq j)$, \eqref{eqn:4.4} and \eqref{eqn:4.8} we
have the calculation
\begin{equation}\label{eqn:10.20}
\begin{aligned}
\tfrac{1}{16}h({\rm Tr}\,L,{\rm Tr}\,L)&=\tfrac{1}{2}k_0\eta(\lambda_1+(k_0-1)\mu)\\
&=\tfrac{1}{128}(m-1)\sqrt{\lambda_1^2-4\varepsilon}\Big((m+3)\lambda_1-(m-5)
\sqrt{\lambda_1^2-4\varepsilon}\,\Big)\\
&=:\rho^2
\end{aligned}
\end{equation}
for $\rho\geq0$. From \eqref{eqn:10.20} and that
$\lambda_1^2-4\varepsilon>0$, the following result is obvious.

\begin{lemma}\label{lm:10.2}
${\rm Tr}\,L=0$ if and only if $\lambda_1=\tfrac{m-5}{2\sqrt{m-1}}$
and $\varepsilon=-1$.
\end{lemma}

On the other hand, by similar proof of Lemma \ref{lm:8.2} and Lemma
\ref{lm:9.3}, we also obtain the following implicit characterization
of the statement ${\rm Tr}\,L=0$.

\begin{lemma}\label{lm:10.3}
${\rm Tr}\,L=0$ if and only if $n=\tfrac{1}{8}(m+1)(m+3)-1$.
\end{lemma}

Now, we are ready to complete the proof of Theorem \ref{th:10.1}.

\begin{proof}[Proof of Theorem \ref{th:10.1}]
We consider three cases:
\begin{enumerate}
\item[(i)] $n=\tfrac{1}{8}(m+1)(m+3)$.
\item[(ii)] $n>\tfrac{1}{8}(m+1)(m+3)$.
\item[(iii)] $n=\tfrac{1}{8}(m+1)(m+3)-1$.
\end{enumerate}

For Cases (i) and (ii), as ${\rm Tr}\,L\not=0$, we can define a unit
vector $t:=\frac1{4\rho}{\rm Tr}\,L$.

For Case (i), from previous discussions we see that
$$
\{t,\, w_j|_{1\leq j\leq k_0-1},\, w_{kl}|_{1\leq k<l\leq k_0}, \,
w_{kl}'|_{1\leq k<l\leq k_0},\, w_{kl}''|_{1\leq k<l\leq k_0}, \,
w_{kl}'''\mid_{1\leq k<l\leq k_0}\}
$$
forms an orthonormal basis of ${\rm Im}\,L=\mathcal{D}_3$.

\vskip 1mm

For Case (ii), as ${\rm Im}\,L\varsubsetneq\mathcal{D}_3$, we choose
$\{\tilde{w}_1,\ldots,\tilde{w}_{\tilde{n}}\}$ in
$\mathcal{D}_3\setminus{\rm Im}\,L$ such that
$$
\begin{aligned}
\{t,\, w_j\mid_{ 1\leq j\leq k_0-1},\, w_{kl}\mid_{1\leq k<l\leq
k_0},&\, w_{kl}'\mid_{1\leq k<l\leq
k_0},\, w_{kl}''\mid_{1\leq k<l\leq k_0},\\
&\, w_{kl}'''\mid_{1\leq k<l\leq k_0},\, \tilde{w}_r\mid_{ 1\leq
r\leq \tilde{n}}\}
\end{aligned}
$$
is an orthonormal basis of $\mathcal{D}_3$.

\vskip 1mm

For Case (iii), we see that
$$
\{w_j|_{ 1\leq j\leq k_0-1},\, w_{kl}|_{1\leq k<l\leq k_0}, \,
w_{kl}'|_{1\leq k<l\leq k_0},\, w_{kl}''|_{1\leq k<l\leq k_0}, \,
w_{kl}'''|_{1\leq k<l\leq k_0}\}
$$
is an orthonormal basis of ${\rm Im}\,L=\mathcal{D}_3$.

Now, following the proof of Theorem 7.1 in \cite{HLV2}, we can
proceed in the same way as in the proof of Theorem \ref{th:8.1} to
obtain the following conclusions:

If $n=\tfrac{1}{8}(m+1)(m+3)$, we can apply Theorem \ref{th:3.4} to
conclude that $M^n$ can be decomposed as the Calabi product of a
locally strongly convex centroaffine hypersurface with parallel
cubic form and a point.

If $n>\tfrac{1}{8}(m+1)(m+3)$, we can apply Theorem \ref{th:3.2} to
conclude that $M^n$ can be decomposed as the Calabi product of two
locally strongly convex centroaffine hypersurfaces with parallel
cubic form.

If $n=\tfrac{1}{8}(m+1)(m+3)-1$, then $M^n$ is centroaffinely
equivalent to the standard embedding
$\mathrm{SU}^*(\tfrac{m+3}{2})/\mathrm{Sp}(\tfrac{m+3}{4})
\hookrightarrow\mathbb{R}^{n+1}$.
\end{proof}

%%%%%%%%%%%%%%%%%%%%%%%%%%%%%%%%%%%%%%%%%%%%%%%%%%%%%%%%%%%%%%%
\numberwithin{equation}{section}

\section{Case $\{\mathfrak{C}_m\}_{2\le m\le n-1}$
with $k_0\geq 2$ and $\mathfrak{p}=7$}\label{sect:11}
In this section, we will prove the following theorem.

\begin{theorem}\label{th:11.1}
Let $M^n$ be a locally strongly convex centroaffine hypersurface in
$\mathbb{R}^{n+1}$ which has parallel and non-vanishing cubic form.
If $\mathfrak{C}_m$ with $2\le m\le n-1$ occurs and the integers
$k_0$ and $\mathfrak{p}$, as defined in subsection \ref{sect:4.5},
satisfy $k_0\geq 2$ and $\mathfrak{p}=7$, then $k_0=2$, $m=17$ and
$n\geq 26$. Moreover, we have either
\begin{enumerate}
\item[(i)] $n=27$, $M^n$ can be decomposed as the Calabi product of
a locally strongly convex centroaffine hypersurface with parallel
cubic form and a point, or
\item[(ii)] $n>27$, $M^n$ can
be decomposed the Calabi product of two locally strongly convex
centroaffine hypersurfaces with parallel cubic form, or
\item[(iii)] $n=26$,
$M^n$ is centroaffinely equivalent to the standard embedding\\
$\mathrm{E_{6(-26)}/F_4\hookrightarrow\mathbb{R}^{27}}$.
\end{enumerate}
\end{theorem}

\vskip 1mm

To prove Theorem \ref{th:11.1}, a key ingredient is the following
lemma whose proof is similar to that of Lemma 8.1 in \cite{HLV2}.
\begin{lemma}\label{lm:11.1}
If in the decomposition \eqref{eqn:4.28}, $k_0\ge2$ and
$\mathfrak{p}=7$, then we can choose an orthonormal basis
$\{x_j\}_{1\leq j\leq 7}$ for $V_{v_1}(0)$ and an orthonormal basis
$\{y_j\}_{1\leq j\leq 7}$ for $V_{v_2}(0)$ so that by identifying
$e_j(v_1)=x_j$ and $e_j(v_2)=y_j$, we have the relations
\begin{equation}\label{eqn:11.1}
L(e_j(v_1),e_l(v_2))=-L(v_1,e_je_l(v_2))=-L(e_le_j(v_1),v_2),
\end{equation}
for $1\leq j,l\leq 7$, where $e_je_l$ denotes a product defined by
the following multiplication table.
\begin{center}
 \begin{tabular}{cccccccc}
  $\cdot$ &$\ e_1$     &$\ e_2$     &$\ e_3$     &$\ e_4$     &$\ e_5$     &$\ e_6$     &$\ e_7$ \\
  $e_1$   &$-{\rm id}$ &$\ e_3$     &$-e_2$      &$\ e_5$     &$-e_4$      &$-e_7$      &$\ e_6$ \\
  $e_2$   &$-e_3$      &$-{\rm id}$ &$\ e_1$     &$\ e_6$     &$\ e_7$     &$-e_4$      &$-e_5$ \\
  $e_3$   &$\ e_2$     &$-e_1$      &$-{\rm id}$ &$\ e_7$     &$-e_6$      &$\ e_5$     &$-e_4$ \\
  $e_4$   &$-e_5$      &$-e_6$      &$-e_7$      &$-{\rm id}$ &$\ e_1$     &$\ e_2$     &$\ e_3$ \\
  $e_5$   &$\ e_4$     &$-e_7$      &$\ e_6$     &$-e_1$      &$-{\rm id}$ &$-e_3$      &$\ e_2$ \\
  $e_6$   &$\ e_7$     &$\ e_4$     &$-e_5$      &$-e_2$      &$\ e_3$     &$-{\rm id}$ &$-e_1$ \\
  $e_7$   &$-e_6$      &$\ e_5$     &$\ e_4$     &$-e_3$      &$-e_2$      &$\ e_1$     &$-{\rm id}$
\end{tabular}
\end{center}

\end{lemma}
\begin{proof}
As before we denote $V_j=\{v_j\}\oplus V_{v_j}(0),\ 1\le j\le k_0$.
First we fix any two orthogonal unit vectors $x_1,x_2\in
V_{v_1}(0)$. Then, by Lemmas \ref{lm:4.12} and \ref{lm:4.13}, we can
consecutively find unit vectors $y_1,y_2\in V_{v_2}(0)$ and $x_3\in
V_{x_2}(0)$, such that
\begin{equation}\label{eqn:11.2}
L(y_1,v_1)=-L(x_1,v_2),\ L(y_1,x_1)=L(v_1,v_2),
\end{equation}
\begin{equation}\label{eqn:11.3}
L(y_2,v_1)=-L(x_2,v_2),\ L(y_2,x_2)=L(v_1,v_2),
\end{equation}
\begin{equation}\label{eqn:11.4}
L(y_1,x_2)=-L(x_3,v_2),\ L(y_1,x_3)=L(x_2,v_2).
\end{equation}
From the computation
\begin{equation}\label{eqn:11.5}
\tau
h(x_3,v_1)=h(L(x_3,v_2),L(v_1,v_2))=h(-L(y_1,x_2),L(y_1,x_1))=0,
\end{equation}
we get $x_3\in V_{v_1}(0)$. Thus, we can further take unit vector
$y_3\in V_{v_2}(0)$ such that
\begin{equation}\label{eqn:11.6}
L(y_3,v_1)=-L(x_3,v_2),\ L(y_3,x_3)=L(v_1,v_2).
\end{equation}

\vskip 1mm

{\bf Claim 1}. $\{x_1,x_2,x_3,v_1\}$ are orthonormal vectors.
Similarly, $\{y_1,y_2,y_3,v_2\}$ are orthonormal vectors.

\vskip 1mm

In fact, by using \eqref{eqn:11.2} and \eqref{eqn:11.4}, we have
\begin{equation*}
\tau h(x_3,x_1)=h(L(x_3,v_2),L(x_1,v_2))=h(L(y_1,x_2),L(y_1,v_1))=0,
\end{equation*}
so we have $x_3\perp x_1$, and the mutual orthogonality of
$\{x_1,x_2,x_3,v_1\}$ immediately follows. The assertion that
$\{y_1,y_2,y_3\}$ are mutually orthogonal vectors can be proved
using Lemmas \ref{lm:4.12} and \ref{lm:4.13}. Hence we have the Claim 1.

\vskip 2mm

By \eqref{eqn:11.2}, \eqref{eqn:11.3} and \eqref{eqn:11.6}, we get the
relation
\begin{equation}\label{eqn:11.7}
L(y_1,x_1)=L(y_2,x_2)=L(y_3,x_3)=L(v_1,v_2),
\end{equation}
which together with Lemmas \ref{lm:4.12}, \ref{lm:4.13},
Claim 1 and \eqref{eqn:11.4}, imply that
\begin{equation}\label{eqn:11.8}
L(y_1,x_3)=-L(x_1,y_3)=L(x_2,v_2),
\end{equation}
\begin{equation}\label{eqn:11.9}
L(x_1,y_2)=-L(y_1,x_2)=L(x_3,v_2),
\end{equation}
\begin{equation}\label{eqn:11.10}
L(y_3,x_2)=-L(x_3,y_2)=-L(y_1,v_1).
\end{equation}

Now we pick an arbitrary unit vector $x_4\in V_{v_1}(0)$ such that
it is orthogonal to all $x_1,\ x_2$ and $x_3$. Then, inductively and
following the preceding argument, we can find unit vectors $y_4 \in
V_{v_2}(0)$, $x_5, x_6, x_7\in V_{v_1}(0)$ and $y_5, y_6, y_7 \in
V_{v_2}(0)$ such that the following relations hold:
\begin{align}\label{eqn:11.11}
\begin{split}
&L(x_4,y_1)=-L(x_1,y_4)=-L(x_5,v_2)=L(y_5,v_1),\\
& L(x_4,y_4)=L(x_1,y_1)=L(x_5,y_5)=L(v_1,v_2), \
L(x_4,v_2)=L(x_5,y_1),
\end{split}
\end{align}
\begin{align}\label{eqn:11.12}
\begin{split}
&L(x_4,y_2)=-L(x_6,v_2)=L(y_6,v_1),\\
&L(x_4,v_2)=L(x_6,y_2),\ L(x_6,y_6)=L(v_1,v_2),
\end{split}
\end{align}
\begin{align}\label{eqn:11.13}
\begin{split}
&L(x_4,y_3)=-L(x_7,v_2)=L(y_7,v_1),\\
&L(x_4,v_2)=L(x_7,y_3),\ L(x_7,y_7)=L(v_1,v_2).
\end{split}
\end{align}

Similar to Claim 1, applying Lemmas \ref{lm:4.12},
\ref{lm:4.13}, \eqref{eqn:11.2}\,--\,\eqref{eqn:11.4}
and \eqref{eqn:11.6}\,--\,\eqref{eqn:11.13}, we obtain:

\vskip 1mm

{\bf Claim 2}. $\{x_1,\ldots, x_7,v_1\}$ are orthonormal vectors.
Similarly, $\{y_1,\ldots, y_7,v_2\}$ are orthonormal vectors.

\vskip 1mm

From \eqref{eqn:11.7}, \eqref{eqn:11.11}\,--\,\eqref{eqn:11.13}, it
follows immediately that
\begin{equation}\label{eqn:11.15}
L(x_i,y_i)=L(v_1,v_2),\ \ i=1,\ldots,7,
\end{equation}
and therefore, by Lemmas \ref{lm:4.12} and \ref{lm:4.13}, we obtain
\begin{equation}\label{eqn:11.16}
L(x_i,y_j)=-L(y_i,x_j),\ L(x_i,v_2)=-L(y_i,v_1), \ 1\leq i\neq j\leq
7.
\end{equation}

Finally, based on the relations \eqref{eqn:11.2}-\eqref{eqn:11.4} and
\eqref{eqn:11.6}-\eqref{eqn:11.16}, the following relations can be
established (cf. proof of Lemma 8.1 in \cite{HLV2}):
\begin{align}\label{eqn:11.17}
\begin{split}
&L(x_4,y_5)=-L(v_1,y_1),\ L(x_4,y_6)=-L(v_1,y_2),\\
&L(x_4,y_7)=-L(v_1,y_3),
\end{split}
\end{align}
\begin{align}\label{eqn:11.18}
\begin{split}
&L(x_5,y_1)=-L(v_1,y_4),\ L(x_5,y_2)=L(v_1,y_7),\\
&L(x_5,y_3)=-L(v_1,y_6),\ L(x_5,y_6)=L(v_1,y_3),\\
&L(x_5,y_7)=-L(v_1,y_2),
\end{split}
\end{align}
\begin{align}\label{eqn:11.19}
\begin{split}
&L(x_6,y_1)=-L(v_1,y_7),\ L(x_6,y_2)=-L(v_1,y_4),\\
&L(x_6,y_3)=L(v_1,y_5),\ L(x_6,y_7)=L(v_1,y_1),
\end{split}
\end{align}
\begin{align}\label{eqn:11.20}
\begin{split}
&L(x_7,y_1)=L(v_1,y_6),\ L(x_7,y_2)=-L(v_1,y_5),\\
&L(x_7,y_3)=-L(v_1,y_4).
\end{split}
\end{align}

In a similar way as above, all relations in \eqref{eqn:11.1} can be
verified, and thus we complete the proof of Lemma \ref{lm:11.1}.
\end{proof}

Now, we can present the following crucial and remarkable lemma with
a simplified proof (comparing to that of Lemma 8.2 in \cite{HLV2})
included.

\begin{lemma}\label{lm:11.2}
Suppose that in the decomposition \eqref{eqn:4.28} we have
$k_0\geq2$ and $\mathfrak{p}=7$. Then it must be the case that
$k_0=2$.
\end{lemma}

\begin{proof}
Suppose on the contrary that $k_0\geq 3$. Following the same
argument as in the proof of Lemma \ref{lm:11.1} for $V_{v_1}(0)$ and
$V_{v_2}(0)$, we choose a basis
$\{x_1,x_2,\tilde{x}_3,x_4,\tilde{x}_5,\tilde{x}_6, \tilde{x}_7\}$
of $V_{v_1}(0)$ and a basis $\{z_1,z_2,z_3,z_4,z_5,z_6, z_7\}$ of
$V_{v_3}(0)$ such that all the following relations hold:
\begin{equation}\label{eqn:11.21}
L(e_j(v_1),e_l(v_3))=-L(v_1,e_je_l(v_3))=-L(e_le_j(v_1),v_3),\ 1\leq
j,l\leq 7.
\end{equation}

Now, we have two orthonormal bases of $V_{v_1}(0)$, i.e.
$\{x_1,x_2,\tilde{x}_3,x_4,\tilde{x}_5,\tilde{x}_6, \tilde{x}_7\}$
and $\{x_1,x_2,x_3,x_4,x_5,x_6, x_7\}$. We first show that
$\tilde{x}_i=x_i$ for $i=3,5,6,7$:

By \eqref{eqn:4.36} and \eqref{eqn:11.21}, we get
\begin{equation*}
\tau L(y_1,z_1)=K(L(y_1,x_2),L(x_2,z_1))
=K(-L(x_3,v_2),-L(x_3,v_3))=\tau L(v_2,v_3).
\end{equation*}
Thus, similarly, we can prove that
\begin{equation}\label{eqn:11.22}
L(y_1,z_1)=\cdots=L(y_7,z_7)=L(v_2,v_3).
\end{equation}

Since $\{x_1,x_2,\tilde{x}_3,x_4,\tilde{x}_5,\tilde{x}_6,
\tilde{x}_7\}$ and $\{x_1,x_2,x_3,x_4,x_5,x_6, x_7\}$ are two
orthonormal bases for $V_{v_1}(0)$, we may assume that
$x_3=b_3\tilde{x}_3+b_5\tilde{x}_5+b_6\tilde{x}_6+b_7 \tilde{x}_7$.
Then we have the following calculation
\begin{equation}\label{eqn:11.23}
\begin{aligned}
\tau L(y_2,z_2)&=K(L(v_1,y_2),L(v_1,z_2))=- K(L(x_3,y_1),L(v_1,z_2))\\
&=b_3K(L(\tilde{x}_3,y_1),L(\tilde{x}_3,z_1))+
b_5K(L(\tilde{x}_5,y_1),L(\tilde{x}_5,z_7))\\
&\ \ \ -b_6K(L(\tilde{x}_6,y_1),L(\tilde{x}_6,z_4))
-b_7K(L(\tilde{x}_7,y_1),L(\tilde{x}_7,z_5))\\
&=b_3\tau L(y_1,z_1)+b_5\tau L(y_1,z_7)
-b_6\tau L(y_1,z_4)-b_7\tau L(y_1,z_5).
\end{aligned}
\end{equation}

On the other hand, by \eqref{eqn:11.22} and that $L(y_1,z_1),
L(y_1,z_4), L(y_1,z_5)$ and $L(y_1,z_7)$ are mutually orthogonal,
\eqref{eqn:11.23} implies that $b_3=1, b_5=b_6=b_7=0$ and hence
$x_3=\tilde{x}_3$. Similarly, we can verify that $x_i=\tilde{x}_i$
for $i=5,6,7$.

\vskip 2mm

In order to complete the proof of Lemma \ref{lm:11.2}, we will first
use \eqref{eqn:11.1} and \eqref{eqn:11.21} to show that we have also
similar relations between $V_{v_2}(0)$ and $V_{v_3}(0)$, i.e.,
\begin{equation}\label{eqn:11.24}
L(e_j(v_2),e_l(v_3))=-L(v_2,e_je_l(v_3))=-L(e_le_j(v_2),v_3),\ \
1\leq j,l\leq 7.
\end{equation}

In fact, for $j=l$, by Lemma \ref{lm:4.14}, \eqref{eqn:11.1} and
\eqref{eqn:11.21}, we have
$$
\begin{aligned}
\tau L(e_j(v_2),e_j(v_3))&=K(L(e_j(v_2),e_k(v_1)),L(e_k(v_1),e_j(v_3))) \\
&=K(L(v_2,e_je_k(v_1)),L(e_je_k(v_1),v_3))=\tau L(v_2,v_3).
\end{aligned}
$$

For $j\neq l$, according to the multiplication table in Lemma
\ref{lm:11.1}, there exists a unique $k$ and $\epsilon=\pm1$ such
that $e_le_j=\epsilon e_k, e_je_k=\epsilon e_l, e_ke_l=\epsilon e_j
$. It follows, by applying \eqref{eqn:4.36}, \eqref{eqn:11.1} and
\eqref{eqn:11.21}, that
$$
\begin{aligned}
\tau L(e_j(v_2),e_l(v_3))&=K(L(e_j(v_2),v_1),L(v_1,e_l(v_3))) \\
&=K(L(-\epsilon e_le_k(v_2),v_1),L(v_1,e_l(v_3))) \\
&=\epsilon K(L(e_k(v_2),e_l(v_1)),-L(v_3,e_l(v_1))) \\
&=-\epsilon\tau L(e_k(v_2),v_3)=-\tau L(e_le_j(v_2),v_3)
\end{aligned}
$$
and that
$$
\begin{aligned}
\tau L(v_2,e_je_l(v_3))&=K(L(e_k(v_1),v_2),L(e_je_l(v_3),e_k(v_1))) \\
&=K(L(v_2,\epsilon e_le_j(v_1)),L(-\epsilon e_k(v_3),e_k(v_1))) \\
&= K(L(v_1,-\epsilon e_le_j(v_2)),L(-\epsilon v_3,v_1))=\tau
L(e_le_j(v_2),v_3).
\end{aligned}
$$
Thus, \eqref{eqn:11.24} holds indeed.

From \eqref{eqn:11.1}, \eqref{eqn:11.21}, \eqref{eqn:11.24} and Lemma
\ref{lm:4.14}, we have
\begin{equation}\label{eqn:11.25}
K(L(v_1,y_6)+L(x_1,y_7),L(x_2,v_3))=0.
\end{equation}

On the other hand, we have
$$
\begin{aligned}
&K(L(v_1,y_6),L(x_2,v_3))=K(L(v_1,y_6),-L(v_1,z_2))=-\tau L(z_2,y_6),\\
&K(L(x_1,y_7),L(x_2,v_3))=K(L(x_1,y_7),-L(x_1,z_3))=-\tau L(z_3,y_7).
\end{aligned}
$$
These, together with \eqref{eqn:11.25}, give that
\begin{equation}\label{eqn:11.26}
L(z_2,y_6)+L(z_3,y_7)=0.
\end{equation}
\eqref{eqn:11.24} implies that $L(z_2,y_6)=L(z_3,y_7)$, and by
\eqref{eqn:11.26} we get $L(z_2,y_6)=0$.

However, we also have the relation $h(L(z_2,y_6),L(z_2,y_6))=\tau$,
which gives the contradiction.

This completes the proof of Lemma \ref{lm:11.2}.
\end{proof}

\vskip 2mm

Now, we are ready to complete the proof of Theorem \ref{th:11.1}.

\begin{proof}[Proof of Theorem \ref{th:11.1}]~

First, Lemma \ref{lm:11.2} implies that $k_0=2$ and
$\dim\,(\mathcal{D}_2)=16$.

Let $\{v_1,v_2,x_j,y_j,\ 1\leq j\leq 7\}$ be the orthonormal basis
of $\mathcal{D}_2$ as constructed in Lemma \ref{lm:11.1} such that
all relations in \eqref{eqn:11.1} hold. Then we easily see that the
image of $L$ is spanned by
$$
\{L(v_1,v_1),\ L(v_1,v_2),\ L(v_2,v_2);\ L(v_1,y_j)\mid_{1\leq j\leq 7}\}.
$$

Define $L_1=L(v_1,v_1)-L(v_2,v_2)$, then we have
\begin{equation}\label{eqn:11.27}
h(L_1,L_1)=4\tau\neq0.
\end{equation}
We now easily see that there exist nine orthonormal vectors in ${\rm
Im}\,L\subset \mathcal{D}_3$:
$$
\begin{aligned}
w_0=\tfrac{1}{\sqrt{4\tau}}L_1,\
w_1=\tfrac{1}{\sqrt{\tau}}L(v_1,v_2),\
w_{j+1}:=\tfrac{1}{\sqrt{\tau}}L(v_1,y_j),\ 1\leq j\leq 7.
\end{aligned}
$$

Note that ${\rm Tr}L=8(L(v_1,v_1)+L(v_2,v_2))$ is orthogonal to
$\{w_0,w_1,w_{j+1}\mid_{1\leq j\leq 7}\}$, by using \eqref{eqn:4.4},
\eqref{eqn:4.8} and the fact $v_1\in V_{v_2}(\tau)$, we obtain
\begin{equation}\label{eqn:11.28}
\tfrac{1}{64}h({\rm Tr}\,L,{\rm Tr}\,L)=\eta(\lambda_1+\mu)=\tfrac{1}{4}
\sqrt{\lambda_1^2-4\varepsilon}\Big(3\lambda_1
-\sqrt{\lambda_1^2-4\varepsilon}\Big)=:\rho^2
\end{equation}
for $\rho\geq0$. Then we have the estimate of the dimension
\begin{equation}\label{eqn:11.29}
n=1+\dim\,(\mathcal{D}_2)+\dim\,(\mathcal{D}_3)\ge26.
\end{equation}
From \eqref{eqn:11.28} and the fact $\lambda_1^2-4\varepsilon>0$, we
have the following result.

\begin{lemma}\label{lm:11.3}
${\rm Tr}\,L=0$ if and only if $2\lambda_1^2=1$ and
$\varepsilon=-1$.
\end{lemma}

On the other hand, by similar proof of Lemma \ref{lm:8.2}, we also
obtain the following implicit characterization of the statement
${\rm Tr}\,L=0$.

\begin{lemma}\label{lm:11.4}
${\rm Tr}\,L=0$ if and only if $n=26$.
\end{lemma}

Then, if $n=27$ or $n\geq 28$, we can define a unit vector
$t=\tfrac1{8\rho}{\rm Tr}\,L$ so that we can construct an orthonormal
basis for $\mathcal{D}_3$ and $T_pM^n$, respectively, and we get the
similar expressions as in Lemmas \ref{lm:8.3}, \ref{lm:8.4} and
\ref{lm:8.5} which allows us to conclude that $M^n$ can be
decomposed as the Calabi product of a locally strongly convex
centroaffine hypersurface with parallel cubic form and a point, or
the Calabi product of two locally strongly convex centroaffine
hypersurfaces with parallel cubic form.

If $n=26$, by calculating the difference tensor $K$ with respect to
the preceding typical basis of $T_pM^n$ totally similar to previous
sections as in Sections \ref{sect:8}-\ref{sect:10}, we can also show
that ${\rm Tr}\,(K_X)=0$ for any $X\in T_pM^n$. Then, according to
Theorem 8.1 of \cite{HLV2}, we can finally conclude that $M^n$ is
locally centroaffnely equivalent to the standard embedding
$\mathrm{E_{6(-26)}/F_4\hookrightarrow\mathbb{R}^{27}}$ that was
introduced in \cite{BD} and also \cite{HLV2}.

In conclusion, we have completed the proof of Theorem \ref{th:11.1}.
\end{proof}

%%%%%%%%%%%%%%%%%%%%%%%%%%%%%%%%%%%%%%%%%%%%%%%%%%%%%%%%%%%%%%%%%%%%%%
\numberwithin{equation}{section}

\section{Completion of the proof of Theorem \ref{th:1.1}}\label{sect:12}

If $C=0$, according to subsection 7.1.1 of \cite{SSV}, and also
Lemma 2.1 of \cite{LLSSW}, we have (i).

For hypersurfaces with $C\not=0$, according to Lemma \ref{lm:4.1},
it is necessary and sufficient to consider the cases
$\{\mathfrak{C}_m\}_{1\le m\le n}$ as well as the exceptional case
$\mathfrak{B}$.

Firstly, by Theorems \ref{th:4.1} and \ref{th:4.2}, we have
settled the two cases, $\mathfrak{C}_1$ and $\mathfrak{C}_n$, from
which we have (ii).

Next, case $\mathfrak{B}$ is settled by Theorem \ref{th:5.1}, from
which we have (viii).

Then, being of independent meaning we have Theorem \ref{th:6.1}, by
which a complete classification is given for the lowest dimension
$n=2$. Theorem \ref{th:6.1} verifies the assertion of Theorem
\ref{th:1.1} explicitly for $n=2$.

The remaining cases, i.e. $\mathfrak{C}_m$ with $2\leq m\leq n-1$,
are completely settled by Proposition \ref{pr:4.2} and subsequent
five theorems, i.e. Theorems \ref{th:7.1}, \ref{th:8.1},
\ref{th:9.1}, \ref{th:10.1} and \ref{th:11.1}. In these cases, we have
(ii)-(vii).

From all of above discussions, we have completed the proof of
Theorem \ref{th:1.1}.

%%%%%%%%%%%%%%%%%%%%%%%%%%%%%%%%%%%%%%%%%%%%%%%%%%%%%%%%%%%%%%%%%%%

\vskip 5mm

\begin{flushleft}

Xiuxiu Cheng and Zejun Hu:\\
{\sc School of Mathematics and Statistics, Zhengzhou University,\\
Zhengzhou 450001, People's Republic of China.}\\
E-mails: chengxiuxiu1988@163.com; huzj@zzu.edu.cn.

\vskip 2mm

Marilena Moruz:\\
{\sc Universit\'e de Valenciennes et du Hainaut-Cambr\'esis, LAMAV,
ISTV 2 59300 Valenciennes, France.}\\
E-mail: marilena.moruz@gmail.com.

\end{flushleft}
\end{document}